\documentclass[11pt,a4paper]{article}
\usepackage[utf8]{inputenc}
\usepackage[margin=2cm]{geometry}

\usepackage{xcolor}
\definecolor{myblue}{rgb}{0.1 0.1 0.6}
\usepackage{hyperref}
\hypersetup{
   colorlinks=true,
   linkcolor=myblue,
   citecolor=myblue,
   urlcolor=myblue
}

\usepackage{amsmath,amssymb,amsthm}
\usepackage{url, enumitem}

\newtheorem{theorem}{Theorem}[section]
\newtheorem{proposition}[theorem]{Proposition}
\newtheorem{corollary}[theorem]{Corollary}
\newtheorem{lemma}[theorem]{Lemma}

\newtheorem{definition}[theorem]{Definition}
\theoremstyle{remark}
\newtheorem{remarkx}[theorem]{Remark}

\newenvironment{remark}
  {\pushQED{\qed}\remarkx}
  {\popQED\endremarkx}

\newcommand{\mc}{\mathcal}

\newcommand{\id}{\operatorname{id}}
\newcommand{\lin}{\operatorname{lin}}

\newcommand{\op}{{\operatorname{op}}}
\newcommand{\Tr}{{\operatorname{Tr}}}
\newcommand{\im}{{\operatorname{Im}}}
\newcommand{\vnten}{\bar\otimes}
\newcommand{\aut}{\operatorname{Aut}}
\newcommand{\qaut}{\operatorname{QAut}}
\newcommand{\sym}{\operatorname{Sym}}

\begin{document}
\title{Quantum graphs: different perspectives, homomorphisms and quantum automorphisms}
\author{Matthew Daws}
\maketitle

\begin{abstract}
We undertake a study of the notion of a quantum graph over arbitrary finite-dimensional $C^*$-algebras $B$ equipped with arbitrary faithful states.  Quantum graphs are realised principally as either certain operators on $L^2(B)$, the quantum adjacency matrices, or as certain operator bimodules over $B'$.  We present a simple, purely algebraic approach to proving equivalence between these settings, thus recovering existing results in the tracial state setting.  For non-tracial states, our approach naturally suggests a generalisation of the operator bimodule definition, which takes account of (some aspect of) the modular automorphism group of the state.  Furthermore, we show that each such ``non-tracial'' quantum graphs corresponds to a ``tracial'' quantum graph which satisfies an extra symmetry condition.  We study homomorphisms (or CP-morphisms) of quantum graphs arising from UCP maps, and the closely related examples of quantum graphs constructed from UCP maps.  We show that these constructions satisfy automatic bimodule properties.  We study quantum automorphisms of quantum graphs, give a definition of what it means for a compact quantum group to act on an operator bimodule, and prove an equivalence between this definition, and the usual notion defined using a quantum adjacency matrix.  We strive to give a relatively self-contained, elementary, account, in the hope this will be of use to other researchers in the field.
\end{abstract}

\section{Introduction}

The notion of a \emph{Quantum} (or sometimes \emph{non-commutative}) \emph{Graph} is a non-commutative
generalisation of a classical (finite) graph which has been defined and explored in a number of different
presentations.  The initial definition arose from the study of quantum communication channels,
\cite{dsw}, and treats them as operator systems. The series of papers \cite{kw, weaver1, weaver2} extend
this definition to arbitrary von Neumann algebras (instead of just $\mathbb M_n$), here being more motivated
by a general notion of a \emph{quantum relation}.   A graphical approach, using string diagrams, was given in 
\cite{mrv}, showing that a notion of the adjacency matrix of quantum graph bijects with the quantum relation
definition.  Here again motivation comes from quantum computing considerations.  Also here
links with Quantum Groups arose, which has been further explored in \cite{bcehpsw}.
Recent papers exploring quantum graphs are \cite{grom, matsuda}, while generalisations to the non-tracial
case have been given in \cite{bcehpsw, matsuda}, and a more minimal set of axioms explored in \cite{bevw, bhinw, cw}.
Papers making use of quantum graphs include \cite{btw, cds, cw, eifler, ekt, ganesan, stahlke} to list but a few.

This paper revisits the bijection between adjacency matrices and quantum relations, and develops a new,
elementary, algebraic approach, eschewing graphical methods.  We take the time to survey the area, 
presenting the different approaches, and then show that they are equivalent.  Our different approach has
the benefit of seamlessly allowing us to treat the non-tracial situation; we further can work with
arbitrary faithful states, not $\delta$-forms.  We present a large number of new results.
Once we have established our approach to equivalence, it becomes clear
that in fact quantum graphs over a non-tracial state correspond to ``tracial'' quantum graphs satisfying
an extra symmetry, see Theorem~\ref{thm:reduce_tracial_case}.
We give a definition of the quantum automorphism group of a quantum graph which operates at the quantum relation
level.  We explore how quantum channels give rise to quantum graphs, showing how an analysis of the Stinespring
dilation yields generalisations and improvements of results from \cite{dsw, weaver2}.

The reader is warned that the term ``quantum graph'' has (at least) two meanings, the other meaning is more established (for example, the introductory text \cite{bk}) and considers a (classical) graph where each edge is identified with an interval of the real line, and the whole structure is equipped with a differential operator, the adjacency structure of the edges giving boundary conditions.  To the author's knowledge, there is no non-trivial inter-relation between the two fields of study, and while one can imagine possible generalisations involving both concepts, this appears to be research which has yet to be done.

Let us now quickly introduce the principal objects we shall study, and then outline the organisation of the paper.

For us a \emph{graph} $G$ is a (finite) vertex set $V=V_G$ together with a set of edges
$E=E_G\subseteq V\times V$.  Thus, in full generality, we allow directed graphs, and loops at vertices
(but not multiple edges between vertices, and each loop, if present, is undirected).  In this formalism,
an \emph{undirected graph} is when $(x,y)\in E \implies (y,x)\in E$; the graph has \emph{no loops} when
$(x,x)\not\in E$ for all $x\in V$; and the graph has \emph{a loop at every vertex} when $(x,x)\in E$ for
all $x\in V$.  We will typically only be interested in the case of no loops or a loop at every vertex,
and there is an obvious bijection between these cases.  Thus having a loop at every vertex can be viewed 
as simply a convention, and not some profound condition.

Let $B$ be a finite-dimensional $C^*$-algebra equipped with a faithful state $\psi$.  We shall say
that we are in the \emph{classical} case when $B$ is commutative, so $B \cong C(V) \cong \ell^\infty(V)$
for some finite set $V$ and $\psi$ corresponds to the counting measure on $V$.
All the definitions of a quantum graph correspond to a genuine graph on vertex set $V$ in the classical
case.  Briefly, the principal definitions of a quantum graph are: 
\begin{itemize}
\item a self-adjoint operator $A$ on $L^2(B)$ satisfying certain conditions corresponding to, in the classical
case, being $\{0,1\}$-valued, being undirected, and having a loop at every vertex.  See \cite{mrv, grom}
with generalisations to non-tracial states in \cite{bcehpsw, bevw, matsuda}; see Section~\ref{sec:quan_adj_mat}.
\item when $B\subseteq\mc B(H)$ for some Hilbert space $H$, a unital, self-adjoint subspace $S$ of
$\mc B(H)$ which is a bimodule over $B'$, the commutant of $B$, meaning that for $a,b\in B', x\in S$ also
$axb\in S$.  This definition is essentially independent of the choice of $H$.
We call such $S$ an \emph{operator bimodule}.  For $B=\mathbb M_n$ this is \cite{duan, dsw}, and in general,
\cite{kw, weaver1, weaver2}; see Section~\ref{sec:op_bimod}.
\item A completely positive map $B\rightarrow B$ satisfying similar axioms to the first case, explored
at least in the case $B=\mathbb M_n$.  See \cite{cw}; see Section~\ref{sec:super_ops_defn}.
\end{itemize}
We remark that examples of the second case, when $B=\mc B(H)$, arise naturally in the theory of quantum
communication channels (this being the motivation of \cite{duan, dsw}).  We explore what happens for general
$B$ (beyond $B=\mathbb M_n$) in Section~\ref{sec:from_quan_channels}.

In the literature, it is common to restrict to the case when $\psi$ is a \emph{$\delta$-form}, compare
Remark~\ref{rem:quantum_space} and Definition~\ref{defn:delta_form}.
However, this condition will play no role for us, excepting perhaps meaning that our definition
of the empty quantum graph is slightly more complicated.  See also Remark~\ref{rem:no_splitting_delta_form}
which shows that seemingly we cannot just reduce to the $\delta$-form case, in general.

When the state $\psi$ is a trace (perhaps a tracial $\delta$-form) these axiomatisations are
well-known to be equivalent, see \cite[Section~VII]{mrv} and \cite[Section~1]{grom}, the proofs using
graphical techniques, and it is not clear what happens in the non-tracial case
(but see \cite{matsuda} for some recent progress in adapting graphical ideas to the more general situation).
In contrast, we present a simple, purely algebraic approach which allows us to give a quick, direct
proof of the equivalence of the first two definitions.  Our arguments are motivated by thinking about
the problem using operator-algebraic techniques (e.g. working with commutants and projections, and using
aspects of Tomita--Takesaki theory).  Indeed, it has already been recognised that a key intermediate step
between the first two axiomatisations is:

\begin{itemize}
\item An (orthogonal) projection $e\in B\otimes B^\op$ satisfying certain symmetry axioms.  Here
$B^\op$ is the algebra $B$ with the \emph{opposite} product.  To see why such a projection $e$ might arise,
see Theorem~\ref{thm:non_sa_op_bimods} for example.
\end{itemize}

For example, \cite{grom} takes this as the principal definition.
We also extend the ``completely positive'' approach suggested in \cite{cw} to arbitrary (finite-dimensional)
$C^*$-algebras.  The technical tool here is a version of the Choi matrix correspondence, and
we present a version which works for any $B$; while not technically hard, this might be of independent interest.

Our approach to these equivalences is Section~\ref{sec:equiv}.  
A major benefit of our approach is that, with relatively little work, it suggests an approach which works
for any faithful state $\psi$, tracial or not.  Here we take the first definition (of operators on $L^2(B)$)
as primary, and show how to adapt the 2nd (operator bimodule) definition so that we still obtain a
correspondence.  This is an new definition, and a major benefit of the operator bimodule approach
is that it is extremely easy to find non-trivial examples in this setting.
We first introduce our key bijection(s)
and begin to indicate how to pass from adjacency matrices to projections.  This can be viewed as an algebraic
version of the graphical methods, which are suitably ``twisted'' by the modular automorphism group, in the
non-tracial setting.  It is well-known that projections in a von Neumann algebra correspond to closed
subspaces invariant for the commutant, and we quickly explore in Section~\ref{sec:proj_to_ss} what happens with
(non-self-adjoint) idempotents in place of projections.  In Section~\ref{sec:main_equivs} we state and prove
our main results: Theorem~\ref{thm:qu_ad_mats_axiom_1} which works with just the ``idempotency axiom'',
Theorem~\ref{thm:qu_ad_mats_axioms_12} which deals with general adjacency matrices, and
Proposition~\ref{prop:qu_ad_mats_axiom_5} deals with the ``loop at every vertex'' axiom.  
Along the way we prove a number of auxiliary results: for example, the adjacency matrix commutes with the
modular operator and conjugation, Proposition~\ref{prop:A_comms_J_nabla}.

In Section~\ref{sec:std_egs} we see how standard examples present under these equivalences.
In Section~\ref{sec:from_quan_channels} we explore how quantum channels give rise to quantum graphs,
extending known constructions to the case of general $B$.  Here a careful study of Stinespring dilations
are key to our results.  In Section~\ref{sec:no_loops} we look at the condition of having a loop at
every vertex (or that of having no loops) and produce bijections between these situations; we also look
at graph complements.  Here our main task is to clarify some comments in the literature, and to deal with
the non-$\delta$-form case.  In Section~\ref{sec:small_examples} we quickly look at examples over
$B=\mathbb M_2$ and $\mathbb M_2\oplus\mathbb M_2$.

In Section~\ref{sec:hm} we survey a number of different notions of ``homomorphism'' which have appeared
in the literature, mostly concentrating upon \cite{stahlke, weaver2}.  By extending our careful study
of Stinespring dilations from Section~\ref{sec:from_quan_channels} we improve results of \cite{weaver2}
by showing that ``pushouts'' and ``pullbacks'' are automatically bimodules over the appropriate algebra
(this also avoids the need to use the techniques of \cite{weaver1}).  This area is less settled that the
main definitions, and in Section~\ref{sec:hm_further_ideas} we survey some other notions
of ``homomorphism'', ``subgraph'' and so forth.

We finally turn our attention to automorphisms.  In Section~\ref{sec:auts_quan_graphs} we quickly look at
automorphisms of quantum graphs: this is well-known in the adjacency matrix picture, but we also look
at operator bimodules, and prove an equivalence.  This also motivates the quantum case considered in
Section~\ref{sec:qauts_quan_graphs}.  Again, for quantum adjacency matrices, this is well-known.
However, we find some of the literature here a little hard to follow.  For example, \cite[Section~4]{grom}
is readable, but only looks at the Hopf $*$-algebra case, while we work with the $C^*$-algebra picture
of compact quantum groups.  We hope that this section will hence be helpful to the non-expert.
After some set-up and technical calculation, in Section~\ref{sec:qaut_op_bimod} we propose a definition
of a coaction on an operator bimodule,
and our main result, Theorem~\ref{thm:bimod_quan_action}, shows that we then obtain exactly the same notion
as that of a coaction on a quantum adjacency matrix.
Here our result that every ``generalised'' quantum graph corresponds to a ``tracial'' quantum graph,
see Remark~\ref{rem:gen_by_A_itself}, is vital in linking the two pictures of the quantum automorphism group.
As operator bimodules are almost trivial to give examples of, we hope that this definition of a quantum
automorphism group might aid, for example, in understanding which quantum automorphism groups can arise.

It is notable that in our list of recent papers ``using'' quantum graphs, almost all use the Operator
Bimodule approach, while papers looking at quantum automorphisms use the adjacency matrix approach.
It is hence useful to have a comprehensive ``bridge'' between these different perspectives.

Let us end this introduction by saying that almost everything we do is finite-dimensional (but clearly
motivated in many places by infinite-dimensional theory), an exception being results in Section~\ref{sec:hm}.
Of course, operator bimodules are an infinite-dimensional framework, \cite{weaver1}, but it seems unclear
what to expect of an ``infinite-dimensional'' adjacency matrix: many of the core ideas of Section~\ref{sec:equiv}
are finite-dimensional.

\subsection{Notational conventions}\label{sec:conventions}

We shall attempt to introduce terminology and technical background results, as needed, rather than collecting
them up-front.  However, let us fix some conventions here.  The inner-product on a Hilbert space $H$ will be
written $(\cdot|\cdot)$ and is assumed linear in the \emph{right} variable.  Write $\mc B(H),\mc B_0(H)$ for
the algebras of bounded, respectively compact, operators on $H$.  The identity on $H$ is denoted by $1_H=1$, while
the identity map on other spaces is denoted by $\id$.

For $\xi,\eta\in H$, denote by $\omega_{\xi,\eta}$ the functional $\mc B(H)\rightarrow\mathbb C$ given by
$x\mapsto (\xi|x(\eta))$.  Then $\omega_{\xi,\eta}$ is a normal functional, and the closed span of such functionals gives the space of trace-class operators, which we denote by $\mc B(H)_*$.  Denote by $\theta_{\xi,\eta}$ the rank-one operator $H\rightarrow H; \zeta\mapsto
(\xi|\zeta)\eta$.  Then $\theta_{\xi_1,\eta_1} \circ\theta_{\xi_2,\eta_2} = (\xi_1|\eta_2) \theta_{\xi_2,\eta_1}$
which is a little awkward, but it means we are consistent in our convention that sesquilinear maps are linear on
the right.
We denote by $\Tr$ the non-normalised trace on $\mc B(H)$, so $\Tr(\theta_{\xi,\eta}) = (\xi|\eta)$, and
if $H\cong\mathbb C^n$ then $\Tr(1_H)=n$.

Let $\overline H$ be the \emph{conjugate Hilbert space} to $H$, so vectors in $\overline H$ are $\{ \overline\xi:
\xi\in H\}$ with the same addition, but with the complex-conjugate scalar multiplication $\lambda \overline\xi
= \overline{ \overline\lambda \xi }$.  For $x\in\mc B(H)$ we define $x^\top, \overline x \in \mc B(\overline H)$
by
\[ x^\top(\overline\xi) = \overline{ x^*\xi }, \quad
\overline x (\overline\xi) = \overline{ x\xi } \qquad (\overline\xi\in\overline H). \]
Thus $(x^\top)^* = (x^*)^\top = \overline x$.  The map $x\mapsto x^\top$ is a linear anti-homomorphism $\mc B(H)
\rightarrow\mc B(\overline H)$.  If $(e_i)_{i=1}^n$ is an orthonormal basis for $H$ then $(\overline e_i)_{i=1}^n$
is an orthonormal basis for $\overline H$, so with respect to these bases we have that $\mc B(H) \cong \mathbb M_n$
and $\mc B(\overline H) \cong\mathbb M_n$, and $x\mapsto x^\top$ is simply taking the transpose of a matrix.

Throughout the paper, $B$ will denote a finite-dimensional $C^*$-algebra, and $\psi$ will be a state on $B$,
always assumed faithful, with further conditions imposed if necessary.

\subsection{Acknowledgements}

The author thanks Simon Schmidt and Christian Voigt for helpful conversations at an early stage of this
project, and thanks Larissa Kroell and Monica Abu Omar for useful correspondence, along with the anonymous referees for helpful comments.

The author is partially supported by EPSRC grant EP/T030992/1.
For the purpose of open access, the author has applied a CC BY public copyright licence to any Author Accepted Manuscript version arising.

No data were created or analysed in this study.

\section{Quantum adjacency matrices}\label{sec:quan_adj_mat}

We introduce the notion of a quantum adjacency matrix.  We first briefly study finite-dimensional
$C^*$-algebras.  Any finite-dimensional $C^*$-algebra $B$ has the form
\[ B \cong \bigoplus_{k=1}^n \mathbb M_{n_k}. \]
The addition and multiplication in $B$ respects this direct sum decomposition, and so often it suffices
to consider just a matrix algebra $\mathbb M_n$.  We write $e_{ij} \in \mathbb M_n$ for the matrix units, so
a typical matrix $a = (a_{ij})$ is $a = \sum_{i,j} a_{ij} e_{ij}$, and the trace satisfies $\Tr(e_{ij}) =
\delta_{i,j}$ the Kronecker delta.  We write $a^\top$ for the transpose of $a$ (compare with Section~\ref{sec:conventions}).

For $\psi\in \mathbb M_n^*$ there is $Q\in\mathbb M_n$ with $\psi(a) = \Tr(Qa)$ for each $a\in\mathbb M_n$;
indeed, set $Q_{ij} = \psi(e_{ji})$.
Then $\psi$ being positive corresponds to $Q$ being positive.  If $\psi$ is positive, then $\psi$ being faithful
corresponds to $Q$ being invertible, and $\psi$ is a state exactly when $\Tr(Q)=1$.  The same remarks hence apply
to $B$, choosing $Q$ in each matrix factor.

\begin{definition}
Given $B,\psi$ as throughout the paper, we shall denote by $Q\in B$ the unique positive invertible element with
$\psi(a) = \Tr(Qa)$ for each $a\in B$.  Let $L^2(B) = L^2(B,\psi)$ denote the GNS space for $\psi$.
\end{definition}

Thus $L^2(B)$ is a finite-dimensional Hilbert space.  Let $\Lambda:B\rightarrow L^2(B)$ be the GNS map, which
in this situation is a linear isomorphism.  If necessary, we write $\pi:B\rightarrow\mc B(L^2(B))$ for the GNS
representation, but we shall often suppress this.  Hence
\[ \pi(a)\Lambda(b) = a\Lambda(b) = \Lambda(ab), \quad \Lambda(a) = a\Lambda(1), \quad
(\Lambda(a)|\Lambda(b)) = \psi(a^*b) = \Tr(Q a^* b)
\qquad (a,b\in B). \]

Thus also $\Lambda\otimes\Lambda$ is a linear isomorphism between $B\otimes B$ and $L^2(B)\otimes L^2(B)$.
The multiplication map $m:B\otimes B\rightarrow B$ may hence be regarded as a (bounded) operator
$L^2(B)\otimes L^2(B) \rightarrow L^2(B)$.  The identity $1\in B$ induces a linear map $\mathbb C \rightarrow B;
\lambda \mapsto \lambda 1$.

\begin{definition}
Let $m:L^2(B)\otimes L^2(B) \rightarrow L^2(B)$ be the multiplication map, and $m^*:L^2(B) \rightarrow
L^2(B)\otimes L^2(B)$ its Hilbert space adjoint.  Let $\eta:\mathbb C\rightarrow L^2(B)$ be induced by the identity
of $B$, and let $\eta^*: L^2(B)\rightarrow\mathbb C$ be its Hilbert space adjoint.
\end{definition}

Hence $\eta:\lambda \mapsto \lambda\Lambda(1)$ and $\eta^*(\xi) = (\Lambda(1)|\xi)$ for $\xi\in L^2(B)$, so
$\eta^*(\Lambda(a)) = \psi(a)$ for $a\in B$.  Similarly, $m(\Lambda(a)\otimes\Lambda(b)) = \Lambda(ab)$.

\begin{remark}\label{rem:quantum_space}
In the literature (for example, \cite{bcehpsw, bevw, grom, matsuda, mrv})
it is very common to assume that additionally $\psi$ is a \emph{$\delta$-form}
(see Definition~\ref{defn:delta_form})
and in this case, we use the terminology that $(B,\psi)$ is a \emph{Finite Quantum Space}.
However, being a $\delta$-form will actually play no role in our arguments, and so it seems more sensible to
not add this as a seemingly unnecessary hypothesis.  To avoid confusion with the existing literature, we hence
avoid the ``Quantum Space'' terminology.
\end{remark}

We can now proceed with the definition of a quantum adjacency matrix.

\begin{definition}\label{defn:adj_op}
A \emph{quantum adjacency matrix} is a self-adjoint operator $A\in\mc B(L^2(B))$ satisfying:
\begin{enumerate}[series=qam_defn]
\item\label{defn:quan_adj_mat:idem} $m(A\otimes A)m^* = A$;
\item\label{defn:quan_adj_mat:undir} $(1\otimes\eta^* m)(1\otimes A\otimes 1)(m^*\eta\otimes 1) = A$;
\item\label{defn:quan_adj_mat:reflexive} $m(A\otimes 1)m^*=1$.
\end{enumerate}
\end{definition}

By symmetry we also define variations of these axioms:
\begin{enumerate}[resume*=qam_defn]
\item\label{defn:quan_adj_mat:undir:other} $(\eta^* m\otimes 1)(1\otimes A\otimes 1)(1\otimes m^*\eta) = A$;
\item\label{defn:quan_adj_mat:reflexive:other} $m(1\otimes A)m^*=1$.
\end{enumerate}
Notice that, by taking the adjoint, axiom (\ref{defn:quan_adj_mat:undir}) holds for $A$ if and only if axiom (\ref{defn:quan_adj_mat:undir:other}) holds for $A^*$.  Hence, if $A$ is self-adjoint, then axioms (\ref{defn:quan_adj_mat:undir}) and (\ref{defn:quan_adj_mat:undir:other}) are equivalent; in fact, these axioms are equivalent for any $A$, see Proposition~\ref{prop:undir_axioms_same} below.
We will show later that, in the presence of the other axioms,
also (\ref{defn:quan_adj_mat:reflexive}) and (\ref{defn:quan_adj_mat:reflexive:other}) are equivalent, see
Proposition~\ref{prop:qu_ad_mats_axiom_3}.  To our knowledge, the only discussion of these ``variations'' is
in \cite[Section~2.2]{matsuda}; in particular, here (\ref{defn:quan_adj_mat:undir:other}) is termed $A$ being
``self-transpose''.

While the terminology is to call $A$ a ``matrix'', it is of course merely a linear map, as $L^2(B)$ usually
carries no canonical choice of basis, and in our approach, we shall not think of $A$ as an actual matrix,
except in the classical situation which motivates the more general theory.

There is some further discussion of variations of these axioms in \cite[Definition~2.19]{matsuda}, and
we single out the following for later study.

\begin{definition}\label{defn:adj_op:real}
Consider an operator $A\in\mc B(L^2(B))$.  Let
$A_0:B\rightarrow B$ be the linear map induced by $A$, if we regard $B$ and $L^2(B)$ as being linearly
isomorphic.  We say that $A$ is \emph{real} when $A_0(a^*) = A_0(a)^*$ for each $a\in B$.
\end{definition}

Let us see what happens in the classical situation, where $B=\ell^\infty(V) = C(V)$ for a finite set
$V$, and $\psi$ is given by the uniform measure on $V$.
Let $(\delta_v)_{v\in V}$ be the obvious basis of $B$ consisting of idempotents, so $\psi(\delta_v)=1$ for each
$v$, and $m(\Lambda(\delta_v) \otimes\Lambda(\delta_u)) = \delta_{v,u} \Lambda(\delta_v)$
and hence $m^*\Lambda(\delta_v) = \Lambda(\delta_v) \otimes \Lambda(\delta_v)$.
Also $\eta$ is given by the vector $\Lambda(1) = \sum_v \Lambda(\delta_v)$, and
$\eta^*\Lambda(\delta_v)=1$ for each $v\in V$.  Finally, we see that $L^2(B) \cong \mathbb C^V$ with orthonormal
basis $(e_v)_{v\in V}$ where $e_v=\Lambda(\delta_v)$.  Hence $A$ may be identified with a matrix $(A_{u,v})$.
Indeed, given two operators $A_1,A_2\in \mc B(L^2(B))$,
\[ m(A_1\otimes A_2)m^*(e_v) = m(A_1(e_v) \otimes A_2(e_v))
= \sum_{u,w} A^{(1)}_{u,v} A^{(2)}_{w,v} m(e_u\otimes e_w)
= \sum_u A^{(1)}_{u,v} A^{(2)}_{u,v} e_u, \]
and so $m(A_1\otimes A_2)m^*$ has matrix $(A^{(1)}_{u,v} A^{(2)}_{u,v})$.  Thus $m(A_1\otimes A_2)m^*$ is the
entry-wise, or \emph{Schur}, or \emph{Hadamard}, product of the matrices $A_1, A_2$.  Thus axiom
(\ref{defn:quan_adj_mat:idem}) corresponds to $A$ being Schur-idempotent, equivalently, that $A$ is
$0,1$-valued.  As the adjoint on $B=C(V)$ is the pointwise complex conjugation, $A$ being real,
Definition~\ref{defn:adj_op:real}, corresponds to the matrix of $A$ being real-valued, which is of course
automatic if $A$ is Schur-idempotent.

Similarly, we calculate that
\begin{align*}
(1\otimes\eta^* m) & (1\otimes A\otimes 1)(m^*\eta\otimes 1)(e_v)
= \sum_u (1\otimes\eta^* m)(1\otimes A\otimes 1)(e_u\otimes e_u\otimes e_v) \\
&= \sum_{u,w} (1\otimes\eta^* m)(e_u \otimes A_{w,u} e_w \otimes e_v)
= \sum_u e_u \otimes \eta^*(A_{v,u} e_v)
= \sum_u A_{v,u} e_u
= A^\top(e_v)
\end{align*}
Thus axiom (\ref{defn:quan_adj_mat:undir}) is equivalent to $A^\top = A$.  In this situation, axiom 
(\ref{defn:quan_adj_mat:undir:other}) gives the same condition.  Finally,
\begin{align*}
m(A\otimes 1)m^*(e_v) &= m(A(e_v) \otimes e_v) = A_{v,v} e_v
\end{align*}
and so axiom (\ref{defn:quan_adj_mat:reflexive}) (and here also (\ref{defn:quan_adj_mat:reflexive:other}))
is that $A$ has $1$ down the main diagonal.

\begin{proposition}
Let $V$ be a finite set, and let $B=C(V)$.  An operator $A\in\mc B(L^2(B))$ which satisfies axiom
(\ref{defn:quan_adj_mat:idem}) corresponds to a directed graph structure on $V$.  That $A$ is self-adjoint,
equivalently, that $A$ satisfies axiom (\ref{defn:quan_adj_mat:undir}) or (\ref{defn:quan_adj_mat:undir:other}),
corresponds to the graph being undirected.  That $A$ satisfies axiom (\ref{defn:quan_adj_mat:reflexive})
or (\ref{defn:quan_adj_mat:reflexive:other}) corresponds to the graph having a loop at every vertex.
\end{proposition}
\begin{proof}
A $0,1$-valued matrix corresponds to a directed graph structure on $V$ by setting the edge set to
be $E = \{ (u,v) : A_{u,v} = 1 \}$.  The results now follows from the discussion above.
\end{proof}

\begin{remark}
We already see that there is a certain redundancy between requiring that $A$ be \emph{self-adjoint} and
that $A$ satisfy axiom (\ref{defn:quan_adj_mat:undir}) (and/or (\ref{defn:quan_adj_mat:undir:other})).
This is discussed a little more in \cite{matsuda} after Definition~2.19, in particular,
\cite[Lemma~2.22]{matsuda}.  To an extent, this issue occurs even for
general, non-commutative $B$, compare Remark~\ref{rem:needs_axioms_joined}.
\end{remark}

If we prefer to have a loop at no vertex, then given the above discussion, it is natural to consider the
axioms:
\begin{enumerate}[resume*=qam_defn]
\item\label{defn:quan_adj_mat:irrefl} $m(A\otimes 1)m^*=0$;
\item\label{defn:quan_adj_mat:irrefl:other} $m(1\otimes A)m^*=0$.
\end{enumerate}

Again, we show later that these axioms are equivalent, in the presence of the other axioms.

We give some simple examples for quantum adjacency matrices.

\begin{definition}\label{defn:complete_qg}
The \emph{complete quantum graph} is given by the adjacency matrix $A = \theta_{\Lambda(1), \Lambda(1)}$, so that
$A\Lambda(a) = \eta^*(\Lambda(a)) \eta(1) = \psi(a) \Lambda(1)$ for each $a\in B$.
\end{definition}

A direct computation shows that this $A$ indeed satisfies all the axioms
(\ref{defn:quan_adj_mat:idem}) through (\ref{defn:quan_adj_mat:reflexive:other}).

We study the map $mm^*$ a little further.  Suppose for the moment that $B=\mathbb M_n$.  Given $i,j$
set $m^*\Lambda(e_{ij}Q^{-1}) = \sum_t \Lambda(a_t) \otimes \Lambda(b_t)$, so for $x,y\in B$,
\begin{align*}
\sum_t \Tr(Qx^*a_t) \Tr(Qy^*b_t)
&= \sum_t (\Lambda(x)|\Lambda(a_t)) (\Lambda(y)|\Lambda(b_t))
= (\Lambda(x)\otimes \Lambda(y)|m^*\Lambda(e_{ij}Q^{-1})) \\
&= \Tr(Q(xy)^*e_{ij}Q^{-1})
= \Tr(y^* x^* e_{ij})
= (y^*x^*)_{ji} = \sum_k (y^*)_{jk} (x^*)_{ki} \\
&= \sum_k \Tr(e_{kj} y^*) \Tr(e_{ik} x^*)
= \sum_k \Tr(Q x^* e_{ik} Q^{-1}) \Tr(Q y^* e_{kj} Q^{-1}).
\end{align*}
As $x,y$ are arbitrary, this shows that
\[ m^*\Lambda(e_{ij}Q^{-1}) = \sum_k \Lambda(e_{ik} Q^{-1}) \otimes \Lambda(e_{kj} Q^{-1}). \]
Thus also
\[ m m^*\Lambda(e_{ij}Q^{-1}) = \sum_k \Lambda(e_{ik} Q^{-1} e_{kj} Q^{-1})
=\sum_k \Lambda(Q^{-1}_{kk} e_{ij} Q^{-1})
\quad\text{and so}\quad
mm^* = \delta^2 1 \]
where $\delta^2 = \Tr(Q^{-1})$.  As $Q^{-1}$ is positive, $\delta^2$ is indeed positive.
In the general case, $B$ is a direct
sum of matrix factors, and $Q = (Q_k)$ say.  Then $mm^*$ respects the direct sum decomposition, and on each
matrix factor is the identity times $\Tr(Q_k^{-1})$, say $mm^* = (\Tr(Q_k^{-1}) 1_k)$.

As the following appears extensively in the existing literature, we make a formal definition.

\begin{definition}\label{defn:delta_form}
When $mm^* = \delta^2 1$, that is, $\Tr(Q_k^{-1})=\delta^2$ for each $k$,
we say that $\psi$ is a \emph{$\delta$-form}.
\end{definition}

We now define the empty graph.

\begin{definition}\label{defn:empty_qg}
The \emph{empty quantum graph} is given by the adjacency matrix $A = (\Tr(Q_k^{-1})^{-1} 1_k) = (mm^*)^{-1}$.
\end{definition}

A moment's thought reveals that if $A\in\mc B(L^2(B))$ respects the direct sum decomposition of $L^2(B)$ which
is induced by that of $B$, then all of the axioms also respect the decomposition: that is, it suffices to check
the axioms on each matrix block.  With this observation, a direct calculation shows that the empty quantum graph
adjacency matrix satisfies all the axioms
(\ref{defn:quan_adj_mat:idem}) through (\ref{defn:quan_adj_mat:reflexive:other}); in this calculation it is
helpful to show that $(1\otimes\eta^* m)(m^*\eta\otimes 1)=1$.

The notion of a quantum adjacency matrix was first considered in \cite{mrv}, in the special case that $\psi$
is a trace and a $\delta$-form.  The paper \cite{mrv} makes extensive use of graphical methods, and the
axioms are stated purely in terms of graphical methods.  These methods seem to make extensive use of the trace
condition, but \cite[Section~2]{matsuda} shows how these ideas can be adapted to the non-tracial (but
still $\delta$-form) setting.  When $\psi$ is merely a $\delta$-form, the definition we have given is
equivalent to that given in \cite[Definition~3.4]{bcehpsw}, except that we have chosen to normalise $A$
differently.  Again when $\psi$ is a $\delta$-form, \cite[Definition~3.3]{bevw} just considers axiom
(\ref{defn:quan_adj_mat:idem}) (and without the requirement that $A$ be self-adjoint).
A different starting point is chosen in \cite[Definition~1.8]{grom} which corresponds to the projections
$e \in B\otimes B^\op$ which we consider in Section~\ref{sec:equiv} below, again when $\psi$ is a tracial
$\delta$-form, and again graphical methods are used to make links to adjacency matrices.

\section{Adjacency super-operators}\label{sec:super_ops_defn}

We slightly abuse notation and in this section consider $m$ and $m^*$ as maps $B\otimes B\rightarrow B$
and $B\rightarrow B\otimes B$ respectively.

\begin{definition}\label{defn:adj_superop}
A \emph{quantum adjacency super-operator} is a completely positive
map $A:B\rightarrow B$ with $m(A\otimes A)m^*=A$.
\end{definition}

This terminology is our own; we chose ``super-operator'' from its use in Quantum Information Theory to refer
to a linear map between $C^*$-algebras rather than Hilbert spaces, as is the case here.  Notice that the functional
$\psi$ enters the definition in the formation of the map $m^*$.

We have stated this definition with the minimal axiom, $m(A\otimes A)m^*=A$, corresponding to axiom
(\ref{defn:quan_adj_mat:idem}) in Definition~\ref{defn:adj_op}.  In the obvious way, any other axiom can be
added.

This definition is made in \cite[Definition~1.8]{cw}, in the case when $B = \mathbb M_n$ is a single matrix
block and $\psi$ is a trace.  In \cite[Section~1.2]{cw} it is briefly argued that for this $B$, this
definition agrees with the other notions of a quantum graph.  At the level of generality stated here, the
definition is considered in \cite[Proposition~2.23]{matsuda}.  We shall give a different approach to
\cite[Proposition~2.23]{matsuda} below, using ideas much closer to the approach of \cite{cw}.
Indeed, we shall see that the super-operator $A$ in Definition~\ref{defn:adj_superop} agrees with the operator
$A$ in Definition~\ref{defn:adj_op}, if we identify $B$ and $L^2(B)$ (see Theorem~\ref{thm:super_op_to_op_new}
for a precise statement).

In the non-tracial situation, \cite[Section~2]{bhinw} makes further study of the quantum adjacency super-operators.
We make links between this and our approach in Remark~\ref{rem:cp_case_in_general}.

\section{Operator bimodules}\label{sec:op_bimod}

The following definitions make perfect sense in the infinite-dimensional setting, and so we shall briefly work
in more generality.  We follow \cite[Section~1]{kw} and \cite{weaver1}.

\begin{definition}\label{defn:quan_rel}
Let $M\subseteq\mc B(H)$ be a von Neumann algebra.  A \emph{quantum relation} on $M$ is a
$W^*$-bimodule over $M'$, that is, a weak$^*$-closed subspace $S\subseteq \mc B(H)$ with
$M'SM' = \{ xsy : x,y\in M', s\in S \} \subseteq S$.
\end{definition}

We also call such an $S$ an \emph{operator bimodule}.  As stated, this definition depends on $M\subseteq\mc B(H)$
and not just on $M$.  However, the definition is essentially independent of the choice of $H$, compare
\cite[Theorem~2.7]{weaver1}.  This will also follow from Theorem~\ref{thm:non_sa_op_bimods}.

\begin{definition}
A quantum relation $S$ on $M\subseteq\mc B(H)$ is \emph{reflexive} if $M'\subseteq S$, and is \emph{symmetric} if
$S^* = \{ x^* : x\in S \}$ equals $S$.
\end{definition}

Notice that as $M'SM'\subseteq S$, we have that $S$ is reflexive if and only if $1\in S$.  Recall that an \emph{operator system} is a unital, self-adjoint subspace of $\mc B(H)$, for some Hilbert space $H$.  \cite{kw, weaver1} define a \emph{quantum graph} on $M\subseteq\mc B(H)$ to be a reflexive, symmetric quantum relation on $M$.  That is, a weak$^*$-closed operator system which is an operator bimodule over $M'$.

We shall now specialise to the finite-dimensional setting, and also introduce some terminology of our own, separating out the conditions.

\begin{definition}\label{defn:quan_graph}
Given a finite-dimensional $C^*$-algebra $B$, a \emph{quantum graph} on $B$ is given by an embedding
$B\subseteq\mc B(H)$ for some finite-dimensional $H$, and a subspace $S\subseteq\mc B(H)$ which is an
operator bimodule over $B'$.  We define further that $S$ is:
\begin{enumerate}
\item\label{defn:quan_graph:undir} \emph{undirected} when $S$ is self-adjoint;
\item\label{defn:quan_graph:} \emph{has all loops} when $1\in S$.
\end{enumerate}
\end{definition}

There is no dependence on our functional $\psi$.
Notice that it is rather easy to give examples: for example, when
$B=\mathbb M_n$ we can choose $H=\mathbb{C}^n$, and so $B'=\mathbb C$.  Hence a quantum graph is simply a
(self-adjoint, unital) subspace of $\mathbb M_n$.
There are two extremal choices if we assume undirected and that we have all loops.
We shall see later in Section~\ref{sec:std_egs} that this definition agrees with
Definitions~\ref{defn:complete_qg} and~\ref{defn:empty_qg}.

\begin{definition}\label{defn:qg:basic_egs}
Given $B$ and any $B\subseteq\mc B(H)$, the \emph{empty quantum graph} is $S = B'$.
The \emph{complete quantum graph} is $S = \mc B(H)$.
\end{definition}

In the classical situation, when $B = C(V)$ for a finite set $V$, we choose $H=\ell^2(V)$ so $B$ has basis
$(\delta_v)$ with $\delta_v \delta_u = \delta_{v,u} \delta_v$, and $H$ has orthonormal basis $(e_v)$ with
$\delta_v(e_u) = \delta_{v,u} e_u$.  We hence identify $\mc B(H)$ with matrices over $V$, and we see that
$B' = B$.  Thus, given an operator bimodule $S$ over $B'$, if $x=(x_{u,v})\in S$ also $\delta_v x \delta_u 
= x_{v,u} e_{v,u} \in S$.  So $S$ is the span of the matrix units it contains.  We hence obtain a bijection
between such $S$ and arbitrary subsets $E\subseteq V\times V$.

\begin{proposition}
With $B = C(V)$, quantum graphs over $B$ biject with graphs on vertex set $V$.  Furthermore, the quantum graph
is undirected if and only if the graph is, and the quantum graph has all loops if and only if the graph has a
loop at every vertex.
\end{proposition}
\begin{proof}
We have already established a bijection between operator bimodules $S$ over $B'$ and subsets $E\subseteq
V\times V$, that is, graphs on vertex set $V$.  Then $S$ is undirected when it is self-adjoint, which clearly
corresponds to the condition that $(u,v)\in E \implies (v,u)\in E$.  Further, $S$ has all loops when $1\in S$,
so that $e_{v,v}\in S$ for each $v$, so that $(v,v)\in E$ for all $v$.
\end{proof}

This result also holds for infinite $V$, compare \cite[Proposition~2.5]{weaver1}.

This notion of a quantum graph arose both from Weaver's work in \cite{weaver1}, see also \cite{weaver2},
but also from Quantum Information Theory considerations, \cite{dsw}, also \cite{duan, stahlke}.

\section{Equivalence of definitions}\label{sec:equiv}

We present a unified way to move between quantum adjacency matrices and the operator bimodule notion of a quantum
graph in the tracial case, and also show how to obtain a correspondence in the general setting, by adjusting
the definition of an operator bimodule to take account of (some aspect of) the functional $\psi$.  Once we have
these results, we obtain a further link with the usual definition of an operator bimodule, 
Section~\ref{sec:back_to_qgs}.  We then also look at super-operators.
We make some remarks about the existing literature.

We start by looking at a ``toy version'' of Tomita--Takesaki theory, which in our finite-dimensional situation
is easy: all of the following claimed properties may be verified by direct computation.
Recall the GNS construction $(L^2(B), \Lambda)$, and our invertible positive operator $Q\in B$.  We use
functional calculus to define $Q^z$ for all $z\in\mathbb C$.  We define $\sigma_z(a) = Q^{iz} a Q^{-iz}$ for
$a\in B$.  This is the (analytic extension of the) modular automorphism group for $\psi$.  Indeed, for
$a,b,c\in B$,
\begin{align*}
(\Lambda(a)|\Lambda(bc)) &= \psi(a^*bc) = \Tr(Qa^*bc) = \Tr(cQa^*b) = \Tr(Q Q^{-1} cQa^*b)
= (\Lambda(a Q c^* Q^{-1})|\Lambda(b)) \\
&= (\Lambda(a \sigma_{-i}(c^*)) | \Lambda(b)).
\end{align*}
Notice that $\sigma_z(a)^* = \sigma_{\overline z}(a^*)$ and $\sigma_w(\sigma_z(a)) = \sigma_{w+z}(a)$ for
each $a,z,w$. The \emph{modular conjugation} is the anti-linear operator $J:L^2(B)\rightarrow L^2(B)$ defined by
\[ J\Lambda(a) = \Lambda(\sigma_{i/2}(a)^*) = \Lambda(\sigma_{-i/2}(a^*)) \qquad (a\in B). \]
One can easily check that $(J\xi|J\eta) = (\eta|\xi)$ for each $\xi,\eta\in L^2(B)$ and that $J^2=1$, so that
$J$ is an isometry.  Furthermore,
\begin{align}
JaJ\Lambda(b) &= J\Lambda(a \sigma_{-i/2}(b^*))
= \Lambda(\sigma_{-i/2}( \sigma_{-i/2}(b^*)^*  a^* )) \notag \\
&= \Lambda(\sigma_{-i/2}( \sigma_{i/2}(b)  a^* ))
= \Lambda(b \sigma_{-i/2}(a^*)) \qquad (a,b\in B)
\label{eq:2}
\end{align}
and so $JaJ$ is the operation of right multiplication by $\sigma_{i/2}(a)^*$.  In fact, $B' = JBJ$
inside $\mc B(L^2(B))$.  Notice that
\[ J\theta_{\xi,\eta}J(\alpha) = J\big( (\xi|J(\alpha)) \eta \big)
= (J(\alpha)|\xi) J(\eta) = (J(\xi)|\alpha) J(\eta) = \theta_{J(\xi), J(\eta)}(\alpha)
\qquad (\xi,\eta,\alpha\in L^2(B)). \]
Recall that the \emph{modular operator} is $\nabla\in\mc B(L^2(B))^+$ given by
$\nabla\Lambda(a) = \Lambda(\sigma_{-i}(a)) = \Lambda(QaQ^{-1})$.  More generally,
$\nabla^z\Lambda(a) = \Lambda(\sigma_{-iz}(a)) = \Lambda(Q^z a Q^{-z})$.

We now define two bijections which will be central to our arguments.

\begin{definition}
Let $B^\op$ be the \emph{opposite} algebra to $B$, so $B^\op$ is the same vector space as $B$, and has the
same $*$-operation, but has the reversed product $a \stackrel{\op}{\cdot} b = ba$.  When it is clear from
context, we shall write the multiplication in $B^\op$ just by juxtaposition.
\end{definition}

As $L^2(B)$ is finite-dimensional, the linear span of the operators $\theta_{\Lambda(a), \Lambda(b)}$,
as $a,b\in B$ vary, gives all of $\mc B(L^2(B))$.  Given $s,t\in\mathbb R$, we define bijections from
$\mc B(L^2(B))$ to $B\otimes B^\op$ by
\[ \Psi_{t,s} = \Psi : \theta_{\Lambda(a), \Lambda(b)} \mapsto \sigma_{it}(a)^* \otimes \sigma_{is}(b),
\qquad
\Psi'_{t,s} = \Psi' : \theta_{\Lambda(a), \Lambda(b)} \mapsto \sigma_{it}(b) \otimes \sigma_{is}(a)^*. \]
As $\theta_{\cdot,\cdot}$ is anti-linear in the first variable, and linear in the second, we see that
$\Psi$ is itself linear, and so extends by linearity to a bijection $\mc B(L^2(B)) \rightarrow
B\otimes B^\op$.  Similarly, $\Psi'$ is linear.  We work with $B\otimes B^\op$ to match the existing literature,
which is natural for $\Psi'$, but we shall shortly see that when working
with $\Psi$, it might have been more natural to consider $B^\op\otimes B$.

\begin{definition}
We write $\sigma$ for the tensor swap map $B\otimes B^\op \rightarrow B\otimes B^\op; a\otimes b\mapsto
b\otimes a$, which is an anti-$*$-homomorphism.
\end{definition}

We now see how the bijections $\Psi$ and $\Psi'$ interact with the constructions which arise in
Definition~\ref{defn:adj_op}.  We let the modular automorphism group act also on $B^\op$.

\begin{proposition}\label{prop:psiiso}
The bijection $\Psi = \Psi_{t,s}$ gives the following correspondences between $A\in\mc B(L^2(B))$ and
$e=\Psi(A)\in B\otimes B^\op$:
\begin{enumerate}
\item\label{prop:psiiso:one}
  $A^*$ corresponds to $(\sigma_{i(s-t)}\otimes\sigma_{i(s-t)})\sigma(e^*)$;
\item\label{prop:psiiso:two}
  $(1\otimes\eta^* m)(1\otimes A\otimes 1)(m^*\eta\otimes 1)$ corresponds to
  $(\sigma_{-i(s+t)} \otimes \sigma_{i(t+s-1)})\sigma(e)$.
\item\label{prop:psiiso:twoa}
  $(\eta^*m\otimes 1)(1\otimes A\otimes 1)(1\otimes m^*\eta)$ corresponds to
  $(\sigma_{i(1-t-s)} \otimes \sigma_{i(s+t)})\sigma(e)$.
\item\label{prop:psiiso:three}
  we have that $\Psi(m(A_1\otimes A_2)m^*) = e_2e_1$ for $e_i = \Psi(A_i)$, $i=1,2$.
\end{enumerate} 
\end{proposition}
\begin{proof}
By linearity, it suffices to consider the case when $A=\theta_{\Lambda(a), \Lambda(b)}$ and so
$e = \sigma_{it}(a)^* \otimes \sigma_{is}(b)$.  Then $A^* = \theta_{\Lambda(b), \Lambda(a)}$ which
corresponds to
\[ \sigma_{it}(b)^* \otimes \sigma_{is}(a)
= \sigma_{i(s-t)}\sigma_{-is}(b^*) \otimes \sigma_{i(s-t)}\sigma_{it}(a)
= (\sigma_{i(s-t)}\otimes\sigma_{i(s-t)})\sigma(e^*), \]
as claimed for (\ref{prop:psiiso:one}).

Set $T = (1\otimes\eta^* m)(1\otimes A\otimes 1)(m^*\eta\otimes 1)$.  Let $m^*\eta = m^*\Lambda(1)
= \sum_i \xi_i \otimes \eta_i$, so that $(\Lambda(xy)|\Lambda(1)) = \sum_i (\Lambda(x)|\xi_i)
(\Lambda(y)|\eta_i)$ for each $x,y\in B$.  Then for $x,y\in B$,
\begin{align*}
(\Lambda(x)|T\Lambda(y)) &=
\sum_i \big( \Lambda(x)\otimes\Lambda(1) \big| (1\otimes m) (1\otimes A\otimes 1)
   \xi_i\otimes\eta_i\otimes\Lambda(y) \big) \\
&= \sum_i \big( \Lambda(x)\otimes\Lambda(1) \big| (1\otimes m)
   \xi_i\otimes\Lambda(b)\otimes\Lambda(y) \big) (\Lambda(a)|\eta_i) \\
&= \sum_i (\Lambda(x)|\xi_i) (\Lambda(1)|\Lambda(by)) (\Lambda(a)|\eta_i) \\
&= (\Lambda(xa)|\Lambda(1)) (\Lambda(1)|\Lambda(by))
= (\Lambda(x)|\Lambda(\sigma_{-i}(a^*)) ) (\Lambda(b^*)|\Lambda(y)).
\end{align*}
Hence $T = \theta_{\Lambda(b^*), \Lambda(\sigma_{-i}(a^*))}$ and so
\[ \Psi(T) = \sigma_{it}(b^*)^* \otimes \sigma_{is}(\sigma_i(a)^*)
= \sigma_{-it-is}\sigma_{is}(b) \otimes \sigma_{it+is-i}\sigma_{-it}(a^*)
= (\sigma_{-i(t+s)} \otimes \sigma_{i(t+s-1)}) \sigma(e), \]
showing (\ref{prop:psiiso:two}).

Similarly, if we now set $T = (\eta^*m\otimes 1)(1\otimes A\otimes 1)(1\otimes m^*\eta)$, then for $x,y\in B$,
\begin{align*}
(\Lambda(x)|T\Lambda(y)) &=
\sum_i \big( \Lambda(1)\otimes\Lambda(x) \big| (m\otimes 1) (1\otimes A\otimes 1)
   \Lambda(y)\otimes\xi_i\otimes\eta_i \big) \\
&= \sum_i \big( \Lambda(1)\otimes\Lambda(x) \big| \Lambda(yb)\otimes\eta_i \big) (\Lambda(a)|\xi_i) \\
&= \sum_i (\Lambda(1)|\Lambda(yb)) ( \Lambda(x) | \eta_i \big) (\Lambda(a)|\xi_i) \\
&= (\Lambda(\sigma_{-i}(b^*))|\Lambda(y)) (\Lambda(x)|\Lambda(a^*)).
\end{align*}
Thus $T = \theta_{\Lambda(\sigma_{-i}(b^*)), \Lambda(a^*)}$ and so
\[ \Psi(T) = (\sigma_{it}\sigma_{-i}(b^*))^* \otimes \sigma_{is}(a^*)
= \sigma_{-it+i-is}\sigma_{is}(b) \otimes \sigma_{is+it}\sigma_{-it}(a^*) 
= \sigma_{-i(t+s)+i} \otimes \sigma_{i(s+t)}) \sigma(e), \]
showing (\ref{prop:psiiso:twoa}).

Given $A_k = \theta_{\Lambda(a_k), \Lambda(b_k)}$, corresponding to $e_k$, for $k=1,2$, we see that
\[ m(A_1\otimes A_2)m^*\Lambda(c)
= \theta_{\Lambda(a_1 a_2), \Lambda(b_1 b_2)} \Lambda(c) \qquad (c\in B), \]
and so $m(A_1\otimes A_2)m^*$ corresponds to
\[ \sigma_{it}(a_1a_2)^* \otimes \sigma_{is}(b_1b_2)
= \sigma_{it}(a_2)^* \sigma_{it}(a_1)^* \otimes \sigma_{is}(b_2) \sigma_{is}(b_1)
= e_2 e_1, \]
as we work in $B\otimes B^\op$.  This shows (\ref{prop:psiiso:three}).
\end{proof}

We can immediately use this result to show one equivalence, removing the self-adjointness assumption from the discussion after Definition~\ref{defn:adj_op}.  I thank Monica Abu Omar for this observation.

\begin{proposition}\label{prop:undir_axioms_same}
Let $A\in\mc B(L^2(B))$ be any operator.  Then axioms (\ref{defn:quan_adj_mat:undir}) and
(\ref{defn:quan_adj_mat:undir:other}) of Definition~\ref{defn:adj_op} are equivalent.
\end{proposition}
\begin{proof}
Let $e = \Psi_{s,t}(A)$ (for some choice of $s,t$), so the previous proposition shows that axiom (\ref{defn:quan_adj_mat:undir}) is equivalent to $(\sigma_{-i(s+t)} \otimes \sigma_{i(t+s-1)})\sigma(e)=e$.  Applying $\sigma$, this is in turn equivalent to
\[ (\sigma_{i(t+s-1)} \otimes \sigma_{-i(s+t)})(e) = \sigma(e), \]
and then applying $\sigma_{-i(t+s-1)} \otimes \sigma_{i(s+t)}$, this is equivalent to
\[ e = (\sigma_{-i(t+s-1)} \otimes \sigma_{i(s+t)})\sigma(e)
= (\sigma_{i(1-t-s)} \otimes \sigma_{i(s+t)})\sigma(e), \]
which by the previous proposition, is equivalent to axiom (\ref{defn:quan_adj_mat:undir:other}).
\end{proof}

There is an analogous equivalence for $\Psi'$.

\begin{proposition}\label{prop:psiprimeiso}
The bijection $\Psi' = \Psi'_{t,s}$ gives the following correspondences between $A\in\mc B(L^2(B))$ and
$e=\Psi'(A)\in B\otimes B^\op$:
\begin{enumerate}
\item\label{prop:psiprimeiso:one}
  $A^*$ corresponds to $(\sigma_{i(t-s)}\otimes\sigma_{i(t-s)})\sigma(e^*)$;
\item\label{prop:psiprimeiso:two}
  $(1\otimes\eta^* m)(1\otimes A\otimes 1)(m^*\eta\otimes 1)$ corresponds to
  $(\sigma_{i(s+t-1)} \otimes \sigma_{-i(s+t)})\sigma(e)$.
\item\label{prop:psiprimeiso:twoa}
  $(\eta^*m\otimes 1)(1\otimes A\otimes 1)(1\otimes m^*\eta)$ corresponds to
  $(\sigma_{i(t+s)} \otimes \sigma_{i(1-s-t)})\sigma(e)$.
\item\label{prop:psiprimeiso:three}
  we have that $\Psi'(m(A_1\otimes A_2)m^*) = e_1e_2$ for $e_i = \Psi'(A_i)$, $i=1,2$.
\end{enumerate} 
\end{proposition}
\begin{proof}
This could be proved in an entirely analogous way.  However, we also observe that $\Psi'_{t,s}
= \sigma \circ \Psi_{s,t}$, from which these claims follow immediately from Proposition~\ref{prop:psiiso}.
\end{proof}

The way to prove the equivalence between Definitions~\ref{defn:adj_op} and~\ref{defn:quan_graph} requires
two steps:
\begin{itemize}
\item we use $\Psi_{t,s}$ or $\Psi'_{t,s}$, for carefully chosen $t,s$, to associate $A$ to an
\emph{idempotent} (ideally, an \emph{orthogonal projection}) $e$;
\item associate idempotents (or projections) $e$ to subspaces $S\subseteq\mc B(H)$.
\end{itemize}
We will shortly deal with the second point in detail, before then returning to the first point.
However, the reader can already see that if $\psi$ is a trace, then each $\sigma_z$ is trivial, and it is
immediate that $A=A^*$ satisfying axioms (\ref{defn:quan_adj_mat:idem}) and (\ref{defn:quan_adj_mat:undir}) of Definition~\ref{defn:adj_op}
corresponds to $e\in B\otimes B^\op$ with $\sigma(e^*) = e, e^2=e$ and $e = \sigma(e)$, or equivalently,
$e=e^*=e^2=\sigma(e)$, this using any choice of $\Psi$ and $\Psi'$.  In the non-tracial case,
it is not possible to choose any $s,t$ which gives such a direct bijection, though of course the benefit is that
we obtain an interesting, different, theory in this case.

\begin{remark}\label{rem:existing_bijects}
When $\psi$ is a tracial $\delta$-form, it was already shown in \cite[Theorem~7.7]{mrv} that a self-adjoint $A$
satisfying axioms (\ref{defn:quan_adj_mat:idem}) and (\ref{defn:quan_adj_mat:undir}) of Definition~\ref{defn:adj_op} (which are exactly
the diagrams (47) in \cite[Definition~5.1]{mrv}) then we obtain an undirected quantum graph, in the
sense of Definition~\ref{defn:quan_graph}.  The proof is graphical, and indirect, as it passes by way of
``projectors'', equivalently by \cite[Remark~7.3]{mrv} projections in $B^\op\otimes B$.  Examining the proof of
\cite[Theorem~7.7]{mrv} shows that $A = \theta_{\Lambda(a), \Lambda(b)}$ corresponds to $b\otimes a^*$, that is,
$\Psi'_{0,0}(A)$.  (Of course, $\psi$ is tracial here, so $\Psi'_{0,0} = \Psi'_{t,s}$ for all $t,s$).
Using $B^\op\otimes B$ instead of $B\otimes B^\op$ is just a convention.

Elsewhere in the literature, compare for example \cite[Lemma~1.6]{cw} or \cite[Definition~1.12]{grom},
the relation between $A\in\mc B(L^2(B))$ and $e\in B\otimes B^\op$ is given by $e = (A\otimes 1)m^*\eta$,
and this again is $\Psi'_{0,0}$.  To make sense of the expression $e = (A\otimes 1)m^*\eta$, we identify, as
linear spaces, $B\otimes B^\op$ with $B\otimes B$, and in turn, with $L^2(B)\otimes L^2(B)$.

We shall shortly see that $\Psi$ and $\Psi'$ are essentially interchangeable, but it is helpful to have the
choice between them for different arguments.
\end{remark}

Finally, we consider the real condition, Definition~\ref{defn:adj_op:real}.

\begin{lemma}\label{lem:what_real_means_general}
Given $A\in\mc B(L^2(B))$, let $A_0:B\rightarrow B$ be the associated linear map.  Let $A^r \in \mc B(L^2(B))$
correspond to the linear map $B\rightarrow B; a\mapsto A_0(a^*)^*$.  For $e = \Psi_{t,s}(A)$, we have that
$\Psi_{t,s}(A^r) = (\sigma_{i-2it}\otimes\sigma_{2is})(e^*)$.  For $e = \Psi'_{t,s}(A)$, we have that
$\Psi'_{t,s}(A^r) = (\sigma_{2it}\otimes\sigma_{i-2is})(e^*)$.
\end{lemma}
\begin{proof}
We wish to consider the anti-linear bijection $s:L^2(B)\rightarrow L^2(B); \Lambda(a) \mapsto \Lambda(a^*)$.
A calculation shows that $s = \nabla^{-1/2} J = J \nabla^{1/2}$.  A further calculation then establishes that
when $A = \theta_{\Lambda(a), \Lambda(b)}$,
\begin{align*}
A^r &= s A s = \nabla^{-1/2} J \theta_{\Lambda(a), \Lambda(b)} J \nabla^{1/2}
= \nabla^{-1/2} \theta_{J\Lambda(a), J\Lambda(b)} \nabla^{1/2}
= \theta_{\nabla^{1/2}J\Lambda(a), \nabla^{-1/2}J\Lambda(b)} \\
&= \theta_{\nabla\Lambda(a^*), \Lambda(b^*)}
=  \theta_{\Lambda(\sigma_{-i}(a^*)), \Lambda(b^*)}.
\end{align*}
Thus
\begin{align*}
\Psi_{t,s}(A^r) &= \sigma_{it}(\sigma_{-i}(a^*))^* \otimes \sigma_{is}(b^*)
= \sigma_{i-it}(a) \otimes \sigma_{is}(b^*) \\
&= (\sigma_{i-2it}\otimes\sigma_{2is})(\sigma_{it}(a) \otimes \sigma_{-is}(b^*))
= (\sigma_{i-2it}\otimes\sigma_{2is})(e^*).
\end{align*}
By linearity, this relation holds for all $A$, as claimed.

As $\Psi'_{t,s} = \sigma\Psi_{s,t}$, with $\sigma$ the tensor swap map,
we see that $\Psi'_{t,s}(A^r) = \sigma\Psi_{s,t}(A_r)
= \sigma (\sigma_{i-2is}\otimes\sigma_{2it}) \Psi_{s,t}(A)^*
= (\sigma_{2it}\otimes\sigma_{i-2is}) \sigma \Psi_{s,t}(A)^*
= (\sigma_{2it}\otimes\sigma_{i-2is}) \Psi'_{t,s}(A)^*$, as claimed.
\end{proof}

\subsection{From projections to subspaces}\label{sec:proj_to_ss}

In this section, $H$ will be a finite-dimensional Hilbert space with $B\subseteq\mc B(H)$ (which is equivalent,
of course, to there being an injective $*$-homomorphism $B\rightarrow \mc B(H)$ which allows us to identify $B$
with its image in $\mc B(H)$).  Then $B^\op$ is naturally a subalgebra of $\mc B(\overline H)$ where we identify
$a\in B^\op$ with $a^\top \in \mc B(\overline H)$.  Consequently, $B\otimes B^\op \subseteq
\mc B(H\otimes \overline H)$.  For the following, notice that $(B')^\op$ is also naturally a subalgebra of
$\mc B(\overline H)$, where we identify $x\in (B')^\op \subseteq\mc B(H)$ with $x^\top$.  A simple calculation
shows that then $(B')^\op = (B^\op)'$ as subalgebras of $\mc B(\overline H)$.  Hence $B' \otimes (B')^\op
\subseteq\mc B(H\otimes \overline H)$.

\begin{lemma}\label{lem:app_tomita}
The commutant of $B\otimes B^\op \subseteq \mc B(H\otimes\overline H)$ is $B'\otimes (B')^\op$.
\end{lemma}
\begin{proof}
The result follows from Tomita's theorem, \cite[Theorem~5.9, Chapter~IV]{tak1}: $(B\otimes B^\op)'
= B' \otimes (B^\op)' \cong B' \otimes (B')^\op$.
\end{proof}

As $H$ is finite-dimensional, every operator on $\mc B(H)$ is in the linear span of the rank-one operators,
and so there is a linear isomorphism
\[ \mc B(H) \cong H\otimes\overline{H};\quad \theta_{\xi,\eta}\mapsto \eta\otimes\overline\xi
\qquad (\xi,\eta\in H). \]
Indeed, this is simply the GNS map for the functional $\Tr$ on $\mc B(H)$, or equivalently, comes from
identifying $\mc B(H)$ with the Hilbert-Schmidt operators on $H$.

\begin{remark}
We make links with the existing literature, and in particular \cite[Proposition~2.23]{weaver1} where a bijection
between quantum relations $S\subseteq\mc B(H)$ and left ideals in $B\otimes B^\op$ is established (when $B$ is
finite-dimensional).  A key technical
tool for Weaver is the action $\Phi$ of $B\otimes B^\op$ on $\mc B(H)$ given by
\[ \Phi_{a\otimes b}(x) = abx \qquad (a\otimes b\in B\otimes B^\op, x\in\mc B(H)). \]
Once we identify $\mc B(H)$ with $H\otimes\overline H$, we find that
\[ (a\otimes b)\cdot (\eta\otimes\overline\xi) \cong \Phi_{a\otimes b}(\theta_{\xi,\eta})
= a \theta_{\xi,\eta} b = \theta_{b^*(\xi), a(\eta)}
\cong a(\eta) \otimes b^\top(\overline\xi). \]
That is, we obtain the natural action of $B\otimes B^\op$ on $H\otimes\overline H$.

Elsewhere in the literature, links with completely bounded maps (and implicitly with the (extended) Haagerup tensor
product) are made, compare \cite{gp}, or \cite[Section~2.4]{ganesan}.
In contrast, our perspective in this
section is that it is Hilbert Space techniques which are vitally important, for which $H$ being finite-dimensional
seems essential.
\end{remark}

An elementary result from von Neumann algebra theory is that invariant subspaces for a von Neumann algebra $M$
biject with projections in the commutant $M'$.  We would like to have a version of this result for (possibly)
non-self-adjoint idempotents.

\begin{lemma}\label{lem:non_sa_idem}
Let $M \subseteq\mc B(H)$ be a von Neumann algebra.  Let $e\in\mc B(H)$ be an idempotent, $e^2=e$, and let
$H_0 = \im(e), H_1=\ker(e)$.  Then $H = H_0 \oplus H_1$ (not necessarily orthogonal).  Furthermore,
$e\in M$ if and only if $H_0$ and $H_1$ are invariant subspaces for $M'$.
\end{lemma}
\begin{proof}
Any $\xi\in H$ is written as $\xi = e(\xi) + (\xi-e(\xi)) \in H_0 + H_1$, and if $\xi\in H_0\cap H_1$ then
$\xi=e(\xi)=0$.  Thus $H = H_0 \oplus H_1$ as a direct sum of Banach spaces (that is, a not necessarily
orthogonal direct sum).  If $e\in M$ then given $n\in M'$, for $\xi\in H_0$ we see that $n(\xi) = ne(\xi)
= en(\xi)$ so $n(\xi)\in H_0$, while for $\xi\in H_1$ we have that $en(\xi) = ne(\xi) = 0$ so $n(\xi)\in H_1$.
Conversely, if $H_0,H_1$ are $M'$-invariant, then to show that $e\in M$, by the bicommutant theorem, it
suffices to show that $ne=en$ for each $n\in M'$.  Given $\xi\in H$ write $\xi = \xi_0 + \xi_1 \in H_0
\oplus H_1$, so $ne(\xi) = n(\xi_0) = en(\xi_0)$ as $n(\xi_0)\in H_0$, while $n(\xi) = n(\xi_0) + n(\xi_1)
\in H_0\oplus H_1$ so $en(\xi) = en(\xi_0)$ because $en(\xi_1)=0$.  Hence $ne(\xi) = en(\xi)$ as required.
\end{proof}

We now collect some definitions which will allow us to talk about properties of subspaces of $H\otimes\overline H$
which correspond to the properties of idempotents $e\in B\otimes B^\op$ which occurred in the previous section.

\begin{definition}\label{defn:subsp_defs}
Given $H = H_0 \oplus H_1$ not necessarily orthogonal, we shall say that the idempotent $e\in \mc B(H)$
with $\im(e)=H_0, \ker(e)=H_1$ is the \emph{projection onto $H_0$ along $H_1$}.

We define an anti-linear map
\[ J_0: H\otimes\overline H \rightarrow H\otimes\overline H; \quad
\eta\otimes\overline\xi \mapsto \xi\otimes\overline\eta. \]
This corresponds to the adjoint on $\mc B(H)$, and is in fact the modular conjugation with respect to the
functional $\Tr$ on $\mc B(H)$.

Let $\hat 1\in H\otimes\overline H$ correspond to the identity operator $1\in\mc B(H)$.
\end{definition}

Let $(e_i)$ be some orthonormal basis for $H$.  As $1 = \sum_i \theta_{e_i, e_i}$ we find that
\[ \hat 1 = \sum_i e_i\otimes\overline{e_i}, \]
and that this is independent of the choice of $(e_i)$.

\begin{theorem}\label{thm:non_sa_op_bimods}
There is a bijection between idempotents $e\in B\otimes B^\op$ and decompositions
$H\otimes\overline H = V \oplus W$ where $V$ and $W$ are $B'\otimes (B')^\op$-invariant subspaces of
$H\otimes\overline H$, given by $e$ being the projection onto $V$ along $W$.  Furthermore:
\begin{enumerate}
\item\label{thm:non_sa_op_bimods:one}
$e = \sigma(e^*)$ if and only if $J_0(V)=V$ and $J_0(W)=W$;
\item\label{thm:non_sa_op_bimods:two}
$m(e)=1$ if and only if $\hat 1\in V$.  (Here $m:B\otimes B^\op\rightarrow B$ is the multiplication map).
\end{enumerate}
\end{theorem}
\begin{proof}
The bijection follows from Lemma~\ref{lem:non_sa_idem} combined with Lemma~\ref{lem:app_tomita}.

When $e=a\otimes b$ acting as $a\otimes b^\top$ on $H\otimes\overline H$, so that $e(\eta\otimes\overline\xi)
= a(\eta) \otimes \overline{b^*(\xi)}$, we see that $\sigma(e^*) = b^* \otimes a^*$, so that 
\[ \sigma(e^*)(\eta\otimes\overline\xi) = b^*(\eta) \otimes \overline{a(\xi)}
= J_0 e J_0 (\eta\otimes\overline\xi). \]
Thus in general $\sigma(e^*) = J_0 e J_0$, and so
$e = \sigma(e^*)$ if and only if $e J_0 = J_0 e$.  If this holds, then 
$J_0(V) = J_0e(H\otimes\overline H) = eJ_0(H\otimes\overline H) = V$, while also if $\xi\in W$ then
$eJ_0(\xi) = J_0e(\xi) = 0$ so $J_0(\xi)\in\ker(e)=W$ and hence $W = J_0(W)$ (as $J_0^2 = 1$).

Conversely, if $J_0(V) = V$ and $J_0(W) = W$, then given $\xi \in H\otimes\overline H$ take the decomposition
$\xi = v + w \in V\oplus W$.  By assumption, $J_0(\xi)$ has decomposition $J_0(v) + J_0(w) \in V \oplus W$.
Hence $e(\xi) = v$ and $eJ_0(\xi) = J_0(v)$, and so
\[ \sigma(e^*)(\xi) = J_0 e J_0 (\xi) = J_0 J_0(v) = v = e(\xi). \]
Hence $\sigma(e^*) = e$, and we have shown (\ref{thm:non_sa_op_bimods:one}).

Again, when $e=a\otimes b$, and with $(e_i)$ some orthonormal basis of $H$,
\begin{align*}
e(\hat 1) &= \sum_i a(e_i) \otimes \overline{b^*(e_i)}
= \sum_{i,j,k} (e_j|a(e_i)) e_j \otimes (\overline{e_k}|\overline{b^*(e_i)}) \overline{e_k}
= \sum_{i,j,k} (a^*(e_j)|e_i)(e_i|b(e_k)) e_j \otimes \overline{e_k} \\
&= \sum_{j,k} (a^*(e_j)|b(e_k)) e_j \otimes \overline{e_k}
= \sum_{j,k} (e_j|m(e)(e_k)) e_j \otimes \overline{e_k}.
\end{align*}
Thus, if $m(e)=1$ then $e(\hat 1) = \hat 1$.  Conversely, if $e(\hat 1)=\hat 1$ then $(e_j|m(e)(e_k))
= \delta_{j,k}$ for each $j,k$ so that $m(e)=1$.  So (\ref{thm:non_sa_op_bimods:two}) holds.
\end{proof}

In the orthogonal case, we recover, in particular, \cite[Proposition~2.23]{weaver1}.

\begin{corollary}\label{corr:sa_op_bimods}
There is a bijection between $B'$-operator bimodules $S\subseteq\mc B(H)$ and projections $e\in B\otimes B^\op$,
under which $S$ is self-adjoint if and only if $e = \sigma(e)$, and $1\in S$ if and only if $m(e)=1$.
\end{corollary}
\begin{proof}
Under the bijection between $\mc B(H)$ and $H\otimes \overline H$, we see that $B'$-operator bimodules
correspond to $B'\otimes (B')^\op$-invariant subspaces of $H\otimes \overline H$.  These biject with
projections $e\in B\otimes B^\op$, given by taking $V=S$ and $W=S^\perp$ in the previous theorem.

If $J_0(S)=S$ then for $\xi\in S$ and $\eta\in S^\perp$, we see that $(\xi|J_0(\eta)) = (\eta|J_0(\xi)) = 0$
as $J_0(\xi)\in S$.  Hence also $J_0(S^\perp) \subseteq S^\perp$ (so we have equality as $J_0^2=1$).  As
$e=e^*$, we conclude that $e = \sigma(e)$ is equivalent to $J_0(S)=S$, which viewing $S$ as a subspace of
$\mc B(H)$ corresponds to $S$ being self-adjoint.  Similarly, as $\hat 1$ corresponds to the unit $1\in\mc B(H)$,
we see that $S$ is unital if and only if $m(e)=1$.
\end{proof}

\subsection{From adjacency matrices to subspaces}\label{sec:main_equivs}

We now put together the results of the previous two sections to establish bijections between our two main
definitions.  We first consider just looking at the first axiom (\ref{defn:quan_adj_mat:idem}), $m(A\otimes A)m^*=A$.

\begin{theorem}\label{thm:qu_ad_mats_axiom_1}
For any $t\in\mathbb R$, consider $\Psi = \Psi_{t,t}$ or $\Psi'=\Psi'_{t,t}$.  These maps
$\mc B(L^2(B)) \rightarrow B\otimes B^\op$ give a bijection between:
\begin{enumerate}
\item\label{thm:qu_ad_mats_axiom_1:one}
self-adjoint $A\in\mc B(L^2(B))$ satisfying axiom (\ref{defn:quan_adj_mat:idem});
\item\label{thm:qu_ad_mats_axiom_1:two}
idempotents $e\in B\otimes B^\op$ with $e = \sigma(e^*)$;
\item\label{thm:qu_ad_mats_axiom_1:threea}
given $B\subseteq\mc B(H)$, decompositions $H\otimes\overline H = V\oplus W$ with $V,W$ being
$B'\otimes (B')^\op$-invariant, with $J_0(V)=V, J_0(W)=W$;
\item\label{thm:qu_ad_mats_axiom_1:three}
given $B\subseteq\mc B(H)$, decompositions $\mc B(H) = V\oplus W$ with $V,W$ self-adjoint, $B'$-bimodules.
\end{enumerate}
\end{theorem}
\begin{proof}
Proposition~\ref{prop:psiiso} shows that $\Psi_{t,t}$ gives a bijection between $A$ and $e$ such that
$A^*$ corresponds to $\sigma(e^*)$ and $m(A\otimes A)m^*$ corresponding to $e^2$.  
Similarly, Proposition~\ref{prop:psiprimeiso} does the same for $\Psi'_{t,t}$.
Hence (\ref{thm:qu_ad_mats_axiom_1:one}) and (\ref{thm:qu_ad_mats_axiom_1:two}) are equivalent.
We now use Theorem~\ref{thm:non_sa_op_bimods} to see that such $e$
correspond to projections of $H\otimes\overline H$ onto $V$ along $W$ where $V,W$ are $J_0$-invariant,
and $B'\otimes (B')^\op$-invariant, which gives equivalence (\ref{thm:qu_ad_mats_axiom_1:threea}).
This is exactly condition (\ref{thm:qu_ad_mats_axiom_1:three}) using $H\otimes\overline H$ in place
of $\mc B(H)$.
\end{proof}

\begin{remark}
In particular, this applies to the bijection $\Psi'_{0,0}$ as used elsewhere in the literature (in the
tracial case), see Remark~\ref{rem:existing_bijects}.  However, we notice that the condition we obtain is that
$e=e^2=\sigma(e^*)$, and certainly not that $e$ need be self-adjoint.
Compare also Remark~\ref{rem:needs_axioms_joined} below.

In light of just the use of axiom (\ref{defn:quan_adj_mat:idem}) in \cite{bevw}, we note that the same proof
shows that we can remove that condition that $A=A^*$ if we remove the conditions $e=\sigma(e^*)$, that $V,W$
are $J_0$-invariant, and the condition that $V,W$ be self-adjoint.
\end{remark}

To make further progress, it really does seem necessary to take account of the modular automorphism group.
First a lemma.

\begin{lemma}\label{lem:proj_commuting_op}
Let $H$ be a Hilbert space, $E\in\mc B(H)$ the (orthogonal) projection onto a subspace $V\subseteq H$,
and let $T\in\mc B(H)$ be invertible.  The following are equivalent:
\begin{enumerate}
\item $T^{-1} E T = E$;
\item $T(V) = V$ and $T(V^\perp)=V^\perp$;
\item $T(V)\subseteq V$ and $T(V^\perp)\subseteq V^\perp$.
\end{enumerate}
When $T = T^*$, these conditions are equivalent to $T(V)\subseteq V$.
\end{lemma}
\begin{proof}
Suppose $T^{-1} E T = E$.  For $\xi\in V$, we have $\xi = E(\xi) = T^{-1}ET(\xi)$ so $T(\xi) = ET(\xi)$,
so $T(\xi)\in V$.  For $\xi\in V^\perp$, we have $0 = E(\xi) = T^{-1}ET(\xi)$ so $ET(\xi)=0$ so
$T(\xi)\in V^\perp$.  As also $TET^{-1}=E$ we get that also $T^{-1}(V)\subseteq V$ and $T^{-1}(V^\perp)
\subseteq V^\perp$.  Thus $T(V) = V$ and $T(V^\perp)=V^\perp$.

If $T(V)\subseteq V$ and $T(V^\perp)\subseteq V^\perp$, then for $\xi\in H$ we have that $\xi = E(\xi)
+ (1-E)(\xi) \in V\oplus V^\perp$ and so also $T(\xi) = TE(\xi) + T(1-E)(\xi) \in V \oplus V^\perp$.
Hence $ET(\xi) = TE(\xi)$, for each $\xi$, and so $T^{-1}ET=E$.

When $T$ is self-adjoint, if $T(V)\subseteq V$ then for $\xi\in V^\perp, \eta\in V$, we see that
$(T(\xi)|\eta) = (\xi|T(\eta)) = 0$ because $T(\eta)\in V$.  Thus $T(V^\perp) \subseteq V^\perp$.
\end{proof}

We now come to the first of two key theorems which establish a bijection between
Definition~\ref{defn:adj_op} and a generalisation of Definition~\ref{defn:quan_graph}.
Recall that
\[ \Psi_{0,1/2}:\theta_{\Lambda(a), \Lambda(b)} \mapsto a^*\otimes\sigma_{i/2}(b),
\qquad
\Psi'_{1/2,0}:\theta_{\Lambda(a), \Lambda(b)} \mapsto \sigma_{i/2}(b)\otimes a^*. \]
Recall the meaning of $J_0$ from Definition~\ref{defn:subsp_defs}.

\begin{theorem}\label{thm:qu_ad_mats_axioms_12}
Consider the isomorphisms $\Psi_{0,1/2}$ or $\Psi'_{1/2,0}$.  These give a bijection between:
\begin{enumerate}
\item\label{thm:qu_ad_mats_axioms_12:one}
self-adjoint $A\in\mc B(L^2(B))$ satisfying axioms (\ref{defn:quan_adj_mat:idem}) and
(\ref{defn:quan_adj_mat:undir}) (or, equivalently, axioms (\ref{defn:quan_adj_mat:idem}) and
(\ref{defn:quan_adj_mat:undir:other})) of Definition~\ref{defn:adj_op};
\item\label{thm:qu_ad_mats_axioms_12:two}
projections $e\in B\otimes B^\op$ with $e = \sigma(e)$ and such that $e = (\sigma_z\otimes\sigma_z)(e)$
for all $z\in\mathbb C$;
\item\label{thm:qu_ad_mats_axioms_12:three}
given $B\subseteq\mc B(H)$, subspaces $V \subseteq H\otimes\overline H$ which are
$B'\otimes (B')^\op$-invariant, and which satisfy $J_0(V) = V$ and
$(Q\otimes (Q^{-1})^\top)(V)\subseteq V$ (equivalently, equal to $V$);
\item\label{thm:qu_ad_mats_axioms_12:four}
given $B\subseteq\mc B(H)$, self-adjoint $B'$-bimodules $S\subseteq\mc B(H)$ with $Q S Q^{-1}
\subseteq S$ (equivalently, equals $S$).
\end{enumerate}
\end{theorem}
\begin{proof}
By Proposition~\ref{prop:psiiso}, $\Psi_{0,1/2}$ gives a bijection between self-adjoint $A$ satisfying
axiom (\ref{defn:quan_adj_mat:undir}), and $e$ with $e = (\sigma_{i/2}\otimes\sigma_{i/2})\sigma(e^*)$ and
$e=(\sigma_{-i/2}\otimes\sigma_{-i/2})\sigma(e)$.  For $\Psi' = \Psi'_{1/2,0}$ we use
Proposition~\ref{prop:psiprimeiso}, and obtain exactly the same results.  For such an $e$,
we have that $e^* = (\sigma_{-i/2}\otimes\sigma_{-i/2})\sigma(e)^* =
(\sigma_{i/2}\otimes\sigma_{i/2})\sigma(e^*) = e$ so $e$ is self-adjoint.  Then that
$e^* = (\sigma_{i/2}\otimes\sigma_{i/2})\sigma(e^*)$ and that $e=e^*$ means that
\[ (\sigma_{i/2}\otimes\sigma_{i/2})\sigma(e) = e = (\sigma_{-i/2}\otimes\sigma_{-i/2})\sigma(e)
\implies \sigma(e) = (\sigma_i\otimes\sigma_i)\sigma(e),  \]
or equivalently, $e = (\sigma_i\otimes\sigma_i)(e)$.  As $\sigma_i(a) = Q^{-1}aQ$ for $a\in B$, and using
the opposite multiplication in $B^\op$, we see that
\[ e = (\sigma_i\otimes\sigma_i)(e) = (Q^{-1}\otimes Q)e(Q\otimes Q^{-1}), \]
and so $e$ commutes with $Q\otimes Q^{-1}$.  By using Functional Calculus, $e$ also commutes with any power,
and we conclude that $(\sigma_{z}\otimes\sigma_{z})(e) = e$ for any $z\in\mathbb C$.  In particular,
then also $e = \sigma(e)$.

Conversely, if $(\sigma_{z}\otimes\sigma_{z})(e) = e$ for all $z$, then by Proposition~\ref{prop:psiiso},
respectively Proposition~\ref{prop:psiprimeiso}, that $e=e^*=\sigma(e)$ implies that $A=A^*$, and that
axiom (\ref{defn:quan_adj_mat:undir}) holds.  Proposition~\ref{prop:undir_axioms_same} tells us that
axioms (\ref{defn:quan_adj_mat:undir}) and (\ref{defn:quan_adj_mat:undir:other}) are equivalent.  Finally, we use that axiom (\ref{defn:quan_adj_mat:idem}) is equivalent to $e^2=e$.
Thus (\ref{thm:qu_ad_mats_axioms_12:one}) and (\ref{thm:qu_ad_mats_axioms_12:two}) are equivalent.

Suppose that (\ref{thm:qu_ad_mats_axioms_12:two}) holds.  By Corollary~\ref{corr:sa_op_bimods}, we know that
$e$ corresponds to a self-adjoint $B'$-bimodule $S$, equivalently, to $V\subseteq H\otimes\overline H$ which is
$B'\otimes (B')^\op$-invariant, and with $J_0(V)=V$.  It hence remains to consider what the additional
condition $(\sigma_z\otimes\sigma_z)(e)=e$ corresponds to.  This condition is equivalent to $e$ commuting with
$Q\otimes Q^{-1}$, so in $\mc B(H\otimes\overline H)$, that it commutes with $Q\otimes (Q^{-1})^\top$.
By Lemma~\ref{lem:proj_commuting_op}, this is equivalent to $(Q\otimes (Q^{-1})^\top)(V)\subseteq V$
(or with equality).  Hence (\ref{thm:qu_ad_mats_axioms_12:two}) and (\ref{thm:qu_ad_mats_axioms_12:three})
are equivalent.

Finally, the bijection from $\mc B(H)$ to $H\otimes\overline{H}$ sends $\theta_{\xi,\eta}$ to
$\eta\otimes\overline\xi$.  Then
\[ (Q \otimes (Q^{-1})^\top)(\eta\otimes\overline\xi)
= Q(\eta) \otimes \overline{ Q^{-1}(\xi) }, \]
which corresponds to $Q \theta_{\xi,\eta} Q^{-1}$.  It follows that (\ref{thm:qu_ad_mats_axioms_12:three})
is equivalent to (\ref{thm:qu_ad_mats_axioms_12:four}).
\end{proof}

We immediately recover the known results in the tracial case, as here the modular automorphism group is
trivial, and so the conditions involving $\sigma_z$ or $Q$ become vacuous.

\begin{remark}\label{rem:psi_psi_prime_agree_A}
As $\Psi' = \sigma\Psi$, and we have shown that $e=\sigma(e)$, it should come as no surprise that
working with $\Psi$ or $\Psi'$ are equivalent, because if $A$ satisfies the first two axioms, then
in fact $\Psi(A) = \Psi'(A)$.
\end{remark}

\begin{remark}\label{rem:dont_need_to_use_Q}
The operator $Q\in B$ appears in parts (\ref{thm:qu_ad_mats_axioms_12:three}) and
(\ref{thm:qu_ad_mats_axioms_12:four}) only because $Q$ implements $(\sigma_t)$.  In particular, if $q$ is any
positive invertible operator on $H$ with $\sigma_t(a) = q^{it} a q^{-it}$ for $a\in B,t\in\mathbb R$, then
we can replace the condition in (\ref{thm:qu_ad_mats_axioms_12:three}) with $(q\otimes (q^{-1})^\top)(V)=V$
and that in (\ref{thm:qu_ad_mats_axioms_12:four}) with $q S q^{-1}=S$.  A special case of this observation
occurs when $H=L^2(B)$ when we could use the operator $\nabla$.
\end{remark}

\begin{remark}\label{rem:needs_axioms_joined}
As we hinted at above, there are really three natural axioms for $A$: that $A$ is self-adjoint, the axiom
(\ref{defn:quan_adj_mat:undir}) (or equivalently (\ref{defn:quan_adj_mat:undir:other})) and the axiom
(\ref{defn:quan_adj_mat:idem}) of Definition~\ref{defn:adj_op}.  Axiom (\ref{defn:quan_adj_mat:idem}) corresponds to $e^2=e$ and hence to
a subspace of $\mc B(H)$ (or more accurately, a choice of subspace and complement, not necessarily the
orthogonal complement).  That $A=A^*$ taken together with axiom (\ref{defn:quan_adj_mat:undir}) then yields that
$e=e^*=\sigma(e)$ (and the condition of being invariant under the modular automorphism group) and these
then correspond to an orthogonal decomposition of $\mc B(H)$, and the subspace $S$ being self-adjoint.

However, we wish to stress that taken individually, it is not clear what $A^*=A$ alone, or axiom 
(\ref{defn:quan_adj_mat:undir}) alone, correspond to, for a subspace of $\mc B(H)$.
\end{remark}

We now consider Definition~\ref{defn:adj_op:real}, that of $A$ being ``real''.  This provides
motivation for looking at only the maps $\Psi_{0,1/2}$ and $\Psi'_{1/2,0}$ (and not other values of $s,t$).
Indeed, perhaps the theory might better be developed by initially looking at just idempotent, real
adjacency matrices $A$, but this would take us far from most of the existing literature, excepting
\cite{matsuda}.

\begin{proposition}\label{prop:what_real_means}
Under the isomorphism $\Psi_{0,1/2}$ or $\Psi'_{1/2,0}$, we obtain a bijection between real $A$  and
$e$ such that $(\sigma_{-i/2}\otimes\sigma_{-i/2})(e)$ is self-adjoint.  In particular, additionally
$A$ also satisfies axiom (\ref{defn:quan_adj_mat:idem}) of Definition~\ref{defn:adj_op} if and only if $(\sigma_{-i/2}\otimes\sigma_{-i/2})(e)$
is a projection.
\end{proposition}
\begin{proof}
From Lemma~\ref{lem:what_real_means_general} we see that $A$ is real if and only if $A = A^r$, and we have that
$\Psi_{0,1/2}(A^r) = (\sigma_i\otimes\sigma_i)(\Psi_{0,1/2}(A)^*)$ and $\Psi'_{1/2,0}(A^r) =
(\sigma_i\otimes\sigma_i) (\Psi'_{1/2,0}(A)^*)$.  Thus, for either isomorphism, $A$ is real if and only if
$(\sigma_i\otimes\sigma_i)(e^*) = e$, equivalently, $(\sigma_{-i/2}\otimes\sigma_{-i/2})(e)
= (\sigma_{i/2}\otimes\sigma_{i/2})(e^*) = (\sigma_{-i/2}\otimes\sigma_{-i/2})(e)^*$, as claimed.

As $A$ satisfying axiom (\ref{defn:quan_adj_mat:idem}) is equivalent to $e^2=e$, and as
$\sigma_{-i/2}\otimes\sigma_{-i/2}$ is an algebra homomorphism (but not of course a $*$-homomorphism)
we see that $e^2=e$ and $f =(\sigma_{-i/2}\otimes\sigma_{-i/2})(e)$ being self-adjoint implies that
$f=f^*=f^2$.  Conversely, if $f =(\sigma_{-i/2}\otimes\sigma_{-i/2})(e)$ is a projection, then $f$
is self-adjoint, and $(\sigma_{-i/2}\otimes\sigma_{-i/2})(e^2) = (\sigma_{-i/2}\otimes\sigma_{-i/2})(e)$
which implies that $e^2 = e$.
\end{proof}

\begin{remark}\label{rem:real_vs_other_axioms}
Using $\Psi_{0,1/2}$ for example, link $A$ and $e$, and set $f = (\sigma_{-i/2}\otimes\sigma_{-i/2})(e)$.
As in the proof of Theorem~\ref{thm:qu_ad_mats_axioms_12},
we see that $A=A^*$ if and only if $e=\sigma(f^*)$, and $A$ is undirected (satisfies axiom
(\ref{defn:quan_adj_mat:undir}) of Definition~\ref{defn:adj_op}) if and only if $e=\sigma(f)$.  Thus $A$ is real and self-adjoint means
that $f=f^*$ and $e=\sigma(f^*)$, implies that $A$ is undirected, in accordance with
\cite[Definition~2.19]{matsuda}.  Similarly, axiom (\ref{defn:quan_adj_mat:undir:other}) is equivalent to
$e = (\sigma_{i/2}\otimes\sigma_{i/2})\sigma(e)$ equivalently, $f = \sigma(e)$.

In our notation, \cite[Lemma~2.22]{matsuda} shows that of the conditions $A=A^*$, axiom
(\ref{defn:quan_adj_mat:undir:other}), and that $A$ is real, any two imply the third.  We translate these
conditions to being that $e=\sigma(f^*)$, $e = \sigma(f)$, and $f=f^*$, and it is now clear that any
two imply the third, giving a different proof of \cite[Lemma~2.22]{matsuda}.
\end{remark}

That $(\sigma_z\otimes\sigma_z)(e) = e$ implies commutation results for the adjacency matrix $A$.

\begin{proposition}\label{prop:A_comms_J_nabla}
Let $A\in\mc B(L^2(B))$ be self-adjoint and satisfy axioms (\ref{defn:quan_adj_mat:idem}) and
(\ref{defn:quan_adj_mat:undir}) of Definition~\ref{defn:adj_op}.  Then $A$ commutes with both the modular operator $\nabla$, and the
modular conjugation $J$.
\end{proposition}
\begin{proof}
Let $\Psi = \Psi_{0,1/2}$.
Suppose that $A=\theta_{\Lambda(a),\Lambda(b)}$ and let $e = \Psi(A) = a^* \otimes \sigma_{i/2}(b)$.  Then
$(\sigma_z\otimes\sigma_z)(e) = \sigma_{-z}(a)^* \otimes \sigma_{i/2+z}(b)$ which corresponds to
$\theta_{\Lambda(\sigma_{-z}(a)), \Lambda(\sigma_z(b))}$.  Let $z=i$ so as $\Lambda(\sigma_{-i}(a))
= \nabla\Lambda(a)$ and $\Lambda(\sigma_{i}(b)) = \nabla^{-1}\Lambda(b)$, it follows that
$\theta_{\Lambda(\sigma_{-z}(a)), \Lambda(\sigma_z(b))} = \nabla^{-1} A \nabla$.

Similarly, a simple calculation shows that
\[ JAJ = \theta_{J\Lambda(a), J\Lambda(b)}
= \theta_{\Lambda(\sigma_{-i/2}(a^*)), \Lambda(\sigma_{-i/2}(b^*))} \]
and applying $\Psi^{-1}$ we obtain
\[ \sigma_{-i/2}(a^*)^* \otimes \sigma_{i/2}(\sigma_{-i/2}(b^*))
= \sigma_{i/2}(a) \otimes b^*
= (\sigma_{i/2} \otimes \sigma_{i/2})((a^*\otimes \sigma_{i/2}(b))^*)
= (\sigma_{i/2} \otimes \sigma_{i/2})(e^*). \]

By linearity, these relations hold for any $A$ and $e=\Psi(A)$.
Thus, if $A=A^*$ satisfy axioms (\ref{defn:quan_adj_mat:idem}) and
(\ref{defn:quan_adj_mat:undir}) then $e=\Psi(A)$ satisfies $(\sigma_i\otimes\sigma_i)(e)=e$ and so
$\nabla^{-1}A\nabla = A$, equivalently, $A$ commutes with $\nabla$.  Furthermore, as $e = e^*$ and
$(\sigma_{i/2}\otimes\sigma_{i/2})(e)=e$, we conclude that $JAJ=A$.
\end{proof}

We now consider the axioms (\ref{defn:quan_adj_mat:reflexive}) and
(\ref{defn:quan_adj_mat:reflexive:other}) of Definition~\ref{defn:adj_op}.

\begin{lemma}\label{lem:meaning_of_axiom3}
Let $A\in\mc B(L^2(B))$ have the form $A = \theta_{\Lambda(a), \Lambda(b)}$ for some $a,b\in B$.  Then,
with $B$ acting on $L^2(B)$ in the usual way,
\begin{enumerate}
\item $m(A\otimes 1)m^* = ba^*$;
\item $m(1\otimes A)m^* = J \sigma_{i/2}(b)^* \sigma_{i/2}(a) J$, this is,
right multiplication by $\sigma_i(a)^*b$, compare with equation \eqref{eq:2}.
\end{enumerate}
\end{lemma}
\begin{proof}
Let $(f_j)$ be elements of $B$ such that $(\Lambda(f_j))$ is an orthonormal basis for $L^2(B)$, so that
$1 = \sum_j \theta_{\Lambda(f_j), \Lambda(f_j)}$.  Thus
\begin{align*}
m(A\otimes 1)m^*
&= \sum_j m(\theta_{\Lambda(a), \Lambda(b)} \otimes
   \theta_{\Lambda(f_j),\Lambda(f_j)} )m^*
= \sum_j \theta_{\Lambda(af_j), \Lambda(bf_j)} \\
&= \sum_j \theta_{a\Lambda(f_j), b\Lambda(f_j)}
= b \Big( \sum_j \theta_{\Lambda(f_j), \Lambda(f_j)} \Big) a^*
= ba^*,
\end{align*}
as claimed.  Similarly,
\begin{align*}
m(1\otimes A)m^*
&= \sum_j \theta_{\Lambda(f_j a), \Lambda(f_j b)}
= \sum_j \theta_{J\sigma_{-i/2}(a^*)J\Lambda(f_j), J\sigma_{-i/2}(b^*)J\Lambda(f_j)}
= J\sigma_{-i/2}(b^*)J J\sigma_{i/2}(a)J,
\end{align*}
as claimed.
\end{proof}

For the following, recall the meaning of $\hat 1$ from Definition~\ref{defn:subsp_defs}.

\begin{proposition}\label{prop:qu_ad_mats_axiom_3}
With the conventions of Theorem~\ref{thm:qu_ad_mats_axioms_12}, let $B\subseteq\mc B(H)$, and let $(e_j)$
be an orthonormal basis of $H$.  Set
\[ u_0 = (1\otimes (Q^{-1/2})^\top)\hat 1 = \sum_j e_j \otimes\overline{ Q^{-1/2}(e_j) },
\quad u_1 = (Q^{-1/2}\otimes 1) \hat 1 = \sum_j Q^{-1/2}(e_j) \otimes \overline{e_j}, \]
both vectors in $H\otimes \overline H$.

Let $e = \Psi_{0,1/2}(A)$ or $e = \Psi'_{1/2,0}(A)$.  The following are equivalent, assuming the equivalent
conditions of Theorem~\ref{thm:qu_ad_mats_axioms_12} already hold.
\begin{enumerate}
\item\label{prop:qu_ad_mats_axiom_3:one}
  $A\in\mc B(L^2(B))$ satisfies axiom (\ref{defn:quan_adj_mat:reflexive}) or, equivalently,
  axiom (\ref{defn:quan_adj_mat:reflexive:other}) of Definition~\ref{defn:adj_op};
\item\label{prop:qu_ad_mats_axiom_3:two}
  the projection $e$, acting on $H\otimes\overline H$, satisfies
  $e (u_0) = u_0$ or, equivalently, $e (u_1) = u_1$;
\item\label{prop:qu_ad_mats_axiom_3:three}
  the subspace $V\subseteq H\otimes\overline H$ contains $u_0$, equivalently, contains $u_1$;
\item\label{prop:qu_ad_mats_axiom_3:four}
  the subspace $S\subseteq \mc B(H)$ contains $Q^{-1/2}$;
\item\label{prop:qu_ad_mats_axiom_3:five}
  $m(\sigma_{-i/2}\otimes\id)(e) = 1$ where $m:B\otimes B^\op\rightarrow B$ is the multiplication map.
\end{enumerate}
\end{proposition}
\begin{proof}
Let $B\subseteq\mc B(H)$, and let $A = \sum_i \theta_{\Lambda(a_i), \Lambda(b_i)}$.  From
Lemma~\ref{lem:meaning_of_axiom3}, axiom (\ref{defn:quan_adj_mat:reflexive}) is equivalent to
$\sum_i a_i b_i^* = 1$.  For $\xi,\eta\in H$ we see that
\begin{align}
(\xi|\eta) &= \sum_j (\xi|e_j)(\overline\eta|\overline{e_j}) = (\xi\otimes\overline\eta|\hat 1), \notag \\
\sum_i (\xi|a_i b_i^*\eta) &= \sum_{i,j} (a_i^*(\xi)|e_j)(\overline{b_i^*(\eta)}|\overline{e_j})
= \sum_i (a_i^*(\xi) \otimes b_i^\top(\overline\eta)|\hat 1). \label{eq:3}
\end{align}
Using $\Psi = \Psi_{0,1/2}$, we see that $e$ acts as $\sum_i a_i^* \otimes \sigma_{i/2}(b_i)^\top$ and so
from \eqref{eq:3} and using that $\sigma_{i/2}(b) = Q^{-1/2} b Q^{1/2}$,
\begin{align}
\sum_i (\xi|a_i b_i^*\eta) &=
\sum_i (a_i^*(\xi) \otimes (Q^{1/2} \sigma_{i/2}(b_i) Q^{-1/2})^\top(\overline\eta)|\hat 1)
= \big( (1\otimes {Q^{-1/2}}^\top) e (1\otimes {Q^{1/2}}^\top)(\xi\otimes\overline\eta) \big| \hat 1 \big).
\label{eq:4}
\end{align}
Hence $\sum_i a_ib_i^*=1$ if and only if $(1\otimes {Q^{1/2}}^\top) e (1\otimes {Q^{-1/2}}^\top)(\hat 1)
= \hat 1$, equivalently, $(1\otimes {Q^{-1/2}}^\top)(\hat 1) = u_0$ is in the image of $e$, namely $V$.

Consider now $m(1\otimes A)m^* = 1$, which by Lemma~\ref{lem:meaning_of_axiom3} is equivalent to
$\sum_i \sigma_{-i}(b_i^*) a_i = 1$.  Observe that for $\xi,\eta\in H$,
\begin{align}
\sum_i (\xi|\sigma_{-i}(b_i^*) a_i(\eta)) &= \sum_{i,j} (\sigma_i(b_i)(\xi)|e_j)(e_j|a_i(\eta))
= \sum_{i,j} (\sigma_i(b_i)^\top(\overline{e_j})|\overline\xi)(a_i^*(e_j)|\eta) \notag \\
&= \sum_i \big((a_i^* \otimes \sigma_i(b_i)^\top)\hat 1\big|\eta \otimes \overline\xi\big)
= \sum_i \big((a_i^* \otimes (Q^{-1/2}\sigma_{i/2}(b_i)Q^{1/2})^\top)\hat 1\big|\eta \otimes \overline\xi\big)
  \notag \\
&= \big( (1\otimes {Q^{1/2}}^\top) e (1\otimes {Q^{-1/2}}^\top) \hat1 \big| \eta \otimes \overline\xi\big),
\label{eq:6}
\end{align}
while $(\xi|\eta) = (\hat 1|\eta\otimes\overline\xi)$.  Thus $m(1\otimes A)m^* = 1$ is again equivalent
to $e(u_0) = u_0$.

As the equivalent conditions of Theorem~\ref{thm:qu_ad_mats_axioms_12} hold, we know that
$e (Q^{-1/2}\otimes (Q^{1/2})^\top) = (Q^{-1/2}\otimes (Q^{1/2})^\top) e$.  As
$(Q^{-1/2}\otimes (Q^{1/2})^\top) u_0 = u_1$, we see that if $e(u_0) = u_0$ then
$e(u_1) = e(Q^{-1/2}\otimes (Q^{1/2})^\top) (u_0) = (Q^{-1/2}\otimes (Q^{1/2})^\top) e (u_0)
= (Q^{-1/2}\otimes (Q^{1/2})^\top) (u_0) = u_1$, and conversely.  Hence 
(\ref{prop:qu_ad_mats_axiom_3:one}), (\ref{prop:qu_ad_mats_axiom_3:two}) and
(\ref{prop:qu_ad_mats_axiom_3:three}) are equivalent, when using $\Psi$.
Given $\Psi' = \Psi'_{1/2,0}$, by Remark~\ref{rem:psi_psi_prime_agree_A}, we know that in fact
$\Psi(A) = \Psi'(A)$, and so the same results hold.

The bijection between $H\otimes \overline H$ and $\mc B(H)$ sends $\eta\otimes\overline\xi$ to
$\theta_{\xi,\eta}$ and so 
\[ u_0 \mapsto \sum_j \theta_{Q^{-1/2}e_j, e_j} = Q^{-1/2}, \qquad
u_1 \mapsto \sum_j \theta_{e_j, Q^{-1/2}e_j} = Q^{-1/2}. \]
Thus (\ref{prop:qu_ad_mats_axiom_3:three}) and (\ref{prop:qu_ad_mats_axiom_3:four}) are equivalent.

For (\ref{prop:qu_ad_mats_axiom_3:five}) consider that
\[ m(\sigma_{-i}\otimes\sigma_{-i/2})\Psi_{0,1/2} (\theta_{\Lambda(a), \Lambda(b)})
= m(\sigma_{-i}\otimes\sigma_{-i/2})(a^*\otimes\sigma_{i/2}(b))
= m(\sigma_{-i}(a^*)\otimes b) = \sigma_{i}(a)^*b. \]
Thus
$m(\sigma_{-i}\otimes\sigma_{-i/2})\Psi_{0,1/2}(A)=1$ is equivalent to $m(1\otimes A)m^*=1$, by
Lemma~\ref{lem:meaning_of_axiom3}.  As, by Theorem~\ref{thm:qu_ad_mats_axioms_12}, we have that
$(\sigma_{-i/2}\otimes\sigma_{-i/2})\Psi_{0,1/2}(A) = \Psi_{0,1/2}(A)$, we see that this condition is
also equivalent to $m(\sigma_{-i/2}\otimes\id)\Psi_{0,1/2}(A) = 1$.  Hence
(\ref{prop:qu_ad_mats_axiom_3:five}) and (\ref{prop:qu_ad_mats_axiom_3:one}) are equivalent.
\end{proof}

Let us quickly adapt this to the case of the axioms (\ref{defn:quan_adj_mat:irrefl}) and 
(\ref{defn:quan_adj_mat:irrefl:other}) of Definition~\ref{defn:adj_op}.

\begin{proposition}\label{prop:qu_ad_mats_axiom_5}
With the conventions of the previous proposition, the following are equivalent:
\begin{enumerate}
\item\label{prop:qu_ad_mats_axiom_5:one}
  $A\in\mc B(L^2(B))$ satisfies axiom (\ref{defn:quan_adj_mat:irrefl}) or, equivalently,
  axiom (\ref{defn:quan_adj_mat:irrefl:other}) of Definition~\ref{defn:adj_op};
\item\label{prop:qu_ad_mats_axiom_5:two}
  the projection $e$, acting on $H\otimes\overline H$, satisfies
  $e (u_0) = 0$ or, equivalently, $e (u_1) = 0$;
\item\label{prop:qu_ad_mats_axiom_5:three}
  the subspace $V\subseteq H\otimes\overline H$ is orthogonal to $u_0$, equivalently, $u_1$;
\item\label{prop:qu_ad_mats_axiom_5:five}
  $m(\sigma_{-i/2}\otimes\id)(e) = 0$ where $m:B\otimes B^\op\rightarrow B$ is the multiplication map.
\end{enumerate}
\end{proposition}
\begin{proof}
We continue with the proof of the previous proposition.  With $A = \sum_i \theta_{\Lambda(a_i), \Lambda(b_i)}$,
from Lemma~\ref{lem:meaning_of_axiom3}, we see that axiom (\ref{defn:quan_adj_mat:irrefl}) is equivalent to
$\sum_i a_i b_i^* = 0$. Looking at equations~\eqref{eq:3} and~\eqref{eq:4} we see that this is equivalent to
$(1\otimes {Q^{1/2}}^\top) e (1\otimes {Q^{-1/2}}^\top)(\hat 1) = 0$, equivalently, $e(u_0)=0$, equivalently,
$u_0$ is in the image of $1-e$.  As $V$ is the image of $e$, equivalently, $u_0$ is orthogonal to $V$.

Now consider axiom (\ref{defn:quan_adj_mat:irrefl:other}).  By similar reasoning, and using calculation
\eqref{eq:6}, this is equivalent again to $e(u_0)=0$.  As $e (Q^{-1/2}\otimes (Q^{1/2})^\top) =
(Q^{-1/2}\otimes (Q^{1/2})^\top) e$, and as $(Q^{-1/2}\otimes (Q^{1/2})^\top) u_0 = u_1$, we see that
$e(u_0) = 0$ if and only if $e(u_1)=0$.  We have shown that conditions (\ref{prop:qu_ad_mats_axiom_5:one}),
(\ref{prop:qu_ad_mats_axiom_5:two}) and (\ref{prop:qu_ad_mats_axiom_5:three}) are equivalent.

For (\ref{prop:qu_ad_mats_axiom_5:five}) we again simply follow the above proof.
\end{proof}

We will address what axioms (\ref{defn:quan_adj_mat:irrefl}) and 
(\ref{defn:quan_adj_mat:irrefl:other}) of Definition~\ref{defn:adj_op} mean for subspaces of $\mc B(H)$ below, see
Proposition~\ref{prop:complements_quan_graph} and the discussion after.

\begin{remark}\label{rem:dont_need_to_use_Q_2}
We continue Remark~\ref{rem:dont_need_to_use_Q}.  In the previous two propositions, the operator $Q$
appears in conditions, and also in the formation of the tensors $u_0, u_1$.  However, again this is only
because $Q$ implements the group $(\sigma_t)$, and so any positive invertible operator $q$ would suffice.
In particular, if $H=L^2(B)$ then we can use $\nabla$ in place of $Q$.
\end{remark}

As an immediate corollary, we can state the known correspondence between Definitions~\ref{defn:adj_op}
and~\ref{defn:quan_graph}.

\begin{corollary}
The bijections $\Psi_{0,1/2}$ and $\Psi'_{1/2,0}$ establish, in the case when $\psi$ is a trace,
a correspondence between quantum adjacency matrices $A$ and quantum graphs $S\subseteq\mc B(H)$,
where $B\subseteq\mc B(H)$.
\end{corollary}
\begin{proof}
Immediate from the correspondence between Theorem~\ref{thm:qu_ad_mats_axioms_12} parts
(\ref{thm:qu_ad_mats_axioms_12:one}) and (\ref{thm:qu_ad_mats_axioms_12:four}), and
Proposition~\ref{prop:qu_ad_mats_axiom_3} parts (\ref{prop:qu_ad_mats_axiom_3:one}) and
(\ref{prop:qu_ad_mats_axiom_3:four}), as when $\psi$ is a trace, $Q=1$ and $\sigma_z=\id$ for all $z$.
\end{proof}

\subsection{Moving to the tracial situation}\label{sec:back_to_qgs}

For general $\psi$, we only obtain a correspondence between Definition~\ref{defn:adj_op} (that of a
quantum adjacency matrix) and a generalisation of Definition~\ref{defn:quan_graph} (the operator bimodule
view of quantum graphs) taking account of the extra conditions of Theorem~\ref{thm:qu_ad_mats_axioms_12}
and Proposition~\ref{prop:qu_ad_mats_axiom_3}.  In this section, we show how to use a further bijection
to come in a full circle back to Definition~\ref{defn:quan_graph}.  We also show some more direct links
between adjacency matrices $A$ and the associated subspaces $S$.

For the following, recall Remark~\ref{rem:dont_need_to_use_Q}.

\begin{theorem}\label{thm:reduce_tracial_case}
Let $B\subseteq\mc B(H)$, and let $q\in\mc B(H)$ be positive and invertible implementing $(\sigma_t)$.
Given a subspace $S\subseteq\mc B(H)$ let $S_0 = q^{1/2}S$, so that $S=q^{-1/2}S_0$
and $S\mapsto S_0$ is a bijection.  Then $S$ is a self-adjoint $B'$-bimodule with $q S q^{-1}=S$ if and only
if the same is true of $S_0$.  Furthermore, $q^{-1/2}\in S$ if and only if $1\in S_0$.
Thus if $S$ is associated to a quantum adjacency matrix, then $S_0$ is a quantum graph with the additional
property that $q S_0 q^{-1} = S_0$, and conversely.
\end{theorem}
\begin{proof}
Firstly, we claim that for any $z\in\mathbb C$, we have that $q^z x q^{-z} \in B'$ for any $x\in B'$.
Indeed, for $b\in B$, we have that $q^{-z} b q^z = \sigma_{iz}(b) \in B$, so $x$ commutes with this, and hence
\[ q^z x q^{-z} b = q^z x q^{-z} b q^z q^{-z} = q^z q^{-z} b q^z x q^{-z} = b q^z x q^{-z}. \]
As $b$ was arbitrary, this indeed shows that $q^z x q^{-z} \in B'$.

Thus, if $S$ is a $B'$-operator bimodule, then for $x,y\in B'$ and $s\in S$, for $s_0 = q^{1/2} s\in S_0$,
we see that
\[ x s_0 y = q^{1/2} q^{-1/2} x q^{1/2} s y \in q^{1/2} S = S_0, \]
as $q^{-1/2} x q^{1/2} \in B'$.  Conversely, if $S_0$ is a $B'$-operator bimodule, then
with $x,y,s,s_0$ as before,
\[ x s y = q^{-1/2} q^{1/2} x q^{-1/2} s_0 y \in q^{-1/2} S_0 = S. \]

Now consider the condition that $q S q^{-1} = S$.  By a functional calculus argument, it follows that also
$q^{1/2} S q^{-1/2} = S$; compare with the proof of Theorem~\ref{thm:qu_ad_mats_axioms_12}.  Thus also
$S_0 = S q^{1/2}$ and so $q^{1/2} S_0 q^{-1/2} = q^{1/2} S = S_0$.  Hence $q S_0 q^{-1} = S_0$.
Furthermore, if also $S$ is self-adjoint, then $S_0^* = (q^{1/2} S)^* = S q^{1/2} = S_0$.
Conversely, if $q S_0 q^{-1} = S_0$ then also $q^{-1} S_0 q = S_0$, so also $q^{-1/2} S_0 q^{1/2} = S_0$.
Thus $q^{1/2} S q^{-1/2} = S_0 q^{-1/2} = q^{-1/2} S_0 = S$, so again $q S q^{-1} = S$.  Similarly, $S_0$
self-adjoint implies that $S$ is self-adjoint.

Clearly $q^{-1/2} \in S$ if and only if $1\in S_0$.
\end{proof}

In particular, we can always apply this result with $q=Q$, and in the special case when $H=L^2(B)$, with
$q=\nabla$.  Let $\psi_0$ be a faithful \emph{trace} on $B$.  Combining Theorem~\ref{thm:qu_ad_mats_axioms_12} and
Proposition~\ref{prop:qu_ad_mats_axiom_3} with the construction just given allows us to do the following.  We
start with a quantum adjacency matrix $A$, with respect to $\psi$, and form the subspace $S$.  Then form $S_0$,
and using $\psi_0$, reverse the process, to form a quantum adjacency matrix $A_0$.  Thus every quantum adjacency
matrix for $\psi$ corresponds (uniquely) to a quantum adjacency matrix $A_0$ for the trace $\psi_0$.

We first of all see what the condition $q S_0 q^{-1}=S_0$ means for the operator $A_0$.  We continue to
write $(\sigma_z)$ for the modular automorphism group of $\psi$.

\begin{lemma}
Let $B\subseteq\mc B(H)$, let $\psi_0$ be a faithful trace on $B$,
and let $S_0\subseteq\mc B(H)$ be a quantum graph associated to the adjacency matrix $A_0\in\mc B(L^2(\psi_0))$.
Furthermore, using that the GNS map $\Lambda_0:B\rightarrow L^2(\psi_0)$ is a linear bijection, let $A_0$
induce a linear map $A_0':B\rightarrow B$.  Then $q S_0q ^{-1}=S_0$ if and only if $A_0' \circ \sigma_{-i}
= \sigma_{-i}\circ A_0'$.
\end{lemma}
\begin{proof}
As in the proof of Theorem~\ref{thm:qu_ad_mats_axioms_12}, that $q S_0 q^{-1}=S_0$ is equivalent to
$e = (\sigma_i\otimes\sigma_i)(e)$.  Supposing for the moment that $A_0 = \theta_{\Lambda_0(a), \Lambda_0(b)}$,
we have that $e = b\otimes a^*$ and that $A_0'(c) = \psi_0(a^*c) b$ for $c\in B$.  Then
$\sigma_i(b)\otimes\sigma_i(a^*)$ is associated to the map on $B$ given by 
\begin{align*} c\mapsto \psi_0(\sigma_i(a^*)c) \sigma_i(b)
&= \psi_0(Q^{-1} a^* Q c) \sigma_i(b) = \psi_0(a^* QcQ^{-1}) \sigma_i(b) \\
&= \psi_0(a^* \sigma_{-i}(c)) \sigma_i(b) = \sigma_i\big( A_0'(\sigma_{-i}(c)) \big).
\end{align*}
By linearity, such a relationship holds in general, and so we conclude that $A_0' = \sigma_i \circ A_0'
\circ \sigma_{-i}$, which is equivalent to the claim.
\end{proof}

\begin{remark}
Suppose that $S, S_0$ are linked as before.  In terms of subspaces of
$H\otimes\overline{H}$ this means that $V$ is associated to $V_0 = (q^{1/2}\otimes\id)(V)$.  As $V$ is invariant
under the operator $q^z \otimes (q^{-z})^\top$, for all $z\in\mathbb C$, also $V_0 =
(q^{1/4} \otimes (q^{1/4})^\top)(V)$, and so forth.

Thus $e_0$ is the orthogonal projection onto $V_0$.  We note that
\[ p = (q^{1/4} \otimes (q^{1/4})^\top )e (q^{-1/4} \otimes (q^{-1/4})^\top) \]
is an idempotent with image $V_0$, here of course letting $e$ act on $H\otimes\overline H$ in the usual way.
Then, in $B\otimes B^\op$, we have that $p = (\sigma_{-i/4}\otimes\sigma_{i/4})(e)$.  So in general, $p\not=e_0$
as $p$ is not self-adjoint.

There is a standard way to convert an idempotent into a projection; we follow \cite[Proposition~3.2.10]{dales}
for example, and just give a sketch.  Setting $a = 1 + (p-p^*)^* (p-p^*)$, we see that $a\geq 1$ so $a$ is
invertible.  Then one can show that $q = pp^*a^{-1}$ is a projection with $pq=q, qp=p$ so that $q$ and $p$ have
the same image.  Hence $e_0 = q$.  This is a concrete formula, but it appears difficult to work with, as in
particular, we have no expression for $a^{-1}$.

Then $A\in\mc B(L^2(\psi))$ is equal to ${\Psi'_{1/2,0}}^{-1}(e)$, and $A_0\in\mc B(L^2(\psi_0))$ is
${\Psi'_{1/2,0}}^{-1}(e_0)$, noting of course that really these are two different maps $\Psi'_{1/2,0}$ as they
have different codomains.  We do not see how to give a concrete formula for $A$ in terms of $A_0$.

However, Remark~\ref{rem:gen_by_A_itself} below does give a direct link between $A$ and $S_0$.
\end{remark}

We seek a way to more directly link the adjacency matrix $A$ with the subspace $S\subseteq\mc B(H)$
(and/or $V\subseteq H\otimes\overline H$).  As $H$ is essentially arbitrary (assuming $B$ is faithfully
represented on $H$) this seems hard, but with a special choice of $H$ we can make progress.

We start by looking further at the opposite algebra $B^\op$.  Define a faithful positive functional
$\psi^\op$ on $B^\op$ by $\psi^\op(a) = \psi(a)$ for $a\in B^\op$.

\begin{lemma}\label{lem:Bop_GNS}
Let $(L^2(B), \Lambda, \pi)$ be the GNS construction for $\psi$ on $B$, and similarly
$(L^2(B^\op), \Lambda^\op, \pi^\op)$ for $\psi^\op$.  The map $w:L^2(B^\op) \rightarrow \overline{L^2(B)};
\Lambda^\op(a) \mapsto \overline{\Lambda(a^*)}$ is a unitary with $w \pi^\op(a) w^* = \pi(a)^\top$ for $a\in B$.
Furthermore, $w J^\op w^* = \overline J$, and $w \nabla^\op w^* = \overline\nabla = \nabla^\top$.
\end{lemma}
\begin{proof}
We have that $(\Lambda^\op(a)|\Lambda^\op(b)) = \psi^\op(ba^*) = \psi(ba^*)$, recalling that the product is
reversed in $B^\op$.  As
\[ (w\Lambda^\op(a)|w\Lambda^\op(b))
= (\overline{\Lambda(a^*)}|\overline{\Lambda(b^*)})
= (\Lambda(b^*)|\Lambda(a^*)) = \psi(ba^*), \]
it follows that $w$ is unitary, and so $w^*\overline{\Lambda(b)} = w^*w \Lambda^\op(b^*) = \Lambda^\op(b^*)$.
Thus
\[ w \pi^\op(a) w^* \overline{\Lambda(b)}
= w \pi^\op(a) \Lambda^\op(b^*) = w \Lambda^\op(b^*a) = \overline{\Lambda(a^*b)}
= \overline{\pi(a^*) \Lambda(b)} = \pi(a)^\top \overline{\Lambda(b)}, \]
as required.

By definition, we have that $\psi^\op(a) = \psi(a) = \Tr(Qa) = \Tr(aQ)$ and so $\psi^\op$ also has density $Q$.
Thus $\sigma_z^\op(a) = Q^{-iz} a Q^{iz} = \sigma_{-z}(a)$, remembering that the multiplication in $B^\op$
is reversed. Thus
\[ w J^\op w^* \overline{\Lambda(a)} = w J^\op\Lambda^\op(a^*)
= w\Lambda^\op(\sigma^\op_{-i/2}(a)) = w\Lambda^\op(\sigma_{i/2}(a))
= \overline{\Lambda(\sigma_{i/2}(a)^*)} = \overline{ J\Lambda(a) }, \]
as claimed.  Similarly,
\[ w \nabla^\op w^*\overline{\Lambda(a)} = w \nabla^\op\Lambda^\op(a^*)
= w \Lambda^\op(Q^{-1} a^* Q) = \overline{ \Lambda(Q a Q^{-1}) }
= \overline{ \nabla\Lambda(a) }, \]
where $\overline\nabla = \nabla^\top$ as $\nabla$ is self-adjoint.
\end{proof}

In the sequel, we shall often suppress $w$, and simply identify $L^2(B^\op)$ with $\overline{L^2(B)}$.
We equip $B\otimes B^\op$ with the faithful functional $\psi\otimes\psi^\op$ and so $L^2(B\otimes B^\op)
\cong L^2(B) \otimes \overline{L^2(B)}$.  As usual, we give $\mc B(L^2(B))$ the functional $\Tr$, and this also
has GNS space $L^2(B) \otimes \overline{L^2(B)}$.

\begin{lemma}
Let $e\in B\otimes B^\op$ be a projection, and let $B\otimes B^\op$ act on $L^2(B) \otimes \overline{L^2(B)}$.
The image of $e$ agrees with the $B'\otimes {B'}^\op$-invariant subspace generated by
$(\Lambda\otimes\Lambda^\op)(e)$.
\end{lemma}
\begin{proof}
Consider $f = a\otimes b\in B\otimes B^\op$ which acts as $a\otimes b^\top$ on $L^2(B) \otimes \overline{L^2(B)}$.
Thus, for $x,y\in B$,
\begin{align*}
f(\Lambda(x)\otimes\overline{\Lambda(y)}) &= \Lambda(ax) \otimes \overline{\Lambda(b^*y)}
= J\sigma_{-i/2}(x)J \Lambda(a) \otimes \overline{ J\sigma_{-i/2}(y)J \Lambda(b^*) } \\
&= \big( J\sigma_{-i/2}(x)J \otimes ( J\sigma_{-i/2}(y)^*J )^\top \big) (\Lambda\otimes\Lambda^\op)(f).
\end{align*}
By linearity, this relationship holds for all $f\in B\otimes B^\op$.
As $B' = JBJ$, taking the linear span over $x,y$, we conclude that the image of $e$ equals
$(B'\otimes (B')^\op) (\Lambda\otimes\Lambda^\op)(e)$, as claimed.
\end{proof}

\begin{proposition}\label{prop:A_generates_S}
Let $B$ act on $L^2(B)$, and let $A\in\mc B(L^2(B))$ be self-adjoint and satisfy axioms
(\ref{defn:quan_adj_mat:idem}) and (\ref{defn:quan_adj_mat:undir}) of Definition~\ref{defn:adj_op}.  Let $A$ be associated to the $B'$-operator
bimodule $S \subseteq \mc B(L^2(B))$, and the $B'\otimes{B'}^\op$-invariant subspace $V\subseteq
L^2(B) \otimes \overline{L^2(B)}$.  Set $A' = \nabla^{-1/2} A$.
Then $S$ is the $B'$-operator bimodule generated by $A'$, and $\Lambda_{\Tr}(A')$ is the projection of
$(\Lambda\otimes\Lambda^\op)(1)$ onto $V$.
\end{proposition}
\begin{proof}
We use $\Psi'_{1/2,0}$ as in Theorem~\ref{thm:qu_ad_mats_axioms_12}, so that $\theta_{\Lambda(a), \Lambda(b)}$
is associated to $\sigma_{i/2}(b) \otimes a^*$.  Notice that
\[ \Lambda_{\Tr}(\nabla^{-1/2} \theta_{\Lambda(a), \Lambda(b)})
= \Lambda_{\Tr}(\theta_{\Lambda(a), \Lambda(\sigma_{i/2}(b))})
= \Lambda(\sigma_{i/2}(b)) \otimes \overline{ \Lambda(a) }
= \Lambda(\sigma_{i/2}(b)) \otimes \Lambda^\op(a^*). \]
By linearity, we see that $\Lambda_{\Tr}(\nabla^{-1/2}A) = (\Lambda\otimes\Lambda^\op)\Psi'_{1/2,0}(A)$.

Use Theorem~\ref{thm:qu_ad_mats_axioms_12} to associate $A$ to the projection $e$.
Thus $\Lambda_{\Tr}(\nabla^{-1/2} A) = (\Lambda\otimes\Lambda^\op)(e)$, and so by the previous lemma,
$S$ is the $B'$-operator bimodule generated by $\nabla^{-1/2} A$.  The first claim follows.
Further, as $e$ is the projection onto $V$, the projection of
$(\Lambda\otimes\Lambda^\op)(1)$ onto $V$ is simply $e (\Lambda\otimes\Lambda^\op)(1) =
(\Lambda\otimes\Lambda^\op)(e)$, and so the second claim follows.
\end{proof}

\begin{remark}\label{rem:gen_by_A_itself}
There is an interesting link with Theorem~\ref{thm:reduce_tracial_case}, which we apply with $q=\nabla$.
Thus $S$ is associated to $S_0 = \nabla^{1/2}S$ which is a ``tracial'' quantum graph.  As $A' = \nabla^{-1/2} A$,
the proposition says that
\[ S_0 = \nabla^{1/2}S =  \nabla^{1/2} \lin \{ xA'y : x,y\in B' \}
= \lin \{ \nabla^{1/2} x \nabla^{-1/2} A y : x,y\in B' \}. \]
As $x\mapsto \nabla^{1/2} x \nabla^{-1/2}$ is a bijection of $B'$, it follows that $S_0$ is the $B'$-operator
bimodule generated by $A$.
\end{remark}

\subsection{Adjacency super-operators to matrices}\label{sec:super-ops}

In this section we show that Definitions~\ref{defn:adj_superop} and~\ref{defn:adj_op} are equivalent.
This was shown in \cite[Section~1.2]{cw} in the case when $B = \mathbb M_n$, using the \emph{Choi matrix}
of a completely positive map.  We shall give an account of how to extend this formalism to general $B$,
which might be of independent interest.  As such, we proceed with a little more generality than strictly needed.
We shall also use the language of Hilbert $C^*$-modules, see \cite{lance} for example, but we shall not
use any deep results from this theory.

In this section, let $C$ be an arbitrary $C^*$-algebra, and let $B,\psi$ be as usual.
In the usual way, regard $B^*$ as a bimodule over $B$.  In particular, for $a\in B$,  we denote by
$\psi a\in B^*$ the functional $b\mapsto \psi(ab)$.  Every $f\in B^*$ is of the form $\psi a$ for some $a\in B$.
This follows, as if not, by Hahn-Banach, there is a non-zero $b\in B^{**} = B$ with $0 = (\psi a)(b) = \psi(ab)$
for all $a$.  Choosing $a=b^*$ shows that $\psi(b^*b)=0$ so $b=0$, contradiction.  Furthermore, notice that
$B\rightarrow B^*, a\mapsto \psi a$ is injective, as $(\psi a)(a^*) = \psi(aa^*)=0$ only when $a=0$, again using
that $\psi$ is faithful.

As $B$ is finite-dimensional, $\mc B(B,C)$ is spanned by finite-rank operators, that is, the span of operators of
the form $\theta_{\psi a, c}$ for $a\in B, c\in C$, where $\theta_{\psi a, c}$ is the map $B\ni b \mapsto
\psi(ab) c \in C$.  In fact, it will be profitable to again introduce the modular automorphism group,
and consider the map
\[ \Psi':\mc B(B,C) \rightarrow C\otimes B^\op; \quad
\theta_{\psi a, c} \mapsto c\otimes \sigma_{-i/2}(a). \]

\begin{theorem}\label{thm:general_c-j}
$\Psi'$ restricts to a bijection between the completely positive maps $B\rightarrow C$, and positive
elements $e\in C\otimes B^\op$.
\end{theorem}

The ``usual'' Choi matrix, compare \cite[Proposition~1.5.12]{bo} associates a CP map $A:\mathbb M_n\rightarrow
C$ with the positive operator $(A(e_{ij})) \in M_n(C)$.  Let us see this in our generalisation.
Set $B=\mathbb M_n$ and $\psi=\Tr$.
We identify $B^\op$ with $\mathbb M_n$ using the matrix transpose, and so our bijection identifies
$A=\theta_{\psi a, c}$ with $u = c\otimes a^\top \in C\otimes\mathbb M_n \cong M_n(C)$.
Then $A(e_{ij}) = \psi(ae_{ij}) c = a_{ji} c = (a^\top)_{ij} c = u_{ij}$, exactly as we expect.

To prove the theorem,
we shall use the language of Hilbert $C^*$-modules, in particular, the KSGNS construction, which we now
briefly recall from \cite[Chapter~5]{lance}.  Given a completely
positive map $A:B\rightarrow C$, we consider the algebraic tensor product $B\odot C$, and define on this a
(in general, degenerate) positive $C$-valued form by
\[ \big( b_1\otimes c_1 \big| b_2\otimes c_2 \big) = c_1^* A(b_1^*b_2) c_2. \]
The separation completion is denoted by $B\otimes_A C$, and we continue to denote by $b\otimes c$ the
equivalence class in $B\otimes_A C$ defined by this simple tensor.  Almost by definition, $B\otimes_A C$
is a (right) Hilbert $C^*$-module over $C$, and there is a $*$-homomorphism $\pi:B\rightarrow
\mc L(B\otimes_A C)$ given by $\pi(b)(b'\otimes c) = bb'\otimes c$.  Finally, there is
$V\in\mc L(C,B\otimes_A C)$ defined by $V(c) = 1\otimes c$, and then
\[ A(b) = V^*\pi(b)V \qquad (b\in B), \]
giving a canonical dilation for the CP map $A$.

\begin{proof}[Proof of Theorem~\ref{thm:general_c-j}]
Suppose that $A:B\rightarrow C$ is completely positive, and set $e=\Psi'(A)$.  Perform the KSGNS construction
to form $E=B\otimes_A C, \pi$ and $V$.  We wish to describe $E$ in terms of $e$.  For the moment
suppose that $A=\theta_{\psi a, c}$, so that $e=c\otimes \sigma_{-i/2}(a)$, and
\begin{align*} \big( b_1\otimes c_1 \big| b_2\otimes c_2 \big)_E &= c_1^* A(b_1^*b_2) c_2
= \psi(ab_1^*b_2) c_1^* c c_2 = \psi(\sigma_{i/2}(b_2) \sigma_{-i/2}(ab_1^*)) c_1^* c c_2 \\
&= \psi(\sigma_{i/2}(b_2) \sigma_{-i/2}(a) \sigma_{i/2}(b_1)^*) c_1^* c c_2 \\
&= (\psi^\op\otimes\id)\big( (\sigma_{i/2}(b_1) \otimes c_1)^* \sigma(e)
   (\sigma_{i/2}(b_2) \otimes c_2) \big),
\end{align*}
where the final product is formed in $B^\op\otimes C$ and $\sigma$ is again the swap map, and $\psi^\op$ is
as in Lemma~\ref{lem:Bop_GNS}.  By linearity, this relationship holds for general $A$.

Let $\mu$ be a state on $C$, with GNS construction $(L^2(C,\mu), \Lambda_\mu)$.  Given
$u=\sum_i \Lambda^\op(b_i)\otimes \Lambda_\mu(c_i) \in L^2(B^\op) \otimes L^2(C,\mu)$, we have
\begin{align*} (u|\sigma(e)u) &= \sum_{i,j} \big(\Lambda^\op(b_i)\otimes \Lambda_\mu(c_i)
   \big| \sigma(e) \Lambda^\op(b_j)\otimes \Lambda_\mu(c_j) \big) \\
&= \sum_{i,j} (\psi^\op\otimes\mu)\big( (b_i\otimes c_i)^* \sigma(e) (b_j\otimes c_j) \big) \\
&= \sum_{i,j} \mu\big( (\sigma_{-i/2}(b_i)\otimes c_i|\sigma_{-i/2}(b_j)\otimes c_j)_E \big) \geq 0.
\end{align*}
As this holds for all $u$, this shows that the image of $\sigma(e)$ in $\mc B(L^2(B^\op)\otimes L^2(C,\mu))$
is positive.  As $\mu$ was arbitrary, this shows that $\sigma(e)$, and hence $e$, is positive.

We now consider the converse, assume that $e = \Psi'(A)$ is positive, and aim to show that $A$ is
completely positive.  We copy the KSGNS construction, but now \emph{define}
a sesquilinear form on $B^\op\otimes C$ by
\[ \big( u_1 \big| u_2 \big) =
(\psi^\op\otimes\id)\big( (\sigma_{i/2}\otimes\id)(u_1)^* \sigma(e) (\sigma_{i/2}\otimes\id)(u_2) \big)
\qquad (u_1,u_2\in B^\op\otimes C). \]
This form is clearly positive, so we may let $E$ be the separation completion (as $B$ is finite-dimensional,
$E$ is in fact a quotient of $B\otimes C$).  

We shall repeatedly move between $B$ and $B^\op$, and it will be helpful to have notation to distinguish these.
We shall write the product in $B^\op$ by juxtaposition, and the product in $B$ by $a\cdot b$, for $a,b\in B$.
We now show that $\pi:B\rightarrow\mc L(E)$ defined by $\pi(b)(b'\otimes c) = b\cdot b'\otimes c$ is well-defined.
Notice that working in $B^\op\otimes C$, we have that $\pi(b)(b'\otimes c) = (b'\otimes c)(b\otimes 1)$, and so
\[ \|\pi(b)(u)\|_E^2 = (\psi^\op\otimes\id)\big(
(\sigma_{i/2}(b)^*\otimes 1)(\sigma_{i/2}\otimes\id)(u)^* \sigma(e)
   (\sigma_{i/2}\otimes\id)(u) (\sigma_{i/2}(b)\otimes 1)
\big). \]
As $\sigma(e)\in B^\op\otimes C$ is positive, there is $f\geq 0$ with $f^2=\sigma(e)$.
Set $x = f (\sigma_{i/2}\otimes\id)(u)$, say $x = \sum_i b_i \otimes c_i \in B^\op\otimes C$.  Then
\begin{align*} \|\pi(b)(u)\|_E^2 &=
\sum_{i,j} \psi^\op\big( \sigma_{i/2}(b)^* b_i^* b_j \sigma_{i/2}(b) \big) c_i^* c_j
= \sum_{i,j} \psi\big( \sigma_{i/2}(b) \cdot b_j \cdot b_i^* \cdot \sigma_{i/2}(b)^* \big) c_i^* c_j \\
&= \sum_{i,j} \psi\big( b_j \cdot b_i^* \cdot \sigma_{i/2}(b)^* \cdot \sigma_{-i/2}(b) \big) c_i^* c_j
= \sum_{i,j} \psi\big( b_j \cdot b_i^* \cdot \sigma_{-i/2}(b^*\cdot b) \big) c_i^* c_j \\
&= \sum_{i,j} \psi\big( \sigma_{i/2}(b_i^* \cdot \sigma_{-i/2}(b^*\cdot b)) \cdot \sigma_{-i/2}(b_j) \big) c_i^* c_j \\
&= \sum_{i,j} \psi\big( \sigma_{-i/2}(b_i)^* \cdot b^*\cdot b \cdot \sigma_{-i/2}(b_j) \big) c_i^* c_j.
\end{align*}
We can now pull out the $b^*\cdot b$ term, at the cost of a $\|b\|^2$ factor, and then reverse to calculation
to conclude that
\[ \|\pi(b)(u)\|_E^2 \leq \|b\|^2 \|u\|_E^2. \]
Thus $\pi$ is well-defined.
A similar calculation shows that $\pi$ is a $*$-homomorphism.

Without loss of generality, we may suppose that $C$ is unital.  Set $u_0 = 1\otimes 1\in E$, so that
\[ (u_0|\pi(b)u_0)_E = (\psi^\op\otimes\id)( \sigma(e)(\sigma_{i/2}(b)\otimes 1) ). \]
We wish to link this to the map $A$.  Suppose for the moment that $e = c\otimes a$ so that $A(a') =
\psi(\sigma_{i/2}(a) \cdot a') c$, and hence
\[ (\psi^\op\otimes\id)( \sigma(e)(\sigma_{i/2}(b)\otimes 1) )
= \psi(\sigma_{i/2}(b)\cdot a) c = \psi(\sigma_{i/2}(a)\cdot b) c = A(b). \]
By linearity, this relationship holds for all $e$, and hence $A(b) = (u_0|\pi(b)u_0)_E$ for all $b\in B$.
We conclude that $A$ is completely positive.
\end{proof}

The first part of the following is \cite[Proposition~2.23]{matsuda}, though we use the ideas developed
above.  It generalises \cite[Proposition~1.7]{cw}, which considers $B = \mathbb M_n$ with $\psi$ a trace.

\begin{theorem}\label{thm:super_op_to_op_new}
Let $A_0:B\rightarrow B$ be a linear map, and let $A:L^2(B)\rightarrow L^2(B)$ be the associated map.
Then $\Psi'(A_0) = \Psi'_{0,1/2}(A)$ in the notation of Section~\ref{sec:equiv}.
The following are equivalent:
\begin{enumerate}
\item\label{thm:super_op_to_op_new:one}
   $A_0$ is completely positive, with $m(A_0\otimes A_0)m^* = A_0$;
\item\label{thm:super_op_to_op_new:two}
   $A$ is real and satisfies axiom (\ref{defn:quan_adj_mat:idem}) of Definition~\ref{defn:adj_op}.
\end{enumerate}
We interpret the axioms of Definition~\ref{defn:adj_op} for $A_0$ in the obvious way.
The following are equivalent:
\begin{enumerate}[resume]
\item\label{thm:super_op_to_op_new:three}
   $A_0$ is completely positive, and satisfies axioms (\ref{defn:quan_adj_mat:idem}) and
   (\ref{defn:quan_adj_mat:undir}).
\item\label{thm:super_op_to_op_new:four}
   $A$ is self-adjoint, satisfying axioms (\ref{defn:quan_adj_mat:idem}) and
   (\ref{defn:quan_adj_mat:undir}).
\end{enumerate}
\end{theorem}
\begin{proof}
We apply Theorem~\ref{thm:general_c-j} with $C=B$, and so consider $\Psi':\mc B(B) \rightarrow B\otimes B^\op$
which maps $A_0=\theta_{\psi a^*,b}$ to $b\otimes\sigma_{i/2}(a)^*$.  The associated $A$ is
$\theta_{\Lambda(a), \Lambda(b)}$, and indeed $b\otimes\sigma_{i/2}(a)^* = \Psi'_{0,1/2}(A)$ as claimed.
As in the proof of Proposition~\ref{prop:what_real_means}
we set $f = (\sigma_{-i/2}\otimes\sigma_{-i/2})(e)$, where $e = \Psi'_{1/2,0}(A)$.
Then $e = \Psi'_{1/2,0}(A) = (\sigma_{i/2}\otimes\sigma_{i/2})\Psi'_{0,1/2}(A)$ and so $f = \Psi'(A_0)$.

We now see that (\ref{thm:super_op_to_op_new:one}) is equivalent to $A$ satisfying axiom
(\ref{defn:quan_adj_mat:idem}) and $\Psi'_{0,1/2}(A)$ being positive.  This is equivalent to $e=e^2$ and
$f\geq 0$, equivalently, $f$ being a projection.  By Proposition~\ref{prop:what_real_means},
this is equivalent to (\ref{thm:super_op_to_op_new:two}).

Then condition (\ref{thm:super_op_to_op_new:three}) is equivalent to $A$ being real, and satisfying
axioms (\ref{defn:quan_adj_mat:idem}) and (\ref{defn:quan_adj_mat:undir}).  As in
Remark~\ref{rem:real_vs_other_axioms} this is equivalent to condition (\ref{thm:super_op_to_op_new:four}).
\end{proof}

\begin{remark}\label{rem:cp_case_in_general}
We make links with \cite[Section~2]{bhinw}.  Here the authors consider the space $B\otimes_\psi B$,
which can be identified with the Hilbert $C^*$-module $L^2(B)\otimes B$, with the usual left action of $B$
on $L^2(B)$.  They define the ``quantum edge indicator'', which with our normalisation conventions, is
$\epsilon=(\id\otimes A_0)m^*(1)$ which agrees with $\Psi_{1,0}(A)$, compare Remark~\ref{rem:existing_bijects}.
From \cite[Proposition~2.3(2)]{bhinw} we see that it is natural to work with $B\otimes B^\op$, and then
\cite[Proposition~2.3(3)]{bhinw} considers the element $(\sigma_{i/2}\otimes\id)(\epsilon) =
\Psi_{1/2,0}(A)$.  This agrees with $\sigma\Psi'_{0,1/2}(A)$.  Under the assumption that
$m(A_0\otimes A_0)m^*=A_0$, we see that $e = \Psi'_{0,1/2}(A)$ is an idempotent, and thus we see that $e$ is
positive if and only if $(\sigma_{i/2}\otimes\id)(\epsilon)$ is self-adjoint.  Theorem~\ref{thm:general_c-j}
shows immediately that $A_0$ is completely positive if and only if $(\sigma_{i/2}\otimes\id)(\epsilon)$ is
self-adjoint, which hence yields a different proof of \cite[Proposition~2.3(3)]{bhinw}.

However, with reference to Proposition~\ref{prop:psiiso}, neither $\Psi_{1/2,0}$ nor $\Psi_{1,0}$
seems to give particularly ``nice'' equivalent properties on $e$ for the other axioms we might wish
$A$ (or $A_0$) to satisfy.
\end{remark}

\section{Constructions and Examples}

In this section, we discuss some examples, explore how quantum channels (for us, UCP maps) give rise to
quantum graphs, and also discuss various ways to construct new quantum graphs from existing ones.

\subsection{Examples}\label{sec:std_egs}

We consider our basic examples, Definitions~\ref{defn:complete_qg} and~\ref{defn:empty_qg}, and see
what the different realisations from Section~\ref{sec:main_equivs} are.  The following is clear: for
the complete quantum graph we obtain the ``maximal'' operator bimodule $\mc B(H)$.

\begin{proposition}\label{prop:equiv_complete_qg}
The complete quantum graph has $A=\theta_{\Lambda(1), \Lambda(1)}$ and this is associated to
$e = \Psi_{0,1/2}(A) = 1\otimes 1$, and to subspaces $H\otimes \overline H$ and $\mc B(H)$.
\end{proposition}

The empty quantum graph given by $A=(mm^*)^{-1}$ is much harder to work with, as ideally we would wish
to express $(mm^*)^{-1}$ as a span of operators of the form $\theta_{\Lambda(a), \Lambda(b)}$, and it
is not immediate how to do this.  However, recall from the discussion before Definition~\ref{defn:empty_qg}
that $mm^*$ is, in particular, an operator which is a central element in $B$.


For the following, recall $u_0$ from Proposition~\ref{prop:qu_ad_mats_axiom_3}, which shows that
axiom (\ref{defn:quan_adj_mat:reflexive}) is equivalent to $u_0\in V$, and hence the characterisation
of $e$ as ``minimal'' is as expected.

\begin{proposition}\label{prop:equiv_empty_qg}
Let $A = (mm^*)^{-1}$ the empty quantum graph adjacency matrix.  This is associated to the operator
bimodule $S = B'Q^{-1/2} B'$, and subspace $V = (B'\otimes (B')^\op)(u_0)$, and hence projection $e$
which is the minimal projection in $B\otimes B^\op$ with $e(u_0)=u_0$.
\end{proposition}
\begin{proof}
Given the equivalences of Theorem~\ref{thm:qu_ad_mats_axioms_12}, we need only show $S = B'Q^{-1/2} B'$.
Furthermore, we may choose which Hilbert space $B$ acts on, so set $H = L^2(B)$.
By Proposition~\ref{prop:A_generates_S}, $S$ is the $B'$-bimodule generated by $A' = \nabla^{-1/2} A$.
There is $x\in B'$ with $mm^* = x$.  Thus
$A = (mm^*)^{-1} = x^{-1}$, and so $S = B' \nabla^{-1/2} x^{-1} B' = B' \nabla^{-1/2} B'$ because $x^{-1}$
is invertible.

For $a\in B$ and $y_1,y_2\in B'$ there are $b_1,b_2\in B$ with $y_i = Jb_iJ$, and so
\begin{align*}
y_1 \nabla^{-1/2} y_2 \Lambda(a) 
&= y_1 \nabla^{-1/2} \Lambda(a \sigma_{-i/2}(b_2^*))
= y_1 \Lambda(\sigma_{i/2}(a) b_2^*) = \Lambda(\sigma_{i/2}(a) b_2^* \sigma_{-i/2}(b_1^*)) \\
&= Q^{-1/2} \Lambda(a Q^{1/2} b_2^* \sigma_{-i/2}(b_1^*)) = Q^{-1/2} JzJ \Lambda(a),
\end{align*}
for some $z\in B$, indeed, $\sigma_{-i/2}(z^*) = Q^{1/2} b_2^* \sigma_{-i/2}(b_1^*)$.  
Thus $S \subseteq Q^{-1/2}B'$, but clearly given an arbitrary $z$, we may set $b_1=1$ and
$b_2^* = Q^{-1/2} \sigma_{-i/2}(z^*)$ to obtain the opposite inclusion.  Thus $S = Q^{-1/2}B'
= B' Q^{-1/2} B'$, as $Q\in B$.
\end{proof}

Let $S_0$ be the tracial quantum graph associated to $S$, given by Theorem~\ref{thm:reduce_tracial_case}.
So $S_0 = Q^{1/2} S = Q^{1/2} B' Q^{-1/2} B' = B'$, as $Q\in B$ commutes with $B'$.  Thus $S_0=B'$ is
the empty quantum graph from Definition~\ref{defn:qg:basic_egs}, as we might expect.

\subsection{From Quantum Channels}\label{sec:from_quan_channels}

We now explore a class of examples (of quantum graphs in the sense of Definition~\ref{defn:quan_graph})
which come from quantum channels, see \cite{cch, duan, dsw, stahlke}.  These examples are amongst the
first which arose, and motivate much of the theory.  We shall generalise the constructions to arbitrary
finite-dimensional $C^*$-algebras, and at the same time streamline them, using the Stinespring representation
theorem directly, rather than passing via Kraus operators.

Quantum Channels are usually considered to be trace-preserving, completely positive maps between
matrix algebras.  From an operator algebraic perspective, this means they are most naturally considered to
be maps between (pre)duals of algebras: namely trace-class operators.  As such, we prefer to work with
the Banach space adjoints, and instead look at unital, completely positive maps.  We shall make comments
below where this leads to slightly different conventions.

\begin{remark}\label{rem:classical_channel_to_graph}
Classically we consider the following situation.  Let $X,Y$ be finite sets, and for each $x\in X$ imagine
we send $x$ down a noisy communication channel, so that $y\in Y$ may be received with probability $p(y|x)
\in [0,1]$.  Thus $\sum_y p(y|x)=1$ for each $x$.  This gives a map $\ell^1(X)\rightarrow\ell^1(Y)$, say
sending a unit basis vector $\delta_x$ to $\sum_y p(y|x) \delta_y$.  The adjoint is a (completely) positive,
unital map $\theta:\ell^\infty(Y)\rightarrow\ell^\infty(X)$ given by
\[ \theta(f)(x) = \sum_y p(y|x) f(y) \qquad (f\in \ell^\infty(Y)). \]
We say that $x_1,x_2\in X$ are ``confusable'' if there is some $y$ with both $p(y|x_1)>0$ and $p(y|x_2)>0$,
that is, non-zero chance that $y$ can be received if either of $x_1$ or $x_2$ is sent.  This establishes
a symmetric, reflexive relation on $X$, that is, a graph structure on $X$.  The associated
$\ell^\infty(X)$-bimodule is
\[ S = \lin \{ e_{x_1,x_2} : \exists\, y\in Y, \ p(y|x_1) p(y|x_2)>0 \} \subseteq \mc B(\ell^2(X)). \]
We shall see shortly how to express this $S$ using $\theta$ directly.
\end{remark}

We shall use a form of the Stinespring representation theorem, which follows \cite[Theorem~3.6, Chapter~IV]{tak1}.
For $C^*$-algebras $B,C$ with $B\subseteq\mc B(H)$, given a UCP map $\theta:C\rightarrow B$ there is a Hilbert space
$K$, a non-degenerate $*$-homomorphism $\pi:C\rightarrow\mc B(K)$, a normal unital $*$-homomorphism $\rho:
\theta(C)'\rightarrow\pi(C)'$, and an isometric linear operator $V:H\rightarrow K$ with
\[ \theta(c) = V^*\pi(c)V \quad (c\in C), \qquad
\rho(x)V=Vx, \quad (x\in \theta(C)'). \]
As $\theta(C)\subseteq B\subseteq\mc B(H)$, we can, and will, always restrict $\rho$ to $B'$.  If we have
the minimality condition that $K = \overline{\lin}\{ \pi(c)V\xi : c\in C, \xi\in H \}$, then this construction
is unique up to unitary conjugation in the obvious way.  When (as is our case) $B,H$ are finite-dimensional,
a minimal $K$ is also finite-dimensional.

The following will be immediate from more general considerations below in Section~\ref{sec:hm},
see Remark~\ref{rem:from_channel_as_pullback}, but we give the proof, as it will in part motivate the
more general constructions to be given later.

\begin{theorem}\label{thm:quan_chan_to_quan_graph}
Let $C$ be some $C^*$-algebra, let $B$ as usual be a finite-dimensional $C^*$-algebra with $B\subseteq\mc B(H)$
faithfully represented.  Given a UCP map $\theta:C\rightarrow B$, construct $V,\rho,\pi,K$ as above, and define
$\hat\theta:\pi(C)'\rightarrow\mc B(H)$ by $\hat\theta(y) = V^*yV$.  Then $S = \hat\theta(\pi(C)')$ is a
quantum graph over $B$ in the sense of Definition~\ref{defn:quan_graph}.  This definition of $S$ is independent
of the choice of $V,\rho,\pi,K$ (so in particular, $K$ is not required to be minimal).
\end{theorem}
\begin{proof}
As $V$ is an isometry, $V^*V=1$, and so $\hat\theta$ is UCP, from which it immediately follows that
$S$ is a self-adjoint unital subspace of $\mc B(H)$.  Given $x_1,x_2\in B'$, as $\rho(x_1),\rho(x_2)\in\pi(C)'$,
we find that
\[ x_1 V^* y_1 V x_2 = V^* \rho(x_1) y \rho(x_2) V \in S \qquad (y\in \pi(C)'). \]
Thus $S$ is a $B'$-operator bimodule.

To show that $S$ does not depend upon the Stinespring dilation chosen, we shall suppose that $K$ is minimal,
pick another dilation $\tilde V, \tilde\rho, \tilde\pi, \tilde K$, and show that the resulting $\tilde S$
agrees with $S$.  We claim that we may define $u:K \rightarrow \tilde K$ by $\pi(a)V\xi\mapsto \tilde\pi(a)
\tilde V\xi$.  As
\[ (\tilde\pi(a) \tilde V\xi|\tilde\pi(b) \tilde V\eta) = (\xi|\tilde V^* \tilde\pi(a^*b) \tilde V\eta)
= \theta(a^*b) = (\pi(a) V \xi | \pi(b)  V \eta), \]
(or more actually, repeating this calculation with linear spans)
it follows that $u$ is well-defined, and an isometry.  By minimality, $u$ is defined on all of $K$ (but $u$
may not be surjective).  Notice that $uV = \tilde V$, and it is easily verified that $\tilde\pi(c) u = u \pi(c)$ for
$c\in C$, and hence also $u^* \tilde\pi(c) = \pi(c) u^*$, as $\pi,\tilde\pi$ are $*$-homomorphisms.
Given $y\in\tilde\pi(C)'$ we find that
\[ \pi(c) u^*yu = u^* \tilde\pi(c) y u = u^* y \tilde\pi(c) u = u^*yu \pi(c)
\qquad (c\in C). \]
Thus $u^*yu \in \pi(C)'$ and so $S \ni V^* u^*yu V = \tilde V^* y \tilde V$, and hence $\tilde S \subseteq S$.
For the converse, consider now $x\in\pi(C)'$ and set $y=uxu^*$.  Similarly, we show that $y\in\tilde\pi(C)'$, and
as $u$ is an isometry, $u^*u=1$ so $u^*\tilde V = V$, hence $\tilde S \ni \tilde V^* y  \tilde V = V^* x V$,
so that $S\subseteq\tilde S$.
\end{proof}

\begin{remark}
Continuing Remark~\ref{rem:classical_channel_to_graph}, we may construct a Stinespring dilation for $\theta$
by setting $K = \ell^2(X)\otimes\ell^2(Y), V:e_x \mapsto \sum_y p(y|x)^{1/2} e_x\otimes e_y, \pi(a)=1\otimes a$
and $\rho(x) = x\otimes 1$ for $x\in \ell^\infty(X)'=\ell^\infty(X)$.  Then $\pi(\ell^\infty(Y))'
= \mc B(\ell^2(X)) \otimes \ell^\infty(Y)$, and so the associated (quantum) graph is
\begin{align*}
V^*( \mc B(\ell^2(X)) \otimes \ell^\infty(Y) )V
&= \lin\big\{ V^*(e_{x_1,x_2}\otimes e_{y,y}) V : x_1,x_2\in X, y\in Y \big\} \\
&= \lin\big\{ \sqrt{ p(y|x_1) p(y|x_2) } e_{x_1,x_2} : x_1,x_2\in X, y\in Y \big\},
\end{align*}
which agrees with the previously defined $S$.
\end{remark}

\begin{remark}\label{rem:quan_chan_eg_Kraus_ops}
Suppose that $C=\mc B(H')$ for some $H'$, and that $B=\mc B(H)$.  Any trace-preserving completely positive map
$\theta_0:B\rightarrow C$ is given by \emph{Kraus operators} $E_k:H\rightarrow H'$ with
\[ \theta_0(x) = \sum_k E_k x E_k^* \qquad (x\in B=\mc B(H)), \]
with $\sum_k E_k^*E_k=1$.  Indeed, the linear map $\theta:C\rightarrow B$ which satisfies
$\Tr(\theta(x)y) = \Tr(x\theta_0(y))$ for $x\in C, y\in B$ is completely positive, and unital if
only if $\theta_0$ is trace-preserving, see Proposition~\ref{prop:ac_adjoints} below.
So $\theta$ has a Stinespring dilation $\theta(x) = V^*\pi(x)V$
for some $\pi:C\rightarrow\mc B(K)$.  As $C = \mc B(H')$, necessarily $K\cong H'\otimes K'$ with
$\pi(x)=x\otimes 1$.  Pick an orthonormal basis $(e_k)_{k=1}^n$ for $K'$, and choose $E_k:H\rightarrow H'$ with
\[ V(\xi) = \sum_k E_k(\xi)\otimes e_k \qquad (x\in H). \]
Then $V^*\pi(x)V = V^*(x\otimes 1)V = \sum_k E_k^* x E_k$, so $\theta_0$ has the stated form, and $V^*V=1$
so $\sum_k E_k^*E_k=1$.

Then $\pi(C)' = \mathbb C\otimes\mathbb B(K') \cong \mathbb C \otimes \mathbb M_n$ and
\[ \hat\theta(e_{k,l}) = V^*(1\otimes e_{k,l})V = E_k^* E_l. \]
Thus $S = \lin\{ E_k^* E_l \}$, in agreement of \cite[Section~IV]{dsw}.  By contrast, \cite{stahlke}
considers a different convention, which we discuss below after Proposition~\ref{prop:complements_quan_graph}.
\end{remark}

When $B=\mc B(H)$, it is shown in \cite[Lemma~2]{duan} that every quantum graph over $B$ arises from
some UCP map.  We have been unable to decide if this remains true for general $B$.

\subsection{The no-loops condition}\label{sec:no_loops}

In this section, we compare the ``loops'' condition $m(A\otimes 1)m^*=1$ with the ``no loops'' condition
$m(A\otimes 1)m^*=0$.  We also see what this means for the operator bimodule picture of quantum graphs.
Into such considerations also comes the non-commutative generalisation of taking a graph complement, which
we consider first.

Recall that, when we view $B$ as a direct sum of matrix algebras, the operator $mm^*$ is a direct sum of scalar
multiples of the identity (equivalently, $mm^*$ is in the centre of $B$) and is invertible, compare the
discussion around Definition~\ref{defn:empty_qg}.  We state the following for axioms
(\ref{defn:quan_adj_mat:reflexive}) and (\ref{defn:quan_adj_mat:irrefl}) of Definition~\ref{defn:adj_op}, but we know that in the
presence of the other axioms, these are equivalent to axioms (\ref{defn:quan_adj_mat:reflexive:other}) and
(\ref{defn:quan_adj_mat:irrefl:other}) respectively.

\begin{proposition}\label{prop:to_no_loops}
Let $A\in\mc B(L^2(B))$ be self-adjoint and define $A' = A - (mm^*)^{-1}$ which is also self-adjoint.
The following are equivalent:
\begin{enumerate}
\item\label{prop:to_no_loops:one}
  $A$ satisfies axioms (\ref{defn:quan_adj_mat:idem}) and
  (\ref{defn:quan_adj_mat:undir}) of Definition~\ref{defn:adj_op}, and (\ref{defn:quan_adj_mat:reflexive}), namely $m(A\otimes 1)m^*=1$;
\item\label{prop:to_no_loops:two}
  $A'$ satisfies axioms (\ref{defn:quan_adj_mat:idem}) and
  (\ref{defn:quan_adj_mat:undir}), and (\ref{defn:quan_adj_mat:irrefl}), namely $m(A'\otimes 1)m^*=0$.
\end{enumerate}
\end{proposition}
\begin{proof}
We know that $(mm^*)^{-1}$ is itself a quantum adjacency matrix, see after Definition~\ref{defn:empty_qg}.
As the axiom (\ref{defn:quan_adj_mat:undir}) is linear in $A$, clearly $A$ satisfies axiom
(\ref{defn:quan_adj_mat:undir}) if and only if $A'$ does.  Similarly, $m(A\otimes 1)m^*=1$ if and only if
$m(A'\otimes 1)m^*=0$; and $m(1\otimes A)m^*=1$ if and only if $m(1\otimes A')m^*=0$.  Recall that in
the presence of the other axioms, $A$ satisfies $m(A\otimes 1)m^*=1$ if and only if $m(1\otimes A)m^*=1$.

It hence remains to consider axiom (\ref{defn:quan_adj_mat:idem}).  As $B$ is a direct sum of $n$ matrix
algebras, also $L^2(B)$ is the direct sum of Hilbert spaces $(H_k)_{k=1}^n$, and $m, m^*$ and hence $(mm^*)^{-1}$
all respect this decomposition.  Let $\xi\in H_k$, so $m^*(\xi) \in H_k\otimes H_k$.  There is a scalar
$\delta_k>0$ with $(mm^*)^{-1}$ equal to $\delta_k^2$ on $H_k$.  Hence
\begin{align*}
m(A'\otimes A')m^*(\xi) &= m(A\otimes A)m^*(\xi) - m((mm^*)^{-1}\otimes A)m^*(\xi)
  \\ &\qquad\qquad- m(A\otimes (mm^*)^{-1})m^*(\xi) + m((mm^*)^{-1}\otimes (mm^*)^{-1})m^*(\xi) \\
&= A(\xi) - \delta_k^2 m(1\otimes A)m^*(\xi) - \delta_k^2 m(A\otimes 1)m^*(\xi)
  + \delta_k^2 \xi
\end{align*}
as $mm^*(\xi) = \delta_k^{-2} \xi$.  Thus, if (\ref{prop:to_no_loops:one}) holds,
then also $m(1\otimes A)m^*=1$, so $m(A'\otimes A')m^*(\xi) = A'(\xi)$ for each $\xi\in H_k$, for each $k$.
Thus (\ref{prop:to_no_loops:two}) holds.

Conversely, let (\ref{prop:to_no_loops:two}) hold.  By Proposition~\ref{prop:qu_ad_mats_axiom_5}, also
$m(1\otimes A')m^*=0$.  Using that $A = A' + (mm^*)^{-1}$ an analogous calculation to the above shows that
\[ m(A\otimes A)m^*(\xi)
= A'(\xi) + \delta_k^2 m(1\otimes A')m^*(\xi) + \delta_k^2 m(A'\otimes 1)m^*(\xi)
  + \delta_k^2 \xi, \]
and so this equals $A'(\xi) + \delta_k^2 \xi$, for $\xi\in H_k$.  As $k$ was arbitrary, $m(A\otimes A)m^* = A$,
and so (\ref{prop:to_no_loops:one}) holds.
\end{proof}

The previous proposition shows that we can move between ``all loops'' and ``no loops''.  We next look at the
notion of a ``complement'', which classically would take an ``all loops'' graph to a ``no loops'' graph.  We
will discuss an intermediate process, which preserves ``all loops'', respectively, ``no loops'', below.

\begin{proposition}
Let $A\in\mc B(L^2(B))$ be self-adjoint and define $A'' = \theta_{\Lambda(1),\Lambda(1)} - A$ which is also
self-adjoint.  The following are equivalent:
\begin{enumerate}
\item $A$ satisfies axioms (\ref{defn:quan_adj_mat:idem}) and
  (\ref{defn:quan_adj_mat:undir}) of Definition~\ref{defn:adj_op}.
\item $A''$ satisfies axioms (\ref{defn:quan_adj_mat:idem}) and
  (\ref{defn:quan_adj_mat:undir}).
\end{enumerate}
Furthermore, $A$ satisfies axiom (\ref{defn:quan_adj_mat:reflexive}) if and only if $A''$ satisfies axiom
(\ref{defn:quan_adj_mat:irrefl}); and $A$ satisfies axiom (\ref{defn:quan_adj_mat:irrefl}) if and only if $A''$
satisfies axiom (\ref{defn:quan_adj_mat:reflexive}).
\end{proposition}
\begin{proof}
As remarked after Definition~\ref{defn:complete_qg}, $A_0=\theta_{\Lambda(1),\Lambda(1)}$ is itself a quantum
adjacency matrix.  By linearity of the axioms (\ref{defn:quan_adj_mat:undir}) through
(\ref{defn:quan_adj_mat:irrefl:other}), we need only consider axiom (\ref{defn:quan_adj_mat:idem}).
Now observe that
\[ m(A''\otimes A'')m^* = m(A_0\otimes A_0)m^* - m(A_0\otimes A)m^* - m(A\otimes A_0)m^*
  + m(A\otimes A)m^*, \]
where we know that $m(A_0\otimes A_0)m^* = A_0$.  For any $a,b\in A$,
\[ m(A_0\otimes\theta_{\Lambda(a),\Lambda(b)})m^*
= m(\theta_{\Lambda(1),\Lambda(1)}\otimes\theta_{\Lambda(a),\Lambda(b)})m^*
= \theta_{\Lambda(1a), \Lambda(1b)} = \theta_{\Lambda(a), \Lambda(b)}, \]
so by linearity, $m(A_0\otimes A)m^* = A$, and similarly $m(A\otimes A_0)m^*=A$.  Thus
\[ m(A''\otimes A'')m^* = A_0 - 2 A + m(A\otimes A)m^*, \]
and from this it follows that $m(A''\otimes A'')m^* = A'' = A_0-A$ if and only if $m(A\otimes A)m^*=A$.
\end{proof}

We can finally define suitable notions of ``complement''.

\begin{corollary}\label{corr:adj_mat_complement}
Let $A\in\mc B(L^2(B))$ and define $A_c = \theta_{\Lambda(1),\Lambda(1)} + (mm^*)^{-1} - A$.
Then $A$ is a quantum adjacency matrix if and only if $A_c$ is.
\end{corollary}
\begin{proof}
We use $A' = A-(mm^*)^{-1}$ and $A'' = \theta_{\Lambda(1),\Lambda(1)}-A$ from above, and notice that
then $A_c = \theta_{\Lambda(1),\Lambda(1)} - A' = (A')''$.  The claim follows from the previous
propositions.
\end{proof}

\begin{corollary}\label{corr:adj_mat_complement:irrefl}
Let $A\in\mc B(L^2(B))$ and define $A_{nc} = \theta_{\Lambda(1),\Lambda(1)} - (mm^*)^{-1} - A$.
Then $A$ is a quantum adjacency matrix, but satisfying axiom (\ref{defn:quan_adj_mat:irrefl}) of Definition~\ref{defn:adj_op} in place
of (\ref{defn:quan_adj_mat:reflexive}), if and only if the same is true of $A_{nc}$.
\end{corollary}
\begin{proof}
Set $A_0 = A + (mm^*)^{-1}$ so $A_{nc} = A_0''$.  Again, the claim follows immediately from the
previous propositions.
\end{proof}

\begin{remark}
The definition of the complement considered in Corollary~\ref{corr:adj_mat_complement} is noted in
\cite[Remark~3.6]{bcehpsw}, and in somewhat more detail in \cite[Propositions~2.26, 2.28]{matsuda},
both in the case when $\psi$ is a $\delta$-form.  Notice that implicit in the proof of 
\cite[Propositions~2.26]{matsuda} is a consideration that $m(A\otimes 1)m^*=1$ implies
(in the presence of the other axioms) that also $m(1\otimes A)m^*=1$, which we also make essential
use of.
\end{remark}

We now look at quantum graphs in the sense of Definition~\ref{defn:quan_graph}, that is, quantum relations
over our algebra $B\subseteq\mc B(H)$ with additional properties.
In the following, $A$ is self-adjoint satisfying axioms (\ref{defn:quan_adj_mat:idem}) and
(\ref{defn:quan_adj_mat:undir}) of Definition~\ref{defn:adj_op}, and satisfies other axiom(s) as appropriate.

\begin{proposition}\label{prop:complements_quan_graph}
Let $B\subseteq\mc B(H)$, and let the quantum adjacency matrix $A$ correspond to the $B'$-operator bimodule
$S \subseteq\mc B(H)$.  Give $\mc B(H)$ the positive functional $\Tr$ which allows us to form the orthogonal
complement to $S$, say $S^\perp$.  Then:
\begin{enumerate}
\item\label{prop:complements_quan_graph:one}
  $A'=A-(mm^*)^{-1}$ corresponds to $S' = S \cap (B'Q^{-1/2}B')^\perp$;
\item\label{prop:complements_quan_graph:two}
  $A''= \theta_{\Lambda(1), \Lambda(1)}-A$ corresponds to $S'' = S^\perp$;
\item\label{prop:complements_quan_graph:three}
  $A_c = \theta_{\Lambda(1), \Lambda(1)} + (mm^*)^{-1} - A$ corresponds to
  $S_c = S^\perp + B'Q^{-1/2}B'$;
\item\label{prop:complements_quan_graph:four}
   $A_{nc} = \theta_{\Lambda(1), \Lambda(1)} - (mm^*)^{-1} - A$
  corresponds to $S_{nc} = S^\perp \cap (B'Q^{-1/2}B')^\perp$.
\end{enumerate}
\end{proposition}
\begin{proof}
Using Theorem~\ref{thm:qu_ad_mats_axioms_12} we associate $A$ with a projection $e$, and $S$ is the image
of $e$ (once we identify $\mc B(H)$ with its GNS space $H\otimes\overline H$).  Let $(mm^*)^{-1}$
correspond to projection $e_{\min}$, which from Proposition~\ref{prop:equiv_empty_qg} is the minimal projection
fixing $u_0$, and corresponds to the bimodule $S_{\min} = B'Q^{-1/2}B'$.

Consider (\ref{prop:complements_quan_graph:one}), where $A$ is assumed to satisfy axiom 
(\ref{defn:quan_adj_mat:reflexive}), so $e(u_0)=u_0$ by Proposition~\ref{prop:qu_ad_mats_axiom_3}, and
hence $e_{\min} \leq e$.  By linearity, $A'$ corresponds to $e - e_{\min}$ which has image the 
relative orthogonal complement of $S_{\min}$ in $S$, namely $S \cap (B'Q^{-1/2}B')^\perp$, as claimed.

Now consider (\ref{prop:complements_quan_graph:two}).  As $\theta_{\Lambda(1),\Lambda(1)}$ corresponds
to the projection $1$ and the subspace $\mc B(H)$, it is immediate that $S''$ corresponds to $S^\perp$.
Then (\ref{prop:complements_quan_graph:three}) and (\ref{prop:complements_quan_graph:four})
follow in a similar way.
\end{proof}

To be precise, $x\in S^\perp$ if and only if $\Tr(y^*x)=0$ for each $y\in S$.  It is easy to verify directly
that when $S$ is a $B'$-operator bimodule, so is $S^\perp$, and that $S$ is self-adjoint if and only if
$S^\perp$ is.  When $\psi$ is a trace, the ``all loops'' condition corresponds to $S$ being unital, and this
is equivalent to $\Tr(x)=0$ for all $x\in S^\perp$, or that $S^\perp$ is ``trace-free''.  This trace-free
condition was suggested in \cite[Section~III]{stahlke} from motivation from quantum information theory, and was explored further in \cite[Section~2.2]{kim}.
Indeed, Remark~\ref{rem:classical_channel_to_graph} showed how to construct the ``confusability graph'' from
a channel, but \cite{stahlke} argues that it is perhaps more useful to look at the ``distinguishability graph''
which indeed classically corresponds to $S\mapsto S^\perp$ at the level of operator bimodules.

\begin{remark}\label{rem:compl_S0_bad}
This is one place where Theorem~\ref{thm:reduce_tracial_case} does not interact as we might wish.
Recall that we associate $S$ to the ``tracial'' quantum graph $S_0 = Q^{1/2}S$.  A simple calculation shows that
$(Q^{1/2}S)^\perp = Q^{-1/2} S^\perp$.  Hence
\[ (S')_0 = Q^{1/2} S' = Q^{1/2} S \cap Q^{1/2}(Q^{-1/2}B')^\perp
= S_0 \cap Q(B')^\perp, \]
and not $S_0 \cap (B')^\perp$ as we might hope.  Similarly, $(S'')_0 = Q^{1/2} S^\perp = (Q^{-1/2}S)^\perp
= (Q^{-1}S_0)^\perp = Q S_0^\perp$, and $(S_c)_0 = Q^{1/2} S^\perp + B' = Q S_0^\perp + B'$, and finally
$(S_{nc})_0 = Q^{1/2}S^\perp \cap Q(B')^\perp = Q S_0^\perp \cap Q(B')^\perp$.  See Remark~\ref{rem:compl_S0_bad_ce}
below for an example suggesting that we cannot improve this.

However, we do note that if $QSQ^{-1}=S$ then the same is true of $S^\perp$, as given $x\in S^\perp, y\in S$, we
compute that $\Tr(y^* QxQ^{-1}) = \Tr(Q^{-1}y^*Q x) = \Tr((QyQ^{-1})^*x) = 0$ as $QyQ^{-1}\in S$.  So
$S\mapsto S^\perp$ preserves all the properties of Theorem~\ref{thm:qu_ad_mats_axioms_12}.
\end{remark}

\subsection{Small examples}\label{sec:small_examples}

A careful classification of quantum adjacency matrices over $B=\mathbb M_2$ is made in \cite[Section~3]{matsuda}.
Let us see how one part of our work intersects with the non-tracial case, \cite[Section~3.3]{matsuda}.
Using Theorem~\ref{thm:qu_ad_mats_axioms_12}, we in particular wish to study subspaces $S\subseteq\mc B(H)$
with $QSQ^{-1}\subseteq S$, or equivalently $V\subseteq H\otimes\overline H$ with $(Q\otimes(Q^{-1})^\top)(V)
\subseteq V$.  With this in mind, we have the following lemma.

\begin{lemma}\label{lem:inv_under_pos}
Let $H$ be a finite-dimensional Hilbert space, let $T\in\mc B(H)^+$ be invertible, and let $V\subseteq H$ be
a subspace.  Then $T(V)\subseteq V$ if and only if $V$ has a basis of eigenvectors of $T$.
\end{lemma}
\begin{proof}
We need only show the ``only if'' clause.
From Lemma~\ref{lem:proj_commuting_op} we know that $T(V)=V$, and so also $(tT)^n(V)=V$ for any
$t>0$ and $n\in\mathbb N$.  At $T$ is self-adjoint, we can write $T = \sum_i t_i P_i$ where $\{t_i\}$ are the
eigenvalues of $T$, and each $P_i$ is the projection onto the eigenspace associated to $t_i$.  Thus the $\{P_i\}$
are mutually orthogonal, and $\sum_i P_i=1$.  Suppose $t_1 > t_2 > \cdots$, and let $\xi\in V$.  Then
\[ (t_1^{-1}T)^n(\xi) = P_1(\xi) + \sum_{i\geq 2} (t_1^{-1}t_i)^n P_i(\xi), \]
which converges to $P_1(\xi)$ as $n\rightarrow\infty$.  Thus $P_1(\xi)\in V$.  Now repeat the argument,
applying $(t_2^{-1}T)^n$ to $\xi - P_1(\xi)$, to conclude that also $P_2(\xi)\in V$.  Continue, and
conclude that $P_i(\xi)\in V$ for each $i$.

Thus $P_i(V)\subseteq V$ for each $i$, and so $V$ is the orthogonal direct sum of the subspaces $P_i(V)$
(some of which may be $\{0\}$).  The result follows.
\end{proof}

Given $Q$, we can diagonalise, and so find a unitary $u\in B$ and a (component-wise) diagonal $Q_d\in B$
with $u Q u^* = Q_d$.  If $S$ satisfies the conditions of Theorem~\ref{thm:qu_ad_mats_axioms_12} for
$Q_d$, then $T = u^*Su$ satisfies the conditions for $Q$.  As $u\in B$, it is easy to see that this
procedure correctly interacts with Theorem~\ref{thm:reduce_tracial_case}; it also interacts smoothly
with taking complements etc., Proposition~\ref{prop:complements_quan_graph}.  Finally, the notion of
``equivalence'' is compatible with the notion of automorphism we explore in Section~\ref{sec:auts_quan_graphs},
and with the obvious notion of ``isomorphism'' between two quantum graphs, as mentioned elsewhere in
the literature.  Thus, there is no loss of generality in working with $Q_d$.

When $B=\mathbb M_2$, say acting on $\mathbb C^2$, we have that $B'=\mathbb C$, and any $Q$ is similar to
a Powers density
\[ Q_q = \frac{1}{1+q^2} \begin{pmatrix} 1 & 0 \\ 0 & q^2 \end{pmatrix}, \]
with $0<q\leq 1$, compare \cite[Section~3.3]{matsuda}.  We will not consider the tracial case, so assume
$q\not=1$.  With Lemma~\ref{lem:inv_under_pos} in mind, the eigenspaces to consider are
\[ \lin\Big\{ \begin{pmatrix} 1 & 0 \\ 0 & 0 \end{pmatrix}, 
\begin{pmatrix} 0 & 0 \\ 0 & 1 \end{pmatrix} \Big\}, \quad
\mathbb C \begin{pmatrix} 0 & 1 \\ 0 & 0 \end{pmatrix}, \quad
\mathbb C \begin{pmatrix} 0 & 0 \\ 1 & 0 \end{pmatrix}. \]
As $S$ should also be self-adjoint, there are just four quantum graphs over $\mathbb M_2$ with $Q_d S Q_d^{-1}=S$
and $Q_d^{-1/2}\in S$ (this from Proposition~\ref{prop:qu_ad_mats_axiom_3}), namely
\begin{gather*}
S_1 = \mathbb C \begin{pmatrix} 1 & 0 \\ 0 & 1/q \end{pmatrix}, \qquad
S_3 = \mathbb M_2, \\
S_{2A} = \lin\Big\{ \begin{pmatrix} 1 & 0 \\ 0 & 1/q \end{pmatrix}, \begin{pmatrix} 0 & 1 \\ 0 & 0 \end{pmatrix},
\begin{pmatrix} 0 & 0 \\ 1 & 0 \end{pmatrix} \Big\}, \quad
S_{2B} = \lin\Big\{ \begin{pmatrix} 1 & 0 \\ 0 & 0 \end{pmatrix},
\begin{pmatrix} 0 & 0 \\ 0 & 1 \end{pmatrix} \Big\}.
\end{gather*}
Here $S_1$ is the empty quantum graph, $S_3$ the complete quantum graph, and $S_{2A}, S_{2B}$ lie strictly
between these, and are incomparable.  A simple computation shows that taking the complement,
$S\mapsto S_c$ from Proposition~\ref{prop:complements_quan_graph}, interchanges $S_1,S_3$ and
$S_{2A}, S_{2B}$.  If we apply Theorem~\ref{thm:reduce_tracial_case}, then $(S_1)_0$ is the span of
the identity, $(S_{2A})_0$ is the span of the identity and the off-diagonal matrices, and $S_{2B}, S_3$
remain unchanged.  Thus, in this special case, Theorem~\ref{thm:reduce_tracial_case} interacts as we might
hope with the complement.

\begin{remark}\label{rem:compl_S0_bad_ce}
We present an example where Theorem~\ref{thm:reduce_tracial_case} really does not interact with the
complement.  Let $B=\mathbb M_2\oplus\mathbb M_2$ acting on $H=\mathbb C^2\oplus\mathbb C^2$, so that $B'
\cong \mathbb C\oplus\mathbb C$.  Pick $0<t<1/2$ and set
\[ Q_1 = \begin{pmatrix} t & 0 \\ 0 & 1-t \end{pmatrix}, \quad
Q_2 = \begin{pmatrix} 1-t & 0 \\ 0 & t \end{pmatrix}, \quad Q=Q_1\oplus Q_2. \]
We identify $\mc B(H)$ with $2\times 2$ matrices of operators in $\mathbb B(\mathbb C^2)\cong\mathbb M_2$.
Set
\[ x_0 = \begin{pmatrix} 0 & 1 \\ 1 & 0 \end{pmatrix}\in\mathbb M_2, \quad
x = \begin{pmatrix} 0 & x_0 \\ 0 & 0 \end{pmatrix} \in \mc B(H). \]
Finally let $T = \lin\{ x, x^* \}$, so $T$ is a self-adjoint, $B'$-bimodule.  Then
\[ Q_1x_0 = \begin{pmatrix} 0 & t \\ 1-t & 0 \end{pmatrix} \not\in\mathbb C x_0 \quad\implies\quad
Qx = \begin{pmatrix} 0 & Q_1 x_0 \\ 0 & 0 \end{pmatrix}, \]
while
\[ Q_1 x_0Q_2^{-1} = \begin{pmatrix} 0 & t \\ 1-t & 0 \end{pmatrix} Q_2^{-1}
= \begin{pmatrix} 0 & 1 \\ 1 & 0 \end{pmatrix} = x_0 \quad\implies\quad
QxQ^{-1} = \begin{pmatrix} 0 & Q_1 x_0Q_2^{-1} \\ 0 & 0 \end{pmatrix} \in T, \]
from which it follows that $QTQ^{-1}=T$.

By the final part of Remark~\ref{rem:compl_S0_bad}, $S_0 = T^\perp$ also satisfies the required conditions,
and so by Theorem~\ref{thm:reduce_tracial_case}, $S = Q^{-1/2} S_0$ satisfies all the conditions of
Theorem~\ref{thm:qu_ad_mats_axioms_12}.  Towards a contradiction, suppose that $(S_c)_0 = (S_0)^\perp
+ B'$.  By Remark~\ref{rem:compl_S0_bad}, always $(S_c)_0 = Q(S_0)^\perp + B'$, so we obtain that
$(S_0)^\perp + B' = Q(S_0)^\perp + B'$, that is, $T + B' = QT + B'$, that is, $QT \subseteq T+B'$.
However, given the form of $Qx$ computed above, it is clear that $Qx \not\in T + B'$, giving the required
contradiction.
\end{remark}

\section{Homomorphisms}\label{sec:hm}

In this section, we look at various notions of ``homomorphism'' between quantum graph.  We note that there
have been notions of a ``quantum homomorphism'' between classical graphs (see \cite[Section~4]{ortiz}
or \cite[Theorem~4.14, Remark~5.5]{mrv} for example) which can further be extended to quantum homomorphisms
between quantum graphs.  These notions are linked to the quantum automorphisms we discuss in
Section~\ref{sec:qauts_quan_graphs} below, and 
are discussed in detail in \cite[Section~4.4]{grom}, \cite[Section~2.3]{matsuda}, for example.
In contrast, in this section we look at what might be thought of as ``classical homomorphisms'' between
(possibly) quantum graphs.  Such notions seem to be less firmly established than other definitions made in
this paper, so we compare and contrast them, without offering a definitive definition.

We start with two ideas from Weaver in \cite{weaver2}; again we work with quantum graphs in the sense of Definition~\ref{defn:quan_graph}.  For exactly what we mean by ``Kraus form'' see the discussion below.
It is not at all clear that the following definition is independent of the choice of Kraus form, but this
follows from \cite{weaver2}, or our, different, argument to follow.

\begin{definition}[{See \cite[Theorem~7.4, Theorem~8.2]{weaver2}}]
Let $B_i\subseteq\mc B(H_i)$ for $i=1,2$ for finite-dimensional Hilbert space $H_i$.  Let $\theta:B_2
\rightarrow B_1$ be a UCP map with Kraus form $\theta(x) = \sum_{i=1}^n b_i^* x b_i$ for some operators
$b_i\in\mc B(H_1, H_2)$.  Let $S_1\subseteq\mc B(H_1)$ be a quantum  graph.  The \emph{pushfoward} of
$S_1$ along $\theta$, denoted $\overrightarrow{S_1}$ is the $B_2'$-operator bimodule generated by
\[ \{ b_i x b_j^* : x\in S_1 \} \subseteq \mc B(H_2). \]
Let $S_2\subseteq\mc B(H_2)$ be a quantum  graph.  The \emph{pullback} of $S_2$ along $\theta$ is,
denoted by $\overleftarrow{S_2}$ is the $B_1'$-operator bimodule generated by
\[ \{ b_i^* x b_j : x\in S_2 \} \subseteq \mc B(H_1). \]
\end{definition}

Again here we have translated from trace-preserving CP maps to UCP maps, equivalently, have ``reversed
the arrows'', compare with Remark~\ref{rem:classical_channel_to_graph} above.  We notice that the pullback
is always unital if $S_2$ is unital, but the pushfoward may not be, in accordance with what happens in
the classical situation.  While we stated this for quantum graphs, obviously the definition works for any
quantum relation.  The reader who would like some motivation can see Remark~\ref{rem:quan_homo_to_graph_homo} below.

The paper \cite{weaver2} develops an ``intrinsic'' way to characterise quantum graphs (making use of results
from \cite{weaver1}), and then shows that the pushfoward and pullback only depend on $\theta$
(and not the Kraus form chosen) when recast using this intrinsic characterisation of $S_1$ and $S_2$.
We shall show how to use Stinespring dilations (and a duality argument) to give a different approach to this fact.  A biproduct of our construction is that in the above definition of the pullback, it suffices to take the linear span of the set given: it is automatically operator bimodule over the relevant commutant.  We also give a link between pushforwards and pullbacks via a duality argument.

For most of this section, we work with possibly infinite-dimensional operator bimodules, as this is not much more work.  Let us recall how to construct the Kraus form of a normal CP map between arbitrary von Neumann algebras.
This is surely known, but we are not aware of a canonical reference.  Let $M,N$ be von Neumann algebras
acting on Hilbert spaces $H_M, H_N$ respectively.  We start by observing that the usual
Stinespring dilation of a normal CP map $\theta:M\rightarrow N$ between von Neumann algebras, say $V,\rho,\pi,K$
automatically has that $\pi$ is normal.  Consider now an arbitrary dilation (so we do not assume minimality).
By the structure theorem for normal $*$-homomorphisms between
von Neumann algebras (for example, \cite[Theorem~5.5, Chapter~IV]{tak1}) there is an index set $I$ and an isometry
$W:K\rightarrow H_M\otimes\ell^2(I)$ with $W^*(x\otimes 1)W = \pi(x)$ for $x\in N$.
Furthermore, the projection $WW^*$ is in $M'\vnten\mc B(\ell^2(I))$.  Then
\[ \theta(x) = V^*\pi(x)V = V^*W^*(x\otimes 1)WV \qquad (x\in M), \]
and there is a family of operators $(b_i)\subseteq\mc B(H_N, H_M)$ with $WV\xi = \sum_i b_i\xi\otimes\delta_i$
and hence $\theta(x) = \sum_i b_i^* x b_i$ gives us a Kraus representation.

So far, it is not clear how the property that $\theta$ maps into $N$ (and not just $\mc B(H_N)$) is
reflected in the Kraus representation.

\begin{lemma}\label{lem:rho_always}
Let $\theta:M\rightarrow \mc B(H_N)$ be a normal CP map with some Stinespring dilation $V,\pi,K$, meaning that
$\pi:M\rightarrow\mc B(K)$ is a normal unital $*$-homomorphism, $V:H_N\rightarrow K$ an operator,
and $\theta(x) = V^*\pi(x)V$ for $x\in M$; no minimality assumption is made.  The following are equivalent:
\begin{enumerate}
\item $\theta$ maps into $N$;
\item there is a normal unital $*$-homomorphism $\rho:N'\rightarrow \pi(M)'$ with $Vx'=\rho(x')V$ for $x'\in N'$.  \end{enumerate}
\end{lemma}
\begin{proof}
Suppose first that $\theta$ maps into $N$.
Let $V_0,\pi_0,K_0,\rho_0$ be the minimal Stinespring dilation for $\theta$.  We may define $u:K_0\rightarrow
K$ by $u\pi_0(x)V_0\xi = \pi(x)V\xi$ for $x\in M, \xi\in H_N$.  As
\[ ( \pi(x)V\xi | \pi(y)V\eta ) = (\xi|\theta(x^*y)\eta)
= ( \pi_0(x)V_0\xi | \pi_0(y)V_0\eta ) \qquad (x,y\in M, \xi,\eta\in H_N), \]
it follows that $u$ is well-defined, an isometry, and by minimality, extends by linearity and continuity
to all of $K_0$.  By construction, $u\pi_0(x) = \pi(x)u$ for $x\in M$, and $uV_0 = V$.
Define $\rho(x') = u \rho_0(x') u^*$ which is a normal $*$-homomorphism, as $u^*u=1$.  It follows that
\[ \rho(x')V = u \rho_0(x') u^* (uV_0) = u \rho_0(x') V_0 = u V_0 x' = Vx'
\qquad (x'\in N'). \]
Furthermore, for $x'\in N', x\in M$, we calculate that $\rho(x')\pi(x) = u \rho_0(x') u^*\pi(x)
= u \rho_0(x') \pi_0(x) u^* = u \pi_0(x) \rho_0(x') u^* = \pi(x) u \rho_0(x') u^* = \pi(x)\rho(x')$.
Thus $\rho$ maps into $\pi(M)'$.

Conversely, if we have such a $\rho$ then for $x'\in N', x\in M$, we see that $\theta(x)x' = V^*\pi(x)Vx'
= V^*\pi(x)\rho(x')V = V^*\rho(x')\pi(x)V = x' V^*\pi(x)V = x'\theta(x)$.  Thus $\theta(x)\in N''=N$,
as claimed.
\end{proof}

As there is a bijection between dilations $V,\pi,K$ with $K=H_M\otimes\ell^2(I)$ and $\pi(x)=x\otimes 1$,
and Kraus forms $\theta(x) = \sum_{i\in I} b_i^* x b_i$, this lemma can be translated into a (slightly
cumbersome) condition on a Kraus form.

Weaver shows in \cite[Theorem~2.7]{weaver1} that the definition of a Quantum Relation, Definition~\ref{defn:quan_rel},
is independent of the ambient Hilbert space $H$ upon which our von Neumann algebra $M$ acts.  We will study
the ideas of this proof, and separate them out into a number of useful constructions.  Let $M\subseteq\mc B(H_M)$
be a von Neumann algebra, let $S\subseteq\mc B(H_M)$ be a quantum relation on $M$, and let $H$ be some other
Hilbert space.  Let $\pi:M\rightarrow\mc B(H_M\otimes H); x\mapsto x\otimes1$ and define
\[ S_\pi = S\vnten\mc B(H), \]
the weak$^*$-closed linear span of $\{x\otimes a:x\in S,a\in\mc B(H) \}$ in $\mc B(H_M\otimes H)$.  It is routine
that $\pi(M)' = M'\vnten\mc B(H)$, and so $S_\pi$ is a quantum relation on $\pi(M)$.  

\begin{lemma}
Every quantum relation on $\pi(M)$ is of the form $S_\pi$ for some quantum relation $S$ on $M$, and the
map $S\mapsto S_\pi$ is an order-preserving bijection, and preserves the properties of $S$ being self-adjoint,
respectively, $S$ being unital.
\end{lemma}
\begin{proof}
The only non-obvious thing to check is that every quantum relation on $\pi(M)$ is of the form $S_\pi$.
Let $T$ be a weak$^*$-closed operator bimodule for $\pi(M)' = M'\vnten\mc B(H)$.  For $x\in T$, consider
\[ (1\otimes\theta_{\xi_1,\eta_1}) x (1\otimes\theta_{\xi_2,\eta_2})
= (\id\otimes\omega_{\xi_1,\eta_2})(x) \otimes \theta_{\xi_2,\eta_1} \in T
\qquad (\xi_1,\xi_2,\eta_1,\eta_2\in H). \]
Let $S = \{ (\id\otimes\omega)(x) : x\in T, \omega\in\mc B(H)_* \}$ a weak$^*$-closed subspace of $\mc B(H_M)$ which is
seen to be an $M'$-bimodule.  By linearity, and that $T$ is weak$^*$-closed, it follows that $y\otimes a\in T$
for each $y\in S,a\in\mc B(H)$.  Furthermore, we can approximate $x\in T$ in the weak$^*$-topology by
multiplying on the left and right by $1\otimes\theta$ for suitable finite-rank operators $\theta$, and so $x$
is in $S\vnten\mc B(H)$.  Thus $T = S\vnten\mc B(H) = S_\pi$, as required.
\end{proof}

Now let $\theta:M\rightarrow\mc B(K)$ be a normal unital $*$-homomorphism, not assumed injective.  Again by the
structure theorem for such homomorphism, \cite[Theorem~5.5, Chapter~IV]{tak1}, there is some
$H$ and an isometry $u:K\rightarrow H_M\otimes H$ such that $u^*(x\otimes 1)u = \theta(x)$ for $x\in M$,
and with $uu^* \in M'\vnten\mc B(H)$; equivalently, $(x\otimes 1)u = u\theta(x)$ for $x\in M$.  Define
\[ S_\theta = u^*S_\pi u = \{ u^*\alpha u : \alpha \in S\vnten\mc B(H) \} \subseteq\mc B(K), \]
where $\pi(x)=x\otimes 1$ as before.  Notice that the definition of $S_\theta$ generalises the definition of
$S_\pi$.

\begin{lemma}\label{lem:quan_graph_cut_down}
The definition of $S_\theta$ is independent of the choice of $u,H$, and $S_\theta$ is a quantum relation
on $\theta(M)$.  The map $S\mapsto S_\theta$ is order preserving, and preserves the properties of $S$ being
self-adjoint, respectively, $S$ being unital.
\end{lemma}
\begin{proof}
We show independence of $u,H$.  Let $u_1:K\rightarrow H_M\otimes H_1$ be an isometry with
$u_1\theta(x) = (x\otimes 1)u_1$ for each $x\in M$.  Let $S_1 = u_1^* (S\vnten\mc B(H_1)) u_1$.

Set $H_0 = H\oplus H_1$ and let $\iota:H\rightarrow H_0$ be the inclusion.  For $x\otimes a \in S\otimes
\mc B(H_0)$ the algebraic tensor product, clearly $x\otimes \iota^* a \iota \in S\otimes\mc B(H)$.
Thus by weak$^*$-continuity, $((1\otimes\iota)u)^* (S\vnten\mc B(H_0)) (1\otimes\iota)u \subseteq S_\theta$.
However, the same reasoning shows that $(1\otimes\iota) (S\vnten\mc B(H)) (1\otimes\iota^*) \subseteq
S\vnten\mc B(H_0)$ and hence $((1\otimes\iota)u)^* (S\vnten\mc B(H_0)) (1\otimes\iota)u = S_\theta$.
We conclude that we can identify $u$ with $(1\otimes\iota)u$ without changing $S_\theta$, and similarly
for $u_1$.  So consider $u,u_1$ as maps $K \rightarrow H_M\otimes H_0$.

With these identifications, still $u\theta(x) = (x\otimes 1)u$ and similarly for $u_1$, and hence
$uu_1^* \in (M\otimes 1)' = M'\vnten\mc B(H_0)$.  As $S\vnten\mc B(H_0)$ is an $M'\vnten\mc B(H_0)$-operator
bimodule, we see that for $\alpha\in S\vnten\mc B(H_0)$ also $\beta = u_1 u^* \alpha uu_1^* \in
S\vnten\mc B(H_0)$.  Thus $S_\theta \ni u^* \alpha u = u_1^* \beta u_1 \in S_1$, and so $S_\theta
\subseteq S_1$.  By symmetry, also $S_1 \subseteq S_\theta$, and so we have equality, as claimed.

Define $\rho:\theta(M)'\rightarrow\mc B(H_M\otimes H)$ by $\rho(x') = ux'u^*$.  As in the proof of
Lemma~\ref{lem:rho_always}, $\rho$ is a $*$-homomorphism mapping into $(M\otimes 1)' = M'\vnten\mc B(H)$.
So, for $x_1',x_2'\in \theta(M)'$ and $\alpha\in S\vnten\mc B(H)$,
\[ x_1' u^* \alpha u x_2' = u^* \rho(x_1') \alpha \rho(x_2') u \in u^* (S\vnten\mc B(H)) u = S_\theta, \]
using the bimodule property of $S\vnten\mc B(H)$.  Hence $S_\theta$ is a $\theta(M)'$-operator
bimodule.  

To show that $S_\theta$ is weak$^*$-closed, we used the Krein-Smulian Theorem, \cite[Theorem~12.1]{conway}
for example, which tells us that it suffices to show that the (norm closed) unit ball of $S_\theta$ is
weak$^*$-closed.  Towards this, let $(u^*\alpha_i u)$ be a bounded net in $S_\theta$ converging weak$^*$
to $x \in \mc B(K)$.  There seems to be no reason why $(\alpha_i)$ is a bounded net, but for each $i$,
set $\beta_i = uu^* \alpha_i uu^* \in\mc B(H_M\otimes H)$.  As $uu^*\in M'\vnten\mc B(H)$, it follows that
each $\beta_i \in S\vnten\mc B(H)$.  As $u$ is an isometry, $u^* \beta_i u = u^*\alpha_i u$.  Of course,
$\|\beta_i\| \leq \|u^*\alpha_i u\|$, and so $(\beta_i)$ is a bounded net.  By moving to a subnet if necessary,
we may suppose that $\beta_i\rightarrow\beta \in S\vnten\mc B(H)$, as $S\vnten\mc B(H)$ is weak$^*$-closed.
Then $u^*\beta_i u \rightarrow u^*\beta u$ and so $x = u^*\beta u\in S_\theta$, as required.

The remaining claims are clear.
\end{proof}

The following shows how to compute the pullback from a Stinespring dilation directly (not using the Kraus
form) and shows that even the Kraus form definition is already an $N'$-operator bimodule.

\begin{theorem}\label{thm:pullback_from_stinespring}
Let $\theta:M\rightarrow N$ be a normal CP map with Kraus form $\theta(x) = \sum_{i\in I} b_i^* x b_i$.
Let $S\subseteq\mc B(H_M)$ be a quantum relation over $M$, and as before, define the pullback of $S$ along $\theta$
to be $\overleftarrow{S}$, the weak$^*$-closed $N'$-operator bimodule generated by elements $b_i^*xb_j$ for
$x\in S, i,j\in I$.  Then:
\begin{enumerate}
\item\label{thm:pullback_from_stinespring:one}
  already the weak$^*$-closed subspace generated by elements $b_i^*xb_j$ for
  $x\in S, i,j\in I$ is an $N'$-operator bimodule;
\item\label{thm:pullback_from_stinespring:two}
  if $V,\pi,K$ is any Stinespring dilation of $\theta$, then $\overleftarrow{S}$ is the weak$^*$-closure of
  $V^* S_\pi V$, where $S_\pi$ is as defined above.
\end{enumerate}
\end{theorem}
\begin{proof}
We first show how to express $\overleftarrow{S}$ using a special Stinespring dilation.  For the moment, suppose
that $V,\pi,K=H_M\otimes\ell^2(I)$ comes from the Kraus representation, so $V\xi = \sum_i b_i\xi\otimes\delta_i$
and $\pi(x) = x\otimes 1$.  Then for each $i,j$,
\[ V^*(x\otimes e_{ij})V(\xi) = V^*(xb_j(\xi)\otimes\delta_i) = b_i^* x b_j(\xi)
\qquad (x\in S, \xi\in H_M), \]
so that $V^*(x\otimes e_{ij})V = b_i^* x b_j$ for $x\in S$.  Denote by $X$ the weak$^*$-closure of
$\lin\{ V^*(x\otimes a)V : x\in S,a\in\mc B(\ell^2(I)) \}$.  It follows that $\overleftarrow{S}
\subseteq X$, but as $\lin\{ e_{ij} \}$ is weak$^*$ dense in $\mc B(\ell^2(I))$, also $X \subseteq
\overleftarrow{S}$, and hence we have equality.  It follows that $\overleftarrow{S}$ equals the
weak$^*$-closure of $V^*S_\pi V$.

Now let $\rho:N'\rightarrow \pi(M)'=M'\vnten\mc B(\ell^2(I))$ be as in Lemma~\ref{lem:rho_always}.
Then, for $x'_1,x_2'\in N', y\in S_\pi$,
\[ x_2'V^* y Vx_1' = V^* \rho(x_2') y \rho(x_1') V \in V^* S_\pi V, \]
by the operator bimodule property of $S_\pi$.  Thus $V^*S_\pi V$ is an $N'$-operator bimodule, and thus
so is its weak$^*$-closure.  We have hence shown claim (\ref{thm:pullback_from_stinespring:one}).

Now let $V_0, \pi_0, K_0$ be the minimal Stinespring dilation for $\theta$.  There is an isometry
$u:K_0\rightarrow K$ with $uV_0 = V$ and $\pi(x)u = u\pi_0(x)$ for $x\in M$, compare the proof of 
Lemma~\ref{lem:rho_always}.  Notice then that $\pi_0(x) = u^*(x\otimes 1)u$ for each $x\in M$, so by
definition, $S_{\pi_0} = u^* S_\pi u$, and hence $V^* S_\pi V = V_0^* u^* S_\pi u V = V_0^* S_{\pi_0} V_0$.

Finally, suppose that $V,\pi,K$ is any Stinespring dilation for $\theta$.  Again, there is an isometry
$u:K_0\rightarrow K$ with $uV_0 = V$ and $\pi(x)u = u\pi_0(x)$ for $x\in M$.  For a suitable Hilbert space
$K'$, we can find an isometry $v:K\rightarrow H_M\otimes K'$ with $v\pi(x) = (x\otimes 1)v$.  Thus
$S_\pi = v^*(S\vnten\mc B(K')) v$, and as $\pi_0(x) = u^*\pi(x)u = u^*v^*(x\otimes 1)vu$, also
$S_{\pi_0} = u^*v^* (S\vnten\mc B(K'))vu = u^* S_\pi u$.  Hence $V_0^* S_{\pi_0} V_0 = V_0^* u^* S_\pi u V_0
= V^* S_\pi V$.  Thus (\ref{thm:pullback_from_stinespring:two}) is shown.
\end{proof}

\begin{remark}\label{rem:from_channel_as_pullback}
Let us compare this with the quantum graph we constructed in Theorem~\ref{thm:quan_chan_to_quan_graph}.
There we started with $\theta:C\rightarrow B$ a UCP map and constructed a quantum graph over $B$ which is
$S = V^* \pi(C)' V$.  Here $V,\pi,K$ is a Stinespring dilation of $\theta$.  As this is independent of the
choice, we can choose $K = H_C\otimes K_0$ for a suitable $K_0$,
and $\pi(x)=x\otimes 1$, so that $\pi(C)' = C'\vnten \mc B(K_0) = T_\pi$, where $T=C'$ the trivial quantum
graph over $C$.  Thus $S$ is actually just the pullback of the trivial quantum graph $C'$ for the UCP map $\theta$.
\end{remark}

\begin{corollary}\label{corr:pullback_fd}
Let $B_1,B_2$ be finite-dimensional $C^*$-algebras.
Let $\theta:B_2\rightarrow B_1$ be a UCP map, and let $S$ be a quantum graph over $B_2$.  Then
$\overleftarrow{S}$ is a quantum graph over $B_1$.  If $V,\pi,K$ is any Stinespring dilation of $\theta$,
then $\overleftarrow{S} = V^* S_\pi V$.
\end{corollary}
\begin{proof}
In the finite-dimensional setting, there is no need to take weak$^*$-closures, and so
$\overleftarrow{S} = V^* S_\pi V$, for any quantum relation $S$.  If $S$ is self-adjoint, clearly also
$V^* S_\pi V$ is.  As $\theta$ is assumed to be unital, $V^*V=1$, and so if $S$ is unital, so is $S_\pi$,
and hence so is $V^* S_\pi V$.
\end{proof}

While these results show that $\overleftarrow{S}$ is independent of the choice of Kraus representative of $\theta:M\rightarrow N$, the definitions still seems to depend upon the choice of embedding $M\subseteq\mc B(H_M)$.  As shown in \cite{weaver2}, the definition is actually independent, which we now show using our methods.

\begin{proposition}\label{prop:pullback_indep_rep}
Let $\theta:M\rightarrow N$ and $S$ be as in Theorem~\ref{thm:pullback_from_stinespring}.  Let $K_M$ be a (possibly) different Hilbert space with $M\subseteq\mc B(K_M)$, and let $S_K\subseteq\mc B(K_M)$ be the quantum relation corresponding to $S$.  Then $\overleftarrow{S} = \overleftarrow{S_K}$.
\end{proposition}
\begin{proof}
Let $\phi:M\rightarrow\mc B(K_M)$ be the formal inclusion, an injective normal $*$-homomorphism.  Then $S_K = S_\phi$ in the sense of Lemma~\ref{lem:quan_graph_cut_down}.  Indeed, choose $H$ and an isometry $u:K_M \rightarrow H_M \otimes H$ with $(x\otimes 1)u = u\phi(x)$ for $x\in M$, and then $S_\phi = u^*(S\vnten\mc B(H))u$.

By Theorem~\ref{thm:pullback_from_stinespring}, if $V,\pi,K$ is any Stinespring dilation of $\theta$, then $\overleftarrow{S}$ is the weak$^*$-closure of $V^*S_\pi V$.  By the discussion above, we can choose $K = H_M \otimes K'$ and $\pi(x)=x\otimes 1$ for $x\in M$.  Then $S_\pi = S\vnten\mc B(K')$.  We now consider $\theta':\phi(M)\rightarrow N$ given by $\theta' = \theta\circ\phi^{-1}$.  It is natural to consider $\pi' = \pi\circ\phi^{-1}$; we find a dilation for $\pi'$.  Let $\xi_0\in H$ be some unit vector, and define $\iota:H_M\otimes K' \rightarrow H_M\otimes H\otimes K'$ by $\xi\otimes\xi' \mapsto \xi\otimes\xi_0\otimes\xi'$.
Then let $U$ be the composition
\[ H_M\otimes K' \xrightarrow{\iota} H_M\otimes H\otimes K'
\xrightarrow{u^*\otimes 1} K_M \otimes K'. \]
Then for $x\in M$, as $uu^*\in M'\vnten\mc B(H)$,
\begin{align*}
(\phi(x)\otimes 1) U
&= (u^*(x\otimes 1)u\otimes 1) (u^*\otimes 1) \iota
= (u^*\otimes 1)(x\otimes 1\otimes 1)(uu^*\otimes 1)\iota \\
&= (u^*uu^*\otimes 1)(x\otimes 1\otimes 1)\iota
= (u^*\otimes 1)(x\otimes 1\otimes 1)\iota \\
&= (u^*\otimes 1)\iota (x\otimes 1)
= U(x\otimes 1) = U \pi'(\phi(x)).
\end{align*}
Hence $(y\otimes 1)U = U \pi'(y)$ for $y\in\phi(M)$, so we have a dilation for $\pi'$, and so $V,\pi',K$ is a dilation for $\theta'$.  Hence $\overleftarrow{S_K} = \overleftarrow{S_\phi}$ is the weak$^*$-closure of
\begin{align*}
V^* (S_\phi)_{\pi'} V &= V^* U^* (S_\phi\vnten\mc B(K')) UV
= V^*U^*(u^*\otimes 1) (S\vnten\mc B(H)\vnten\mc B(K')) (u\otimes 1)UV \\
&= V^*\iota^*(uu^*\otimes 1) (S\vnten\mc B(H)\vnten\mc B(K')) (uu^*\otimes 1)\iota V \\
& \subseteq V^*\iota^* (S\vnten\mc B(H)\vnten\mc B(K')) \iota V
\end{align*}
because $uu^*\in M'\vnten\mc B(H)$ and $S\vnten\mc B(H)$ is a $M'\vnten\mc B(H)$-bimodule.  As $\iota^* z \iota = (\id\otimes\omega_{\xi_0}\otimes\id)z$ for $z\in\mc B(H_M\otimes H\otimes K')$, we find that
\begin{align*}
V^*\iota^* (S\vnten\mc B(H)\vnten\mc B(K')) \iota V
= V^* (S\vnten\mc B(K')) V = V^*S_\pi V \subseteq \overleftarrow{S}.
\end{align*}
We have shown that $\overleftarrow{S_K} \subseteq \overleftarrow{S}$, and reversing the roles of $H_M$ and $K_M$ shows the reverse inclusion, hence giving equality, as claimed.
\end{proof}

While these results gives a rather pleasing form for the pullback at the level of operator bimodules, we
have not been able to formulate a simple result for the associated projections and quantum adjacency matrices.

We now consider pushforwards.  Let $\theta:M\rightarrow N$ be a normal CP map with Kraus form $\theta(x) = \sum_{i\in I} b_i^* x b_i$, and now let $S \subseteq \mc B(H_N)$ be a quantum relation over $N$.  In the infinite-dimensional setting, we define $\overrightarrow{S}$ to be the weak$^*$-closed $M'$-bimodule generated by $\{ b_i x b_j^* :  x\in S \}$, which we shall denote by $\text{w$^*$-$M'$-bimod } \{ b_i x b_j^* :  x\in S \}$.  Let $U:H_N \rightarrow H_M \otimes \ell^2(I)$ be $\xi \mapsto \sum_i b_i(\xi)\otimes\delta_i$, so that $\theta(x) = U^*(x\otimes 1)U$ for $x\in M$.  Then
\[ (\id\otimes\omega_{\delta_i, \delta_j})(UxU^*) = b_i x b_j^* \qquad (x\in S, i,j\in I). \]
Hence $\overrightarrow{S} = \text{w$^*$-$M'$-bimod } (\id\otimes\mc B(\ell^2(I))_*)(USU^*)$, an alternative form which is useful in calculations.  We now proceed to show that the definition of $\overrightarrow{S}$ is independent of the various choices made.

\begin{proposition}\label{prop:pushforward_indep_Kraus}
Let $\theta:M\rightarrow N$ be a normal CP map, and let $S \subseteq \mc B(H_N)$ be a quantum relation over $N$.  The definition of $\overrightarrow{S}$ is independent of the choice of Kraus form for $\theta$.
\end{proposition}
\begin{proof}
Notice that we are free to enlarge the index set $I$, by setting ``extra'' Kraus operators $b_i$ to be zero.  Thus, without loss of generality, two different dilations of $\theta$ may be assumed to be given by $U:H_N \rightarrow H_M \otimes H$ and $V:H_N \rightarrow H_M \otimes H$ for the same Hilbert space $H$.  Let $K$ be the closed linear span of $\{ (x\otimes 1)U\xi : \xi\in H_N, x\in M \} \subseteq H_M\otimes H$.  We may define $u:K\rightarrow H_M\otimes H$ by $u: (x\otimes 1)U\xi \mapsto (x\otimes 1)V\xi$, and linearity and continuity.  Indeed, we calculate that
\[ ( (x\otimes 1)V\xi | (y\otimes 1)V\eta ) = (\xi | \theta(x^*y) \eta)
= ( (x\otimes 1)U\xi | (y\otimes 1)U\eta ) \qquad (\xi,\eta\in H_N, x,y\in M). \]
This calculation shows that $u$ is well-defined and an isometry.  Extend $u$ to a partial isomerty defined on all of $H_M\otimes H$.  By construction, for $x\in M$ we have that $u(x\otimes 1) = (x\otimes 1)u$ on $K$, and as $K$ is invariant for $M\otimes 1$, it follows that $u$ commutes with $x\otimes 1$ on all of $H_M\otimes H$, so that $u \in M'\vnten\mc B(H)$.  By definition, $uU = V$, so
\[ (\id\otimes\mc B(H)_*)(VSV^*) = (\id\otimes\mc B(H)_*)(uUSU^*u^*)
\subseteq \text{w$^*$-$M'$-bimod } (\id\otimes\mc B(H)_*)(USU^*). \]
Reversing the roles of $U$ and $V$ gives the other inclusion, hence equality, which shows that using $U$ or $V$ gives the same definition of $\overrightarrow{S}$.
\end{proof}

We now show the analogue of Proposition~\ref{prop:pullback_indep_rep} for pushforwards.

\begin{proposition}\label{prop:pushforward_indep_rep}
Let $\theta:M\rightarrow N$ be a normal CP map, and let $S \subseteq \mc B(H_N)$ be a quantum relation over $N$.  Let $N\subseteq\mc B(K_N)$ be a normal faithful representation, and let $S_K \subseteq \mc B(K_N)$ be the quantum relation associated to $S$.  Then $\overrightarrow{S} = \overrightarrow{S_K}$.
\end{proposition}
\begin{proof}
Again, let $\phi:N\rightarrow\mc B(K_N)$ with dilation $\phi(x) = u^*(x\otimes 1)u$ for $x\in N$, with $u:K_N\rightarrow H_N\otimes K$ an isometry with $uu^*\in N'\vnten\mc B(K)$.  Then $S_K = S_\phi = u^*(S\vnten\mc B(K))u$.  Let $\theta(x) = U^*(x\otimes 1)U$ for some $U:H_N\rightarrow H_M\otimes H$.  Then $\phi\circ\theta:M\rightarrow \phi(N)$ has dilation $\phi(\theta(x)) = u^*(U^*\otimes 1)(x\otimes 1\otimes 1)(U\otimes 1)u$, so by Proposition~\ref{prop:pushforward_indep_Kraus},
\begin{align*}
\overrightarrow{S_K} &= \overrightarrow{S_\phi} = \text{w$^*$-$M'$-bimod } (\id\otimes\mc B(H\otimes K)_*)((U\otimes 1)u S_\phi u^*(U^*\otimes 1)) \\
&= \text{w$^*$-$M'$-bimod } (\id\otimes\mc B(H\otimes K)_*)((U\otimes 1)u u^*(S\vnten\mc B(K))u u^*(U^*\otimes 1)) \\
&\subseteq \text{w$^*$-$M'$-bimod } (\id\otimes\mc B(H\otimes K)_*)((U\otimes 1)(S\vnten\mc B(K))(U^*\otimes 1)) \\
&= \text{w$^*$-$M'$-bimod } (\id\otimes\mc B(H)_*)(USU^*) = \overrightarrow{S}
\end{align*}
using that $uu^*\in N'\vnten\mc B(K)$ and that $S$ is an $N'$-bimodule.  Reversing the roles of $H_N$ and $K_N$ gives the other inclusion, hence equality, as required.
\end{proof}

In the finite-dimensional case, we can use a duality argument to give a link between pushforwards and pullbacks.
The following should be compared to a more general
construction which holds in the infinite-dimensional case under some conditions, see
\cite[Proposition~3.1]{ac} for example.

\begin{proposition}\label{prop:ac_adjoints}
Let $B_1,B_2$ be finite-dimensional $C^*$-algebras equipped with faithful traces $\psi_1,\psi_2$.  Let $\theta:
B_2\rightarrow B_1$ be a completely positive map.  Define $\hat\theta:B_1\rightarrow B_2$ by
\[ \psi_1( a \theta(b) ) = \psi_2( \hat\theta(a) b )
\qquad (a\in B_1, b\in B_2). \]
Then $\hat\theta$ exists, and is completely positive.  $\hat\theta$ is UCP if and only if $\theta$ is trace
preserving: $\psi_1\circ\theta = \psi_2$.
\end{proposition}
\begin{proof}
As $\psi_2$ is faithful, any functional $f$ on $B_2$ arises from a unique $b_0\in B_2$ as $f(b) = \psi_2(b_0b)$
for $b\in B_2$, compare with the discussion at the start of Section~\ref{sec:super-ops}.
Applying this to the functional $b\mapsto \psi_1( a \theta(b) )$ gives the existence and
uniqueness of $\hat\theta(a)$.  Assuming that $\theta$ is positive, for $a\geq 0$,
\[ \psi_2( b^* \hat\theta(a) b ) = \psi_2( \hat\theta(a) bb^* ) = \psi_1( a \theta(bb^*) )
= \psi_1( a^{1/2} \theta(bb^*) a^{1/2} ) \geq 0 \qquad (b\in B_2). \]
This shows that $\hat\theta(a)\geq 0$, so $\hat\theta$ is positive.

Consider now $\hat\theta\otimes\id : B_1\otimes\mathbb M_n \rightarrow B_2\otimes\mathbb M_n$.
Giving $\mathbb M_n$ the trace $\Tr$ and hence $B_i\otimes\mathbb M_n$ the faithful trace $\psi_i\otimes\Tr$,
it follows readily that $\hat\theta\otimes\id$ and $\theta\otimes\id$ are related in the same way, and hence
the previous argument shows that $\hat\theta\otimes\id$ is positive when $\theta\otimes\id$ is.  It follows that
$\theta$ CP implies that $\hat\theta$ is CP.

Finally, $\psi_2(\hat\theta(1)b) = \psi_1(\theta(b))$ and so $\hat\theta(1)=1$ if and only if $\theta$ is
trace preserving.
\end{proof}

For the following, compare with \cite[Proposition~2.1]{b2} but be aware of differing (normalisation) conventions.  The trace constructed in the following is sometimes called the \emph{Markov trace}.

\begin{lemma}\label{lem:markov_trace}
Let $B = \bigoplus_{i=1}^n \mathbb M_{n_i}$ be a finite-dimensional $C^*$-algebra.  Let $\psi = \oplus_i n_i \Tr_i$ where $\Tr_i:\mathbb M_{n_i}\rightarrow\mathbb C$ is the usual (non-normalised) trace.  Form $L^2(B)$ using $\psi$ and consider the usual trace $\Tr:\mc B(L^2(B))\rightarrow\mathbb C$.  Then $\Tr$ restricts to $\psi$ on $B\subseteq\mc B(L^2(B))$.
\end{lemma}
\begin{proof}
As everything respects the direct sum decomposition, it suffices to prove this in the case $B=\mathbb M_n$ with $\psi = n\Tr$.  We identify the GNS space $L^2(B)$ with $\mathbb C^n \otimes \overline{\mathbb C^n}$ with GNS map $\Lambda(a) = n^{1/2} \sum_{i,j} a_{ij} e_i \otimes \overline{e_j}$, as then $(\Lambda(a)|\Lambda(b)) = \psi(a^*b)$ for $a,b\in B$.  The GNS action of $B$ is $a\in B$ acting as $a\otimes 1$ on $\mathbb C^n \otimes \overline{\mathbb C^n}$.  Then
\begin{align*}
\Tr_{L^2(B)}(a\otimes 1)
&= \sum_{i,j} (e_i\otimes\overline{e_j} | a(e_i)\otimes\overline{e_j})
= n \sum_i (e_i|a(e_i)) = n\Tr(a) = \psi(a),
\end{align*}
as we want.
\end{proof}

\begin{proposition}\label{prop:Kraus_hat_theta}
For $i=1,2$ let $B_i$ be a finite-dimensional $C^*$-algebra equipped with its Markov trace $\psi_i$.  Let $\theta:B_2\rightarrow B_1$ be a completely positive map with Kraus form $\theta(x)=\sum_{i=1}^n b_i^* x b_i$ for some $b_i\in\mc B(H_1,H_2)$.  Let $(x_j)_{j=1}^m \subseteq B_2$ be such that $(\Lambda(x_j))$ is an orthonormal basis for $L^2(B_2)$.
Then $\hat\theta$ has Kraus form $\hat\theta(y) = \sum_{i,j} Jx_jJ b_i y b_i^* Jx_j^*J$ for $y\in B_1$, where $J$ is the modular conjugation on $L^2(B_2)$.
\end{proposition}
\begin{proof}
Let $U:L^2(B_1) \rightarrow L^2(B_2) \otimes \mathbb C^n$ be $\xi\mapsto \sum_{i=1}^n b_i(\xi)\otimes \delta_i$, so that $\theta(x) = U^*(x\otimes 1)U$ for $x\in B_2$.  As $\psi_2$ is a trace, we have in particular that $Jb^*J\Lambda(a) = \Lambda(ab) = a \Lambda(b)$ for $a,b\in B_2$.  Define
\[ V : L^2(B_2) \rightarrow L^2(B_1) \otimes L^2(B_2) \otimes \mathbb C^n;
\quad \xi \mapsto \sum_{i,j} U^*(Jx_j^*J\xi\otimes \delta_i) \otimes \Lambda(x_j) \otimes \delta_i. \]
For $b,c\in B_2$ and $a\in B_1$ we see that
\begin{align*}
\big( \Lambda(b) \big| V^*(a\otimes 1)V \Lambda(c) \big)
&= \sum_{i,j} \big( U^*(b\Lambda(x_j)\otimes \delta_i) \big| a U^*(c\Lambda(x_j)\otimes \delta_i) \big) \\
&= (\Tr_{L^2(B_2)}\otimes\Tr_{n})\big( (b^*\otimes 1)UaU^*(c\otimes 1) \big) \\
&= \Tr_{L^2(B_1)}\big( a U^*(cb^*\otimes 1)U \big)
= \Tr_{L^2(B_1)}(a \theta(cb^*) ) = \psi_1(a\theta(cb^*)) \\
&= \psi_2(\hat\theta(a)cb^*) = \psi_2(b^* \hat\theta(a) c)
= \big(\Lambda(b) \big| \hat\theta(a) \Lambda(c)\big),
\end{align*}
here using Lemma~\ref{lem:markov_trace}.
Hence $V^*(a\otimes 1)V = \hat\theta(a)$.
The Kraus operators are hence given by
\[ c_{i,j}(\xi) = U^*(Jx_j^*J\xi\otimes \delta_i)
= b_i^* Jx_j^*J(\xi) \qquad (\xi\in L^2(B_2)), \]
which gives $\hat\theta(a) = \sum_{i,j} c_{i,j}^* a c_{i,j} =  \sum_{i,j} Jx_jJ b_i a b_i^* Jx_j^*J$ as claimed.
\end{proof}

\begin{corollary}
Let $\theta:B_2\rightarrow B_1$ be a CP map, and let $S$ be a quantum graph over $B_1$.  Form $\hat\theta:B_1 \rightarrow B_2$ using the Markov traces on $B_1, B_2$, and use this to form the pullback $\overleftarrow{S}$.  Then $\overrightarrow{S} = \overleftarrow{S}$.
\end{corollary}
\begin{proof}
We continue with the notation of Proposition~\ref{prop:Kraus_hat_theta}.  By Corollary~\ref{corr:pullback_fd}, Proposition~\ref{prop:pullback_indep_rep} and Proposition~\ref{prop:Kraus_hat_theta},
\[ \overleftarrow{S} = \lin\{ Jx_lJ b_k y b_i^* Jx_j^*J : y\in S \}. \]
As $Jx_jJ \in B_2'$ for each $j$, and as $(x_j)$ is a basis for $B_2$, we see that
\[ \overleftarrow{S} = \text{$B_2'$-bimodule generated by } \{ b_k y b_i^* : y\in S \} = \overrightarrow{S}, \]
as claimed.  In this final step, we use Proposition~\ref{prop:pushforward_indep_rep} which tells us that we are free to represent $B_1$ on $L^2(B_1)$ without changing $\overrightarrow{S}$, and use Proposition~\ref{prop:pushforward_indep_Kraus} to use a Kraus representation of our choice.
\end{proof}

We have not discussed quantum graphs which arise from quantum adjacency matrices with respect to a non-tracial
state.  By Theorem~\ref{thm:reduce_tracial_case}, this is equivalent to studying quantum graphs $S$ with the
additional property that $QSQ^{-1}=S$, but it is not very clear what condition we might place on a CP map
$\theta$ to ensure this.  Further, we have not discussed links with quantum adjacency matrices $A$ nor projections
$e$, because it seems somewhat hard to give concrete formula for, say, $\overleftarrow{e}$ associated to
$\overleftarrow{S}$.

\subsection{Further notions}\label{sec:hm_further_ideas}

We give a quick summary of some of the other suggestions for what a ``homomorphism'' for a quantum graph
should be, and related ideas.

In \cite[Section~6]{weaver2} given a quantum graph $S\subseteq\mc B(H)$ over $B\subseteq\mc B(H)$ and a
projection $p\in B$, the \emph{restriction} $S_p$ is $pSp$, the quantum graph over $pBp \subseteq\mc B(p(H))$.
The map $\theta:B\rightarrow pBp, a\mapsto pap$ is UCP with Stinespring dilation $\pi=\id, K=H, V=p$, and so
$pSp$ is simply the pullback of $S$ along this UCP map.

In \cite[Definition~7]{stahlke}, given quantum graphs $S_i\subseteq\mc B(H_i)$ over $\mc B(H_i)$ for $i=1,2$,
a \emph{homomorphism} $S_1\rightarrow S_2$ is defined to be a UCP map $\theta:\mc B(H_2) \rightarrow \mc B(H_1)$
with Kraus form $\theta(x) = \sum_i b_i^* x b_i$, with the property that $b_i S_1 b_j^* \subseteq S_2$ for
each $i,j$.  This is equivalent to $\overrightarrow{S_1} \subseteq S_2$, which agrees with Weaver's notion
of a \emph{CP morphism}, \cite[Definition~7.5]{weaver2}, once we recall that we always work with UCP maps,
not TPCP maps.  In the case $B=\mathbb M_n$, links between these ideas and TRO equivalence are
explored in \cite[Section~7]{ekt}.

We should note that here Stahkle works with trace-free quantum graphs, see the remarks after
Proposition~\ref{prop:complements_quan_graph}.  If we look instead at the quantum graphs $S_i^\perp$, then
the condition becomes $b_i^* S_2^\perp b_j \subseteq S_1^\perp$ for each $i,j$.  That is, $\overleftarrow{S_2^\perp}
\subseteq S_1^\perp$.  However, it seems to the author that these notions make sense in general, regardless of
whether we work with unital, trace-free, or more general quantum relations.

\begin{remark}\label{rem:quan_homo_to_graph_homo}
We follow \cite{stahlke} and motivate this definition from the classical situation.
Consider a classical graph $G$ with associated $S_G = \lin\{ e_{u,v} : (u,v)\in E(G) \}$ over $C(V(G))$.
Similarly for a graph $H$.  Let $f:G\rightarrow H$ be a homomorphism, so $f$ is a map $V(G)\rightarrow V(H)$
with $(u,v)\in E(G) \implies (f(u), f(v))\in E(H)$.

For the moment, let $f:V(G)\rightarrow V(H)$ be any map.  We obtain a unital $*$-homomorphism $\theta:
C(V(H))\rightarrow C(V(G))$ given by $\theta(a)(u) = a(f(u))$ for $a\in C(V(H)), u\in V(G)$.  Then $\theta$
is in particular a UCP map, and we can construct a Stinespring dilation as follows.

Let $n = |V(G)|$ and
define $V:\ell^2(V(G)) \rightarrow \ell^2(V(H))\otimes\mathbb C^n$ as follows.  Define an equivalence relation
on $V(G)$ by setting $u\sim v$ if $f(u)=f(v)$.  For $u\in V(G)$ fix an enumeration of the equivalence class
$[u] = \{ v_1,\cdots,v_k \}$, where $k\leq n$, and define $V(\delta_u) = \delta_{f(u)} \otimes \delta_i$ where
$v_i = u$.  This construction ensures that $V$ is an isometry, and that
\[ V^*(a\otimes 1)V\delta_u = a(f(u)) \delta_u \qquad (u\in V(G), a\in C(V(H))). \]

The Kraus form is obtained by setting $V(\delta_u) = \sum_i b_i(\delta_u)\otimes\delta_i$, so that
$b_i(\delta_u) = \delta_{f(u)}$ if $u$ is the $i$th vertex which maps to $f(u)$, and $0$ otherwise.
For $(u,v)\in E(G)$, so that $e_{u,v} \in S_G$, we see that
\[ b_i e_{u,v} b_j^* = \theta_{b_j(\delta_v), b_i(\delta_u)}
= e_{f(u), f(v)}, \]
or $0$, depending on $i,j$ (and that there is some choice of $i,j$ which gives $e_{f(u), f(v)}$).
Hence
\[ \overrightarrow{S_G} = \lin\{ e_{f(u), f(v)} : (u,v)\in E(G) \} \]
and so $\overrightarrow{S_G} \subseteq S_H$ if and only if $f$ is a homomorphism.

Thus $V^*(a\otimes 1)V = \theta(a)$ for each $a$, and as the dilation has the correct form, it follows that
\[ \overleftarrow{S_H} = V^*(S_H\otimes \mc B(\mathbb C^n))V. \]
A calculation shows that then $\overleftarrow{S_H}$ is the linear span of ``blocks'' of the form
$\sum \{ e_{u,v} : f(u)=x, f(v)=y \}$ as $(x,y)\in E(H)$ vary.  This indeed corresponds to a ``pullback'' of $H$ along
$f$ to a graph on $G$, where we expand each edge in $H$ to a clique in $G$.
\end{remark}

Stahkle shows further in \cite{stahlke} that given any CP morphism $C(V(H))\rightarrow C(V(G))$, there is
some graph theoretic homomorphism $G\rightarrow H$.  However, this construction is not a bijection, but
rather an existence result (which motivates Weaver to use the new terminology of ``CP morphism'').
Nevertheless, for many graph theory questions, we are merely interested in the \emph{existence} of any
homomorphism between two graphs, and for this, the notion of a CP morphism exactly extends this to
quantum graphs (which is indeed Stahkle's motivation).  This notion of homomorphism is recast into the
theory of \emph{non-local games} in \cite[Section~10.2]{tt}, see also \cite[Section~5]{bhtt}.

\cite{dsw} (see Section~II) and \cite{stahlke} (see after Proposition~9) give tentative definitions for what
subgraphs or induced subgraphs should be.  Gromada in \cite[Section~2.2]{grom} looks at two notions of a
``quantum subgraph'', which are defined using the formalism of projections $e\in B\otimes B^\op$.  Firstly,
$e_1$ is a \emph{subgraph made by removing edges} of $e_2$ if $e_1 \leq e_2$, which is clearly equivalent
to containment of the related operator bimodules.  When we have $q:B_1\rightarrow B_2$ a surjective
$*$-homomorphism, then $e_q$ the \emph{induced subgraph} of $e$ along $q$ is given by $e_q = (q\otimes q)(e)$.
It seems tricky to give a simple description of the associated operator bimodules.
Finally, \cite[Section~2.4]{grom} looks at more general notions of ``quotient'', which requires the
consideration of a more general object, a \emph{weighted quantum graph}.

\section{Automorphisms of Quantum graphs}\label{sec:auts_quan_graphs}

As a warm-up to looking at quantum automorphisms, we look at (classical) automorphisms of quantum graphs.
Let us quickly discuss some well-known results about automorphisms of $B$.
By $\aut(B,\psi)$ we mean the automorphisms of $B$, the bijective $*$-homomorphisms $B\rightarrow B$, which
leave $\psi$ invariant.  As $B$ is finite-dimensional, it is easy to understand automorphisms of $B$.  We give
a quick proof of the following well-known result, for completeness.

\begin{proposition}\label{prop:auts_fin_cstar_alg}
Let $B = \bigoplus_{k=1}^n \mathbb M_{n_k}$, and let $\theta:B\rightarrow B$ be an automorphism.
Then $\theta$ restricts to a bijection from each block $\mathbb M_{n_k}$ to a block of the same dimension,
and each such $*$-isomorphism $\mathbb M_m \rightarrow \mathbb M_m$ is of the form $x\mapsto uxu^*$ for
some unitary $u\in\mathbb M_m$.
\end{proposition}
\begin{proof}
It is well known that a $*$-isomorphism $\mathbb M_m \rightarrow \mathbb M_m$ is of the form stated, so
we just shown the first part of the claim.  For projections $e,f$ we have by definition that $e\leq f$ if and
only if $ef=e$, and so an automorphism preserves this partial ordering.  Clearly also $\theta$ maps the centre
of $B$ to itself.  Thus $\theta$ gives a bijection when restricted to the set of minimal central projections.
The minimal central projections of $B$ are exactly the units of the matrix blocks.
Fix some matrix block $\mathbb M_{n_k}$ and let the unit be $e$.  Let $(e_{ij})_{i,j=1}^{n_k}$ be the matrix
units of this block, so $e e_{ij} = e_{ij} e = e_{ij}$ for each $i,j$.  If $\theta(e)$ is the unit of the matrix
block $\mathbb M_{n_l}$, then we see that $\theta(e_{ij}) \leq \theta(e)$ for each $i,j$, and so each
$\theta(e_{ij})$ is in this block, as required.
\end{proof}

As $\psi$ is given by the positive invertible $Q\in B$, we can work out which automorphisms are in
$\aut(B,\psi)$.

\begin{lemma}\label{lem:aut_psi_Q_inv}
Let $\theta$ be an automorphism of $B$.  Then $\theta\in\aut(B,\psi)$ if and only if $\theta(Q)=Q$,
and in this case, $\theta$ and $(\sigma_t)$ commute.
\end{lemma}
\begin{proof}
We first show that $\theta\in\aut(B,\psi)$ commutes with $(\sigma_t)$.  We use the uniqueness property
of $(\sigma_t)$, see for example \cite[Theorem~1.2, Chapter~VIII]{tak2}, which says that $\psi\circ\sigma_t
=\sigma_t$ for each $t\in\mathbb R$, and for every $a,b\in B$ there is a continuous function $F=F_{a,b}$
on $\{ z\in\mathbb C : 0\leq \im(z)\leq 1\}$ which is holomorphic on the interior, and with
$F(t) = \psi(\sigma_t(a)b)$ and $F(t+i)=\psi(b\sigma_t(a))$.  These properties determine $(\sigma_t)$
uniquely.

As $\psi\circ\theta=\psi$, also $\psi\circ\theta^{-1}=\psi$, so defining $\sigma^\theta_t =
\theta\circ\sigma_t\circ\theta^{-1}$, we have that $\psi\circ\sigma^\theta_t = \psi$.  Also, for $a,b\in B$,
let $F' = F_{\theta^{-1}(a), \theta^{-1}(b)}$, and it is easily verified that $F'$ satisfies the condition
for $(\sigma^\theta_t)$.  By uniqueness, $\sigma_t = \sigma^\theta_t$ for all $t$, equivalently, $\theta$
and $(\sigma_t)$ commute.

By analytic continuation, also $\theta$ commutes with $\sigma_{-i}$, and so $\theta(QaQ^{-1}) =
\theta(\sigma_{-i}(a)) = \sigma_{-i}(\theta(a)) = Q \theta(a) Q^{-1}$ for each $a\in B$.  Thus $\theta(Q)
\theta(a) \theta(Q)^{-1} = Q \theta(a) Q^{-1}$, equivalently, $Q^{-1} \theta(Q) \theta(a) = \theta(a) Q^{-1}
\theta(Q)$ for each $a\in B$.  So $Q^{-1}\theta(Q)$ is central in $B$, and so on each matrix block of $B$ is
a scalar multiple of the identity.  Thus, with $Q=(Q_k)$, there are scalars $q_k$ with $\theta(Q_k) = q_k Q_k$.
As $\psi(a) = \Tr(Qa)$ and $\psi(\theta(a))=\psi(a)$, we see that $\Tr(Qa) = \Tr(Q\theta(a))$ for $a\in B$.
Taking $a=Q_k^z$ (in the $k$th block, and $0$ in other blocks) yields $\Tr(Q_k^{z+1}) = q_k^z \Tr(Q_k^{z+1})$
as $\theta(Q_k^z) = q_k^z Q_k^z$.  Setting $z=-1$ shows that $q_k=1$.  Thus $\theta(Q) = Q$.

Conversely, if $\theta(Q)=Q$, then notice that it follows from Proposition~\ref{prop:auts_fin_cstar_alg} that
$\Tr\circ\theta=\Tr$, and so $\psi(\theta(a))=\Tr(Q\theta(a))=\Tr(\theta(Qa))=\Tr(Qa)=\psi(a)$, so
$\theta\in\aut(B,\psi)$ as claimed.
\end{proof}

\begin{definition}\label{defn:hat_theta}
Given $\theta\in\aut(B,\psi)$ we define $\hat\theta\in\mc B(L^2(B))$ by
\[ \hat\theta\Lambda(b) = \Lambda(\theta(b)) \qquad (b\in B). \]
\end{definition}

Notice that $(\Lambda(\theta(b))|\Lambda(\theta(c))) = \psi(\theta(b)^*\theta(c)) = \psi(\theta(b^*c))
= \psi(b^*c) = (\Lambda(b)|\Lambda(c))$ and so $\hat\theta$ is an isometry.  As $\theta$ preserves $\psi$
also $\theta^{-1}$ does, and so
\[ (\Lambda(b) | \hat\theta^*\Lambda(c)) = (\Lambda(\theta(b)) | \Lambda(c))
= \psi(\theta(b)^* c) = \psi(\theta^{-1}(\theta(b^*)c)) = \psi(b^* \theta^{-1}(c))
= (\Lambda(b) | \widehat{\theta^{-1}}\Lambda(c)). \]
Thus $\hat\theta^* = \widehat{\theta^{-1}}$, and it is now easy to see that $\theta\mapsto \hat\theta$ is
a unitary representation of $\aut(B,\psi)$ on $L^2(B)$.

Shortly $A$ will denote a quantum group, and so in this section we shall always write $A_G$ for a quantum
adjacency matrix.
In the following, we are deliberately a little vague as to which exact axioms we need $A_G$ to satisfy.

\begin{definition}
Let $A_G\in\mc B(L^2(B))$ be a quantum adjacency matrix.  An \emph{automorphism of $A_G$} is $\theta\in\aut(B,\psi)$
such that $A_G$ and $\hat\theta$ commute.
\end{definition}

\begin{remark}
Let $B=C(V)$ and let $A=A_G$ be the adjacency matrix of a classical (perhaps directed) graph, so $A_{ij}=1$
exactly when there is an edge from $j$ to $i$.  Let us denote that there is an edge from $j$ to $i$ by
$j\rightarrow i$.  Any automorphism $\theta$ is given by some $\sigma\in\sym(V)$,
and notice that the uniform measure on $V$ is automatically preserved.  Then $L^2(B)\cong\mathbb C^V$
and $\hat\theta$ is also induced by $\sigma$, in that $\hat\theta(\delta_i) = \delta_{\sigma(i)}$.  As
\[ A(\delta_i) = \sum_j A_{ji} \delta_j = \sum_{i\rightarrow j} \delta_j, \]
we see that
\begin{align*}
\hat\theta A(\delta_i) = \sum_{\{j:i\rightarrow j\}} \hat\theta(\delta_j)
= \sum_{\{j:i\rightarrow j\}} \delta_{\sigma(j)}
= \sum_{\{k:i\rightarrow \sigma^{-1}(k)\}} \delta_{k}, \qquad
A \hat\theta(\delta_i) = A(\delta_{\sigma(i)})
= \sum_{\{j:\sigma(i)\rightarrow j\}} \delta_j.
\end{align*}
Thus $\hat\theta A = A \hat\theta$ if and only if, for each $i$, the sets $\{ j : i\rightarrow\sigma^{-1}(j) \}$
and $\{ j :\sigma(i)\rightarrow j \}$ agree.  This is equivalent to the statement that $i\rightarrow k$ if and
only if $\sigma(i)\rightarrow\sigma(k)$, which we observe by setting $k=\sigma^{-1}(j)$.  That is, $\sigma$ is
an automorphism of the (directed) graph.
\end{remark}

As a final remark, we notice that it is not really necessary to work with $L^2(B)$ here.  Instead, we could
interpret the quantum adjacency matrix $A_G$ as a linear map on $B$ (compare Section~\ref{sec:super-ops}), and then
require that $A_G$ and $\theta$ commute; this is \cite[Definition~2.4]{grom} for example.

So far, we have been discussing well-known definitions.  To our knowledge, however, there is little or no discussion
in the literature around what an automorphism of an operator bimodule might be.  We shall follow
Theorem~\ref{thm:qu_ad_mats_axioms_12}.  Let $A_G\in\mc B(L^2(B))$ be a quantum adjacency matrix, satisfying
axioms (\ref{defn:quan_adj_mat:idem}) and (\ref{defn:quan_adj_mat:undir}).  Then $A_G$ corresponds to a
projection $e\in B\otimes B^\op$ with $e=\sigma(e)$ and with $e$ invariant under $(\sigma_t\otimes\sigma_t)$.
Notice that when $\theta$ is an automorphism of $B$, also $\theta$ is an automorphism of $B^\op$.

\begin{proposition}\label{prop:aut_of_A_to_e}
We have that $A_G\hat\theta = \hat\theta A_G$ if and only if $(\theta\otimes\theta)(e)=e$.
\end{proposition}
\begin{proof}
We can check that $e' = (\theta\otimes\theta)(e)$ satisfies the conditions of Theorem~\ref{thm:qu_ad_mats_axioms_12}
and so corresponds to $A_G'\in\mc B(L^2(B))$ which satisfies axioms (\ref{defn:quan_adj_mat:idem})
and (\ref{defn:quan_adj_mat:undir}) of Definition~\ref{defn:adj_op}.  We will show that $A_G' = \hat\theta A_G \hat\theta^*$, and so
$A_G\hat\theta = \hat\theta A_G$ if and only if $A_G' = A_G$ if and only if $e=e'$, as claimed.

Using $\Psi_{0,1/2}:\mc B(L^2(B)) \rightarrow B\otimes B^\op$ we see that
\begin{align*} \hat\theta \theta_{\Lambda(a),\Lambda(b)} \hat\theta^*
&= \theta_{\hat\theta\Lambda(a), \hat\theta\Lambda(b)}
= \theta_{\Lambda(\theta(a)), \Lambda(\theta(b))} \\
&\mapsto \theta(a)^* \otimes \sigma_{i/2}(\theta(b))
= (\theta\otimes\theta)(a^*\otimes\sigma_{i/2}(b))
= (\theta\otimes\theta)\Psi_{0,1/2}(\theta_{\Lambda(a),\Lambda(b)}),
\end{align*}
and so $\hat\theta A_G \hat\theta^*$ maps to $e'$, showing that $A_G' = \hat\theta A_G \hat\theta^*$,
as claimed.
\end{proof}

Consider now $B\subseteq\mc B(H)$ for some Hilbert space $H$.  Let $\theta_{\mc B}$ be an automorphism of
$\mc B(H)$, and suppose that $\theta_{\mc B}$ restricts to $B$ giving some $\theta\in\aut(B,\psi)$.  Notice that
for $a\in B, x\in B'$ we have that $\theta_{\mc B}(x) \theta(a) = \theta_{\mc B}(xa) = \theta_{\mc B}(ax)
= \theta(a) \theta_{\mc B}(x)$.  As $\theta$ is an automorphism, it follows that $\theta_{\mc B}(x)\in B'$,
so $\theta_{\mc B}$ restricts to an endomorphism of $B'$.  The same argument applies to $\theta_{\mc B}^{-1}$,
and so $\theta_{\mc B}$ restricts to an automorphism of $B'$.  As all automorphisms of $\mc B(H)$ are inner, there
is some unitary $u_\theta\in\mc B(H)$ with $\theta_{\mc B}(x) = u_\theta x u_\theta^*$ for $x\in\mc B(H)$.
Thus $\theta_{\mc B}$ preserves $\Tr$ on $\mc B(H)$.  Using the GNS map $\mc B(H)\rightarrow H\otimes\overline H$,
the automorphism $\theta_{\mc B}$ corresponds to the unitary $u_\theta \otimes \overline{u_\theta}$ on
$H\otimes\overline H$.

\begin{lemma}\label{lem:action_on_S_to_e}
Let $S\subseteq\mc B(H)$ be a $B'$-bimodule satisfying the conditions of Theorem~\ref{thm:qu_ad_mats_axioms_12}.  
Then $S_\theta = \theta_{\mc B}(S)$ is also a $B'$-bimodule satisfying the conditions of 
Theorem~\ref{thm:qu_ad_mats_axioms_12}.  The associated projection in $B\otimes B^\op$ is
$e_\theta = (\theta\otimes\theta)(e)$.
\end{lemma}
\begin{proof}
As $\theta_{\mc B}$ restricts to an automorphism of $B'$, it follows that $B_\theta$ is a self-adjoint $B'$-bimodule.
From Lemma~\ref{lem:aut_psi_Q_inv}, we know that $\theta_{\mc B}(Q) = Q$, and so $Q S Q^{-1} = S$ implies that
$Q S_\theta Q^{-1} = S_\theta$.  Let $S$ corresponds to $V\subseteq H\otimes\overline H$, and similarly
$S_\theta$ correspond to $V_\theta$.  As noted above, $V_\theta = (u_\theta \otimes \overline{u_\theta})(V)$.
Let $e\in B\otimes B^\op$ be the projection onto $V$, so that
\[ V_\theta = (u_\theta \otimes \overline{u_\theta}) e(H\otimes\overline H)
= (u_\theta \otimes \overline{u_\theta}) e (u_\theta \otimes \overline{u_\theta})^* (H\otimes\overline H), \]
and so $e_\theta = (u_\theta \otimes \overline{u_\theta}) e (u_\theta \otimes \overline{u_\theta})^*$
is the orthogonal projection onto $V_\theta$.  However, then $e_\theta = (u_\theta \otimes \overline{u_\theta})
e (u_\theta^* \otimes u_\theta^\top) = (\theta\otimes\theta)(e)$, as claimed.  (Here we use that, acting on
$\overline H$, we have that $\theta(b)^\top = (u_\theta b u_\theta^*)^\top = \overline{u_\theta} b^\top
u_\theta^\top$.)
\end{proof}

Thus we might define an automorphism of $S$ to be such a $\theta_{\mc B}$ with $\theta_{\mc B}(S) = S$,
equivalently, with $(\theta\otimes\theta)(e)=e$.  However, there are a number of problems here.  Firstly,
different $\theta_{\mc B}$ might give the same $\theta$, and hence the same automorphism of $S$.  Secondly,
we would wish there to be a bijection between automorphisms of a quantum adjacency matrix $A_G$ and automorphisms
of the associated $S$, but it is not clear that every suitable $\theta$ arises from some $\theta_{\mc B}$.
Finally, when studying operator bimodules $S$, there is no real dependence on $H$, but it does not appear to
be the case that any $\theta_{\mc B}$ on $\mc B(H)$, say restricting to $\theta$ on $B$, with give rise to some
automorphism of $\mc B(K)$ whenever $B\subseteq\mc B(K)$.

With a special choice of $H$, we can say something.  Set $H=L^2(B)$ and given $\theta\in\aut(B,\psi)$ form
$\hat\theta$ as in Definition~\ref{defn:hat_theta}.  We now define $\theta_{\mc B}(x) = \hat\theta x \hat\theta^*$
for $x\in\mc B(L^2(B))$, an automorphism which restricts to $\theta$ on $B$.

\begin{proposition}
Let $A_G\in\mc B(L^2(B))$ be a quantum adjacency matrix, satisfying axioms (\ref{defn:quan_adj_mat:idem}) and
(\ref{defn:quan_adj_mat:undir}) of Definition~\ref{defn:adj_op}.  Then $\theta\in\aut(B,\psi)$ is an automorphism of $A_G$ if and only if
$\theta_{\mc B}(S)=S$, where $S$ is the $B'$-operator bimodule associated with $A_G$.
\end{proposition}
\begin{proof}
This follows immediately from Proposition~\ref{prop:aut_of_A_to_e} and Lemma~\ref{lem:action_on_S_to_e} by
thinking about the projection associated to $A_G$ and $S$.
\end{proof}

\section{Quantum automorphisms of Quantum graphs}\label{sec:qauts_quan_graphs}

We now look at how (compact) quantum groups act on (quantum) graphs.  For compact quantum groups acting
on classical graphs, this goes back to work of Banica in \cite{b4}.  In the subsequent years, there has been
much activity, with the quantum automorphisms of some small graphs computed, and much interest in the question
of which graphs have a non-commutative (that is, not arising from a classical group) quantum automorphism group.  In this direction, let us just mention
one recent paper \cite{rs} which constructs a finite graph which has a non-commutative, yet finite-dimensional,
quantum automorphism group.  More recently quantum automorphisms of quantum graphs have been defined and
studied, \cite[Section~V, Part~C]{mrv} and more completely, \cite[Section~3]{bcehpsw}, compare
\cite[Section~4]{grom}, \cite[Section~2.3]{matsuda}.

We find some of the literature in this area to be quite hard to follow, and so we shall both motivate the
definitions from the classical situation, and develop certain aspects of compact quantum group theory
to allow us to make broader definitions than available elsewhere.
We will assume the reader is familiar with the basics of compact quantum group theory, for example,
\cite{nt, t}.

We start again from the classical situation, which while well-known, we feel provides interesting motivation.
Let $G=\aut(B,\psi)$ and let $A=C(G)$ the group algebra, with coproduct
\[ \Delta:C(G)\rightarrow C(G)\otimes C(G) = C(G\times G); \quad
\Delta(f)(\theta_1, \theta_2) = f(\theta_1\theta_2) \qquad(f\in C(G), \theta_1, \theta_2\in G). \]
The action of $G$ on $B$ gives a coaction of $C(G)$ on $B$ by
\[ \alpha:B\rightarrow B\otimes C(G) \cong C(G,B); \quad
\alpha(b) = \big( \theta(b) \big)_{\theta\in G}. \]
Let us check that the coaction equation holds.  For $b\in B$ we identify $(\alpha\otimes\id)\alpha(b)
\in B\otimes C(G)\otimes C(G)$ as a member of $C(G\times G, B)$.  Indeed, for arbitrary $f\in C(G,B)
\cong B\otimes C(G)$ we have that $(\alpha\otimes\id)f$ evaluated at $(\theta_1,\theta_2)$ is
$\alpha(f(\theta_2))(\theta_1) = \theta_1(f(\theta_2))$, while $(\id\otimes\Delta)f$ at $(\theta_1,\theta_2)$
is $f(\theta_1\theta_2)$.  Thus
\begin{align*}
(\alpha\otimes\id)\alpha(b)(\theta_1,\theta_2) = \theta_1\big( \alpha(b)(\theta_2) \big)
= \theta_1(\theta_2(b)) = (\id\otimes\Delta)\alpha(b)(\theta_1,\theta_2),
\end{align*}
showing that indeed $(\alpha\otimes\id)\alpha = (\id\otimes\Delta)\alpha$.  Notice that with these conventions,
we are considering \emph{right} coactions.

One can also verify the Podl\'es density condition, that $\lin\{ (1\otimes a) \alpha(b) : b\in B, a\in C(G) \}$
is dense in $B\otimes C(G)$.  Even in this setting, verifying this condition does not appear to be totally trivial
to do by direct calculation, and we point the reader to \cite[Section~2]{soltan1} for an interesting discussion.

That $G$ is a group of automorphisms which preserve $\psi$ means that $(\psi\otimes\id)\alpha(b) = \psi(b) 1$
for each $b\in B$.  If we regard a quantum adjacency matrix $A_G$ as acting on $B$, then that $G$ is a group
of automorphisms of $A_G$ is equivalent to $\theta(A_G(b)) = A_G(\theta(b))$ for each $b\in B$, that is,
$(A_G\otimes\id)\alpha = \alpha A_G$.  If we, as is usual, regard $A_G$ as acting on $L^2(B)$, then we have to
work harder to interpret this at the coaction level, see below.

We now consider an arbitrary, possibly non-commutative, compact quantum group $(A,\Delta)$.

\begin{definition}\label{defn:coaction}
Let $(A,\Delta)$ be a compact quantum group.  A \emph{coaction} of $(A,\Delta)$ on $B$ is a unital
$*$-homomorphism $\alpha:B\rightarrow B\otimes A$ with:
\begin{enumerate}
\item $(\alpha\otimes\id)\alpha = (\id\otimes\Delta)\alpha$;
\item $\lin\{ (1\otimes a)\alpha(b) : a\in A, b\in B \}$ is dense in $B\otimes A$.
\end{enumerate}

A coaction $\alpha:B\rightarrow B\otimes A$ of $(A,\Delta)$ on $B$ \emph{preserves the state $\psi$} when
$(\psi\otimes\id)\alpha(b) = \psi(b)1$ for each $b\in B$.
\end{definition}

We often think of (compact) quantum groups as being abstract objects represented by related algebraic objects (for example, a $C^*$-algebraic compact quantum group, or a Hopf $*$-algebraic compact quantum group).  In the quantum group literature, one then speaks of the (abstract object) quantum group \emph{acting} on some object if, at the level of the algebra in question, there exists a coaction.  We shall mostly stick to considering concrete algebraic objects $(A,\Delta)$ and hence will mostly speak about coactions of $(A,\Delta)$.

The condition of preserving $\psi$ was first considered by Wang in \cite{wang} (in fact, only the case
$\psi=\Tr$ is worked out in \cite[Section~6]{wang}, but as remarked there, the same ideas work in general).
We write $\qaut(B,\psi)$ for the maximal compact quantum group faithfully coacting on $B$ and preserving $\psi$.  This
means that if $(A,\Delta)$ is any compact quantum group coacting on $B$ and preserving $\psi$, there there is
a unique $*$-homomorphism $\qaut(B,\psi)\rightarrow A$ intertwining the coproducts and the coactions.  That
such a maximal quantum group exists is shown in \cite[Theorem~6.1]{wang}, where it is also show that
$\qaut(B)$ does not exist: there is no maximal compact quantum group coacting on $B$, unless $B$ is commutative.  Here we stated the maximality condition as a universal property, but faithfulness can also be described directly, see \cite[Definition~2.4]{wang}.

Let us think more about coactions.  The second condition is often called the \emph{Podle\'s density}
condition after Podle\'s's early work in this area.  Let $\mc O_A$ be the dense Hopf $*$-algebra of $A$
formed from the coefficients of irreducible unitary corepresentations of $A$.  There is a $*$-algebra
$\mc O_\alpha \subseteq B$ (the \emph{Podle\'s subalgebra} or \emph{algebraic core} of $B$)
which is dense, and such that $\alpha$ restricts to a $*$-homomorphism $\mc O_\alpha \rightarrow
\mc O_\alpha \otimes \mc O_A$.  In fact, this restriction is a Hopf $*$-algebraic coaction; in particular, it
is injective.  A modern reference for these results is \cite{dc}.

For us, $B$ is finite-dimensional, so $\mc O_B = B$, and hence $\alpha$ actually maps into $B\otimes\mc O_A$.
A modern approach to constructing $\qaut(B,\psi)$ makes use of the following observation.  
Let $(e_i)_{i=1}^n$ be an algebraic basis of $B$ such that $(\Lambda(e_i))$ is an  orthonormal basis
of $L^2(B)$.  For any linear map $\alpha:B\rightarrow B\otimes A$ there is a matrix $v=v_{ij}\in M_n(A)$ with
$\alpha(e_i) = \sum_j e_j \otimes v_{ji}$ for each $i$.  When $\alpha$ is a coaction, our observations show
that already $v_{ij}\in\mc O_A$ for each $i,j$.

We follow here
ideas of Banica, see \cite[Section~4]{b1}, or \cite[Section~1]{b2} for the case when $\psi$ is a trace.
Given unitary matrices $v\in M_n(A) \cong M_n \otimes A$ and $w\in M_m(A) \cong M_m \otimes A$ we define
$v \boxtimes w = v_{13} w_{23} \in M_n\otimes M_m \otimes A \cong M_{nm}(A)$, here using the common
leg-numbering notation.  Define $\hom(v,w)$ to be the space of linear maps $t:M_n\rightarrow M_m$ with
$(t\otimes 1)v = w(t\otimes 1)$, an equation we can interpret simply using matrix multiplication.  These
definitions should of course be familiar to the reader who knows about corepresentations of compact quantum
groups.

With respect to the basis $(\Lambda(e_i))$, write the
multiplication map $m:B\otimes B\rightarrow B$ as an $n\times n^2$ matrix, and write the unit map
$\eta:\mathbb C\rightarrow B$ as a $n\times 1$ matrix.  Also write the projection $\eta\circ\eta^*$
as a $n\times n$ matrix.  The following proposition makes links between properties of a linear map
$\alpha:B\rightarrow B\otimes A$, and properties of $v$.

\begin{proposition}\label{prop:coacts_to_coreps}
Let $A$ be a compact quantum group.  Under this bijection between $\alpha$ and $v$, the following hold:
\begin{enumerate}
\item\label{prop:coacts_to_coreps:one}
$\alpha$ is multiplicative if and only if $m \in \hom(v \boxtimes v, v)$;
\item\label{prop:coacts_to_coreps:two}
$\alpha$ is a unital map if and only if $\eta\in\hom(1,v)$;
\item\label{prop:coacts_to_coreps:three}
assuming $\alpha$ is unital, $\alpha$ is $\psi$-preserving if and only if $\eta\circ\eta^*\in\hom(v,v)$;
\item\label{prop:coacts_to_coreps:four}
$\alpha$ satisfies the coaction equation if and only if $v$ is a corepresentation.
\end{enumerate}
Assume now that $\alpha$ is multiplicative and unital.  If $v$ is unitary, then $\alpha$ is a $*$-map.
If $\alpha$ is a $*$-map, satisfies the coaction equation, and is $\psi$-preserving, then $v$ is unitary.
\end{proposition}
\begin{proof}
Most of this follows by simple direct calculation, and is rather similar to \cite[Lemma~1.2]{b2} 
or \cite[Lemma~4.5]{grom} (which both just considers the case when $\psi$ is a trace).
However, we prove the final claim, to check that it works
for a non-trace, and to indicate that this step does not follow from direct calculation, but rather from
quite a lot of theory.

Suppose that $\alpha$ is a unital $*$-homomorphism and is $\psi$-preserving.  Then consider, for $i,j$,
\begin{align*}
(\psi\otimes\id)\alpha(e_i^*e_j) &= (\psi\otimes\id)(\alpha(e_i)^*\alpha(e_j))
= \sum_{s,t} \psi(e_s^* e_t) v_{si}^* v_{tj} = \sum_{s,t} (\Lambda(e_s)|\Lambda(e_t)) v_{si}^* v_{tj} \\
&= \sum_k v_{ki}^* v_{kj} = (v^*v)_{ij}.
\end{align*}
As $\alpha$ preserves $\psi$, this is equal to $\psi(e_i^*e_j)1 = \delta_{i,j}1$ and so $v^*v=1$.

As we know that actually $v$ is a matrix in $M_n(\mc O_A)$, Hopf $*$-algebra theory would show that $v$
is invertible, as it has a left inverse.  An alternative way to show this is to check that $v$ is
\emph{non-degenerate} and then to appeal to \cite[Proposition~3.2]{w}.
(We note here that we use ``non-degenerate'' in the sense of \cite[page~630]{w}, which is weaker than
the usual Hopf $*$-algebra definition.)  Either way, we conclude that $v$ is invertible,
so necessarily also $vv^*=1$, so $v$ is unitary.
\end{proof}

\begin{remark}\label{rem:commute_eta_etastar}
The condition that  $\eta\circ\eta^*\in\hom(v,v)$ is often omitted, which is perhaps surprising.
In fact, it is automatic in the presence of the other properties.
Indeed, $\eta\in \hom(1,v)$ means that $\eta\otimes\id = v(\eta\otimes\id)$ so also $\eta^*\otimes\id
= (\eta^*\otimes\id)v^*$ and hence
\[ v(\eta\circ\eta^*\otimes\id)v^* = v(\eta\otimes\id) (\eta^*\otimes\id)v^*
= (\eta\otimes\id) (\eta^*\otimes\id) = (\eta\circ\eta^*\otimes\id). \]
If $v$ is assumed unitary, this is implies that $\eta\circ\eta^*\in\hom(v,v)$.
\end{remark}

It is now easy to see how to define the \emph{quantum automorphism group} of $B$ which preserves $\psi$,
written as $\qaut(B,\psi)$.  We define $\qaut(B,\psi)$ to be the universal $C^*$-algebra generated by the
coefficients of an $n\times n$ unitary matrix $v=(v_{ij})$ with the conditions
\[ m \in \hom(v\boxtimes v, v), \qquad \eta\in \hom(1,v). \]
The coproduct on $\qaut(B,\psi)$ is defined
to satisfy $\Delta(v_{ij}) = \sum_{k=1}^n v_{ik} \otimes v_{kj}$, and exists by universality.
We define $\alpha$ by $\alpha(e_i) = \sum_{j=1}^n e_j \otimes v_{ji}$, and the proposition shows that
$\alpha$ is a coaction.

As $\qaut(B,\psi)$ is generated by elements of the matrix $v$, which is by definition
a (unitary) corepresentation, in fact $\qaut(B,\psi)$ is a compact matrix quantum group.  One point in this
definition which is far from obvious to the author is why the axioms of a compact quantum group hold, equivalently,
why the transpose of $v$ is an invertible matrix (an alternative axiom in the case of a compact matrix quantum group).  
We find it helpful to look at \cite[Section~3]{b3} which explains links with the Tannaka--Krein reconstruction theorem.
Also not considered in this construction is the Podle\'s density condition for $\alpha$,
but this follows easily once we know that $v^t$ is invertible.  In fact, as we shall need the following later,
we shall prove by a more direct argument that $v^t$ is invertible.

\begin{lemma}\label{lem:X_matrix}
Define $X\in\mathbb M_n$ by $X_{ij} = (\Lambda(e_i)|\Lambda(e_j^*)) = \psi(e_i^*e_j^*)$.
Then $X^{-1} = \overline X$, and $v = (X\otimes 1) \overline{v} (\overline X\otimes 1)$ where $\overline{v}$
is the matrix $(v_{ij}^*)$ (the transpose of the adjoint).  Furthermore, $\overline v$ is invertible
with inverse $(\overline X\otimes 1)v^*(X\otimes 1)$, and so also $v^t$ is invertible.
\end{lemma}
\begin{proof}
Notice that $X$ is nothing but the matrix representing the antilinear map $\Lambda(x)\mapsto\Lambda(x^*)$, an observation which helps to explain the formulae.
As $\alpha$ is a $*$-map, and expanding using the orthonormal basis $(\Lambda(e_j))$, we see that
\begin{align*}
\alpha(e_i^*) &= \sum_j (\Lambda(e_j)|\Lambda(e_i^*)) \alpha(e_j)
= \sum_{j,l} e_l \otimes (\Lambda(e_j)|\Lambda(e_i^*)) v_{lj} \\
&= \alpha(e_i)^* = \sum_j e_j^* \otimes v_{ji}^*
= \sum_{j,l} e_l \otimes (\Lambda(e_l)|\Lambda(e_j^*)) v_{ji}^*, 
\end{align*}
and so $\sum_j (\Lambda(e_j)|\Lambda(e_i^*)) v_{lj} = \sum_j (\Lambda(e_l)|\Lambda(e_j^*)) v_{ji}^*$
for each $l$.  Thus $\sum_j v_{lj} X_{ji} = \sum_j X_{lj} v_{ji}^*$ or $v(X\otimes1) = (X\otimes 1)\overline{v}$.
As $\psi(b) = \Tr(Qb)$ for each $b\in B$,
\[ (\Lambda(a^*)|\Lambda(b)) = \Tr(Qab) = \Tr(bQa) = \Tr(QQ^{-1}bQa) = (\Lambda(Qb^*Q^{-1})|\Lambda(a))
\qquad(a,b\in B). \]
Hence, for each $i,j$,
\begin{align*}
(\overline X X)_{ij} &= \sum_k \overline{X_{ik}} X_{kj}
= \sum_k \overline{ (\Lambda(e_i)|\Lambda(e_k^*)) }  (\Lambda(e_k)|\Lambda(e_j^*))
= \sum_k (\Lambda(Qe_i^*Q^{-1})|\Lambda(e_k)) (\Lambda(e_k)|\Lambda(e_j^*)) \\
&= (\Lambda(Qe_i^*Q^{-1})|\Lambda(e_j^*))
= (\Lambda(e_j)|\Lambda(e_i)) = \delta_{i,j}.
\end{align*}
As $X$ is a scalar matrix, it follows that $X^{-1} = \overline X$, and so $v = (X\otimes 1)\overline v
(X^{-1}\otimes 1) = (X\otimes 1)\overline v(\overline X\otimes 1)$, as claimed.
As $v^*v=vv^*=1$ we see that $u = (\overline X\otimes 1)v^*(X\otimes 1)$ is the inverse of $\overline v$,
and hence $u^* = (X^*\otimes 1)v(X^t\otimes 1)$ is the inverse of $v^t$.
\end{proof}

We shall make use of this construction, and our parenthetical comments, below.
We hope also that this quick sketch of the construction will aid the reader in understanding links between
our presentation and, for example, \cite[Definition~3.7, Proposition~3.8]{bcehpsw}.  See also the informative
\cite[Section~4]{grom}.

\subsection{Unitary implementations}\label{sec:unitary_impl}

In this section, let $(A,\Delta)$ be a compact quantum group with a coaction $\alpha:B\rightarrow B\otimes A$
which preserves $\psi$.  Consider $L^2(B)$ with GNS map $\Lambda:B\rightarrow L^2(B)$, and let $A$ be faithfully
and non-degenerately represented on a Hilbert space $K$.  Define an operator $U$ on $L^2(B)\otimes K$ by
\[ U(\Lambda(b)\otimes\xi) = \alpha(b)(\Lambda(1)\otimes\xi) \qquad (b\in B, \xi\in K). \]
Then
\begin{align*} \big( U(\Lambda(b)\otimes\xi) \big| U(\Lambda(d)\otimes\eta) \big)
&= (\psi\otimes\omega_{\xi,\eta})\alpha(b^*d)
= \psi(b^*d) (\xi|\eta)
= \big( \Lambda(b)\otimes\xi \big| \Lambda(d)\otimes\eta \big),
\end{align*}
using that $\alpha$ is $\psi$-invariant.  A similar calculation holds with arbitrary tensors, which shows that
$U$ is well-defined, and a co-isometry.  The density condition
on $\alpha$ (and that $A$ acts non-degenerately on $K$) shows that $U$ is unitary.  Then
\[ U(b\otimes 1)U^* U(\Lambda(d)\otimes\xi)
= U(\Lambda(bd)\otimes\xi) = \alpha(bd)(\Lambda(1)\otimes\xi)
= \alpha(b) U(\Lambda(d)\otimes\xi), \]
and thus $U(b\otimes 1)U^* = \alpha(b)$, acting on $L^2(B)\otimes K$.

We remark that we could instead have worked in the language of Hilbert $C^*$-modules, which would remove the
need to consider $K$, at the cost of more technicalities.  The following now shows how $\alpha$ commuting with
$A_G$ makes sense at the level of $U$.

\begin{lemma}\label{lem:coact_on_qg_iff_U_commutes}
Let $A_0:B\rightarrow B$ be a linear map which induces the linear map $A_G:L^2(B)\rightarrow L^2(B)$,
that is, $A_G\circ\Lambda = \Lambda\circ A_0$.
Then $\alpha A_0 = (A_0\otimes\id)\alpha$ if and only if $(A_G\otimes 1)U = U(A_G\otimes 1)$.
\end{lemma}
\begin{proof}
Given $b\in B, a\in A, \xi\in K$, we see that $(b\otimes a)(\Lambda(1)\otimes\xi) = \Lambda(b)\otimes a\xi$.
As $A$ acts non-degenerately on $K$, it follows that $u\in B\otimes A$ has $u=0$ if and only if
$u(\Lambda(1)\otimes\xi)=0$ for all $\xi\in K$.  We hence see that
$\alpha A_0 = (A_0\otimes\id)\alpha$ if and only if
\begin{align*}
&  \alpha(A_0(b)) = (A_0\otimes\id)\alpha(b)   & (b\in B) \\
\Leftrightarrow\qquad &  \alpha(A_0(b))(\Lambda(1)\otimes\xi)
  = (A_0\otimes\id)\alpha(b)(\Lambda(1)\otimes\xi)
  & (b\in B, \xi\in K) \\
\Leftrightarrow\qquad &  U(\Lambda(A_0(b))\otimes\xi)
  = (A_G\otimes 1) U(\Lambda(b)\otimes\xi)
  & (b\in B, \xi\in K) \\
\Leftrightarrow\qquad &  U\big((A_G\otimes 1) (\Lambda(b)\otimes\xi)\big)
  = (A_G\otimes 1) U(\Lambda(b)\otimes\xi)
  & (b\in B, \xi\in K)
\end{align*}
using that $A_G\Lambda(d) = \Lambda(A_0(d))$ for $d\in B$.  This is equivalent to $U(A_G\otimes 1)
= (A_G\otimes 1)U$, as claimed.
\end{proof}

We shall call $U$ the \emph{unitary implementation of $\alpha$}, compare with the much more general
construction of \cite{v1}.
Recall from Proposition~\ref{prop:coacts_to_coreps} that there is a bijection between coactions $\alpha$
and certain (unitary) matrices $v\in M_n(A)$.  The following clarifies the relation between $v$ and the
unitary $U$.

\begin{lemma}
The unitary $U\in\mc B(L^2(B)\otimes K)$ is equal to the image of $v\in M_n(A) \cong M_n
\otimes A$ acting on $L^2(B)\otimes K$.
\end{lemma}
\begin{proof}
We simply calculate that for $a\in A$,
\[ U\big( \Lambda(e_i) \otimes \xi \big)
= \alpha(e_i)(\Lambda(1)\otimes\xi)
= \sum_j \Lambda(e_j) \otimes v_{ji}\xi. \]
Thus the matrix of $U$ equals $(v_{ij})$, acting on $L^2(B)\otimes K$, as claimed.
\end{proof}

In particular, $U$ is a corepresentation (which also follows by direct calculation).

The following definition should now be clear, given our motivation above from the case of classical
groups, and Lemma~\ref{lem:coact_on_qg_iff_U_commutes}.

\begin{definition}\label{defn:coaction_on_quantum_graph}
Let $(A,\Delta)$ be a compact quantum group, and let $\alpha:B\rightarrow B\otimes A$ be a coaction which
preserves $\psi$.  Let $A_G\in\mc B(L^2(B))$ be a quantum adjacency matrix associated to the linear map $A_0$
on $B$.  Then $\alpha$ is a \emph{coaction on $A_G$} when $\alpha A_0 = (A_0\otimes\id)\alpha$, equivalently,
when $A_G$ and $v$ commute, where $v$ is the unitary corepresentation given by $\alpha$.
\end{definition}

\begin{remark}\label{rem:coact_empty_compl}
Let us consider the complete quantum graph, compare Proposition~\ref{prop:equiv_complete_qg}, so that
$A_G = \theta_{\Lambda(1),\Lambda(1)}$.  Then $A_G = \eta\circ\eta^*$, and so by 
Remark~\ref{rem:commute_eta_etastar}, as $\alpha$ is $\psi$-preserving, automatically $v$ commutes with
$A_G$.

Similarly, consider the empty quantum graph, compare Proposition~\ref{prop:equiv_empty_qg}, so that
$A_G = (mm^*)^{-1}$.  As $\alpha$ is a coaction, by Proposition~\ref{prop:coacts_to_coreps},
$(m\otimes 1) v_{13} v_{23} = v(m\otimes 1)$.  Thus also $v_{23}^* v_{13}^* (m^*\otimes 1)
= (m^*\otimes 1)v^*$, and so as $v$ is unitary, also $(m^*\otimes 1)v = v_{13} v_{23}(m^*\otimes 1)$.
Hence
\[ v(mm^*\otimes 1) = (m\otimes 1) v_{13} v_{23} (m^*\otimes 1)
= (mm^*\otimes 1) v, \]
and so also $A_G$ commutes with $v$.

Hence, in agreement with intuition, as soon as $\alpha$ is a $\psi$-preserving coaction on $B$, it is
automatically a coaction on both the empty and complete quantum graphs.
\end{remark}

We next show that there exists a universal compact quantum group which acts on $A_G$, and which we can then be
justified in calling the \emph{quantum automorphism group of $A_G$}.  We shall compare this definition to
others in the literature in Remark~\ref{rem:other_defns}.

\begin{proposition}\label{prop:construct_qaut_ag}
Consider $\qaut(B,\psi)$ as constructed above from a generating unitary corepresentation $v$.
Let $A$ be the quotient formed by imposing the extra conditions that $A_G v = v A_G$.  Then $\Delta$ drops
to $A$ making $(A,\Delta)$ into a quantum subgroup of $\qaut(B,\psi)$, which is the quantum automorphism
group of $A_G$: that is, universal amongst all compact quantum groups acting on $A_G$.
\end{proposition}
\begin{proof}
Let $J$ be the ideal generated by the relations coming from $A_G v = v A_G$, that is, $J$ is the ideal
generated by the element
\[ \sum_{k=1}^n A^G_{ik} v_{kj} - \sum_{k=1}^n A^G_{kj} v_{ik}
\qquad (1\leq i,j\leq n). \]
If we apply $\Delta$ to one of these elements, we obtain
\begin{align*}
& \sum_{k,l=1}^n A^G_{ik} v_{kl} \otimes v_{lj} - \sum_{k,l=1}^n A^G_{kj} v_{il} \otimes v_{lk} \\
&= \sum_{l=1}^n \Big( \sum_{k=1}^n A^G_{ik} v_{kl} - v_{ik} A^G_{kl} \Big) \otimes v_{lj}
+ \sum_{l=1}^n \Big( \sum_{k=1}^n v_{ik} A^G_{kl} \Big) \otimes v_{lj}
- \sum_{k,l=1}^n A^G_{kj} v_{il} \otimes v_{lk} \\
&= \sum_{l=1}^n \Big( \sum_{k=1}^n A^G_{ik} v_{kl} - v_{ik} A^G_{kl} \Big) \otimes v_{lj}
+ \sum_{k,l=1}^n v_{ik} \otimes A^G_{kl} v_{lj} - v_{ik} \otimes A^G_{lj} v_{kl} \\
&= \sum_{l=1}^n \Big( \sum_{k=1}^n A^G_{ik} v_{kl} - v_{ik} A^G_{kl} \Big) \otimes v_{lj}
+ \sum_{k,l=1}^n v_{ik} \otimes \Big( A^G_{kl} v_{lj} - A^G_{lj} v_{kl}\Big) \\
&\in J\otimes \qaut(B,\psi) + \qaut(B,\psi)\otimes J.
\end{align*}
In the first line we add in and subtract the same element, but notice that the element in the brackets
is manifestly in $J$.  In the second step we re-index, and the re-arrange in the third step so
both brackets are now clearly members of $J$.

This calculation and a simple continuity argument shows that $\Delta$ gives a well-defined map
$A\rightarrow A\otimes A$, where $A = \qaut(B,\psi) / J$.  By definition, $A$ coacts on $A_G$.

Now let $A_1$ be some other compact quantum group which coacts on $A_G$.  Then, by
Proposition~\ref{prop:coacts_to_coreps} for any orthonormal basis of $L^2(B)$ there is a unitary corepresentation
$v \in M_n(A_1)$ which defines the coaction.  If necessary replacing $A_1$ by the $C^*$-subalgebra generated by
the elements of $v$, we may suppose that $A_1$ is a compact matrix quantum group generated by $v$.

Consider now $\qaut(B,\psi)$ with generating corepresentation $w$ say.  As $A_1$ coacts in a $\psi$-preserving way,
$v$ satisfies all the relations which are satisfied by $w$, and so by universality, there is a $*$-homomorphism
$\theta:\qaut(B,\psi)\rightarrow A_1$ which intertwines the coproducts and which sends $w$ to $v$.  Hence $\theta$
is a surjection.  As $A_G v = vA_G$ it follows that $\theta(J)=\{0\}$, and so $\theta$ drops to a morphism from
$A$, which is necessarily uniquely defined.  This has hence demonstrated the required universal property for $A$.
\end{proof}

We shall write $\qaut(A_G)$ for the resulting compact quantum group.

\begin{remark}\label{rem:other_defns}
Let us compare this definition/construction with \cite[Definition~3.7]{bcehpsw}.  There, $\mc O(A_G)$ is
defined to be the universal $*$-algebra generated by a unitary matrix $u$ with the conditions that the map
$\rho:B\rightarrow B\otimes \mc O(A_G), e_i\mapsto \sum_j e_j\otimes u_{ji}$ is a unital $*$-homomorphism
with $\rho A_0 = (A_0\otimes\id)\rho$.  Here $A_0:B\rightarrow B$ is associated to $A_G$ in the usual way.
Furthermore, \cite[Proposition~3.8]{bcehpsw} defines a Hopf $*$-algebra structure on $\mc O(A_G)$ which in
particular ensures that $\rho$ satisfies the coaction equation.

From the discussion above where we sketch the construction of $\qaut(B,\psi)$, it is clear that $\mc O(A_G)$
is a quotient of the Hopf $*$-algebra $\mc O(\qaut(B,\psi))$, where the extra relations come from requiring that
$\rho A_0 = (A_0\otimes\id)\rho$.  As in Lemma~\ref{lem:coact_on_qg_iff_U_commutes}, this is equivalent to
requiring that $u$ and $A_G$ commute.  Thus this construction does agree with
Proposition~\ref{prop:construct_qaut_ag}.

Gromada in \cite[Section~4]{grom} makes much the same definition, but with a slightly different
presentation of the relations.  The links with Tannaka--Krein reconstruction results are made explicit.
However, note that all definitions are made at the level of Hopf $*$-algebras, and so, for example, the
definition of a coaction \cite[Definition~4.2]{grom} differs from ours (but is equivalent).
\end{remark}

\begin{remark}\label{rem:no_splitting_delta_form}
As alluded to a number of times, much of the existing work on quantum graphs restricts to the case when $\psi$
is a $\delta$-form (tracial, or not).  There is some motivation for this if we consider $\qaut(B,\psi)$, as 
a result attributed to Brannan, \cite[Proposition~20]{dCFY}, describes $\qaut(B,\psi)$ in terms of a
free product of a decomposition of $B$ into the coarsest direct sum $B = \bigoplus B_i$ where $\psi$ restricts
to a $\delta$-form on each $B_i$.  Thus for many questions, it suffices to look at each factor $B_i$
separately.

We might wonder if $\qaut(A_G)$ factors similarly.  This does not appear to be the case, as there appears to
be no reason why $A_G$ should restrict to each factor $B_i$.  To find a counter-example, 
let $B=\mathbb M_2 \oplus\mathbb M_2$ where we equip each $\mathbb M_2$
factor with differently normalised traces, say $\psi_1,\psi_2$ with $\psi = \psi_1\oplus\psi_2$.
Then $\psi$ is not a $\delta$-form, but each $\psi_i$ is.

Let $B$ act on $H = \mathbb C^2 \oplus \mathbb C^2$ so that $B' \cong \ell^\infty_2$, say the span of
$e_1,e_2$ where $e_i$ is the projection onto the $i$th $\mathbb C^2$ component.  As usual, we identify
$\mc B(H)$ with $\mathbb M_2(\mc B(\mathbb C^2)) \cong \mathbb M_2(\mathbb M_2)$.  Let
\[ S = \lin \{ e_1, e_2, x_0, x_0^* \}, \quad\text{with}\quad
x_0 = \begin{pmatrix} 0 & t \\ 0 & 0 \end{pmatrix}, \]
for some non-zero $t:\mathbb C^2 \rightarrow \mathbb C^2$.  Then $S$ is an operator system, and so as $B$
has a tracial state, $S$ corresponds to a quantum adjacency matrix $A_G$ say, and to a projection in
$B\otimes B^\op$, say $e_G$.

Towards a contradiction, suppose that $A_G$ commutes with the projections onto each factor of $B$.  As
$A_G$ is self-adjoint, it follows that $A_G$ can be written as a sum of rank-one operators of the form
$\theta_{\Lambda(a), \Lambda(b)}$ where $a,b$ both come from the same factor of $B$.  It follows that
$e_G$ is the span of tensors $b\otimes a^*$ where $a,b$ both come from the same factor of $B$.  Thus
the projection of $\mc B(L^2(B))$ onto $S$ is a sum of maps of the form $\theta\mapsto b\theta a^*$
where $a,b$ both come from the same factor of $B$.  However,
\[ \begin{pmatrix} b & 0 \\ 0 & 0 \end{pmatrix} x_0 \begin{pmatrix} a^* & 0 \\ 0 & 0 \end{pmatrix}
= \begin{pmatrix} 0 & bt \\ 0 & 0 \end{pmatrix} \begin{pmatrix} a^* & 0 \\ 0 & 0 \end{pmatrix} = 0, \]
and similarly if $b,a$ come from the 2nd factor.  This is a contradiction, as the projection onto
$S$ should preserve $x_0$.
\end{remark}

Examples of quantum automorphism groups of quantum graphs are considered in \cite{grom}, and in
particular in \cite[Section~4]{matsuda}, where all quantum graphs over $B = \mathbb M_2$ are considered.

\subsection{Quantum automorphisms of operator bimodules}\label{sec:qaut_op_bimod}

As $B'$-operator bimodules $S$ satisfying the conditions of Theorem~\ref{thm:qu_ad_mats_axioms_12} biject
with quantum adjacency matrices $A_G$, we immediately obtain a notion of an action of a quantum group on $S$.
Of course, it would be more satisfying to have a definition which only used $S$.

If we consider what happens for (non-quantum) automorphisms of $A_G$, say associated to a projection
$e_G$ and $B'$-bimodule $S$, as in Section~\ref{sec:auts_quan_graphs}, then we found a link between
automorphisms $\theta$ of $A_G$, the condition that $(\theta\otimes\theta)(e_G)=e_G$, and $\theta_{\mc B}(S)=S$,
at least when $S\subseteq\mc B(L^2(B))$.  Unfortunately, as is well-known, for quantum groups, forming
the analogue of the ``diagonal action'' $\theta\otimes\theta$ is essentially impossible.  We hence focus on
quantum automorphisms of $S$, but guided by the commutative situation, we shall only consider the situation
when $B$ is acting on $L^2(B)$.

\begin{remark}
A possible definition of $(A,\Delta)$ acting on $S$ was given in \cite[Definition~5.11]{eifler}, but
unfortunately this does not work in general.  However, it will motivate our definition below, so we sketch
the idea, and the problem.  Let $(A,\Delta)$ coact on $B$ is a $\psi$-preserving way.  As in 
Section~\ref{sec:unitary_impl} form the unitary implementation $U$, which is a unitary corepresentation of
$(A,\Delta)$.  We may then define 
\[ \alpha_U: \mc B(L^2(B)) \rightarrow \mc B(L^2(B))\otimes A;
\quad x\mapsto U(x\otimes 1)U^*. \]
As $(\id\otimes\Delta)(U) = U_{12} U_{13}$, we see that
\[ (\id\otimes\Delta)(\alpha_U(x)) = U_{12} U_{13} (x\otimes 1\otimes 1) U_{13}^* U_{12}^*
= (\alpha_U\otimes\id)\alpha_U(x), \]
so $\alpha_U$ satisfies the coaction equation.  As in fact $U\in M_n(\mc O_A)$ (see the discussion after
Definition~\ref{defn:coaction}) it makes sense to apply the counit $\epsilon$ of $\mc O_A$, and we see that
\[ (\id\otimes\epsilon)\alpha_U(x) = (\id\otimes\epsilon)(U) x (\id\otimes\epsilon)(U^*)
= x \qquad (x\in\mc B(L^2(B))). \]
From this, the Podle\'s density condition follows, compare \cite[Remark~2.3]{soltan2}.  So $\alpha_U$ is
a coaction.

It is suggested in \cite[Definition~5.11]{eifler} that $(A,\Delta)$ acts on $S$ via $\alpha$ when
$\alpha_U(S) \subseteq S\otimes A$.  In fact, \cite{eifler} works with the more general notion of
\emph{quantum metrics}, \cite{kw}, and so perhaps also in the graph case we require $\alpha_U(B')
\subseteq B'\otimes A$ (this condition corresponding to the quantum metric at $0$).
It follows from \cite[Proposition~5.14]{eifler} that when $B=C(V)$
and $G$ is a classical graph, then $\alpha$ acts on $S$ if and only if $\alpha$ acts on $A_G$, in the
sense of Definition~\ref{defn:coaction_on_quantum_graph}.

However, suppose now $B$ is non-commutative, and consider the empty quantum graph, compare
Proposition~\ref{prop:equiv_empty_qg}.  Suppose also that $\psi$ is a trace.  Then $A_G = (mm^*)^{-1}$
and $S = B'$.  As we saw in Remark~\ref{rem:coact_empty_compl}, any coaction on $(B,\psi)$ should automatically
be a coaction on $S$.  That is, it should be automatic that $\alpha_U(B') \subseteq B'\otimes A$.
However, there doesn't seem to be any reason why this should be so.  (This property is claimed in
\cite[Remark~5.13]{eifler}, but we do not follow the claimed proof.  Indeed, there it seems to be claimed that
$J e_k J = e_k^*$ where $(e_k)$ is a basis of $B$ with $(\Lambda(e_k))$ an orthonormal basis of $L^2(B)$,
and $J$ the modular conjugation.  This relation, $Je_kJ=e_k^*$ would be true if $e_k$ were central, that is,
$B$ were commutative, but it does not hold in general.)

We shall tweak this proposed definition to more directly take account of $B'$.
\end{remark}

Recall the modular conjugation $J$ from Section~\ref{sec:equiv}, which in particular satisfies that
$JBJ=B'$.  Define $\kappa:\mc B(L^2(B))\rightarrow \mc B(L^2(B))$ by $\kappa(x) = Jx^*J$, so $\kappa$ is
an anti-$*$-homomorphism, $\kappa^2=\id$, and $\kappa(B) = B'$.  We recall that associated to any
compact quantum group $(A,\Delta)$ is both the \emph{reduced} form, acting on the GNS space for the Haar
state of $A$, and the \emph{universal} form, compare \cite{kus1}, which is the enveloping $C^*$-algebra of
$\mc O_A$.  In particular $\qaut(B,\psi)$ must be universal in this sense.  We shall now suppose that
$(A,\Delta)$ is universal, or more precisely, that $A$ admits a \emph{unitary antipode} $R$,
see \cite[Proposition~7.2]{kus1}.  Then $R$ is an anti-$*$-homomorphism, $R^2=\id$ and $\sigma(R\otimes R)
\Delta = \Delta R$, with $\sigma$ the tensor swap map.

\begin{definition}
Given a unitary corepresentation $U$ of $(A,\Delta)$ on $L^2(B)$, define
\[ \alpha_U:\mc B(L^2(B)) \rightarrow \mc B(L^2(B))\otimes A;
\quad x\mapsto U(x\otimes 1)U^*, \]
and define
\[ \alpha^\op_U:\mc B(L^2(B)) \rightarrow \mc B(L^2(B))\otimes A;
\quad x\mapsto (\kappa\otimes R)\alpha_U(\kappa(x)). \]
\end{definition}

\begin{lemma}
With the definition above, $\alpha^\op_U$ is a coaction of $(A, \sigma\Delta)$ on $\mc B(L^2(B))$
and $\alpha^\op_U(B') \subseteq B'\otimes A$.
\end{lemma}
\begin{proof}
It is standard that $(A, \sigma\Delta)$ is also a compact quantum group.  As $\kappa$ and $\kappa\otimes R$
are both anti-$*$-homomorphisms, it follows that $\alpha_U$ is a $*$-homomorphism.  Notice that $R$ intertwines
the coproducts $\Delta$ and $\sigma\Delta$.  Thus
\begin{align*}
(\alpha_U^\op\otimes\id)\alpha_U^\op(x)
&= (\alpha_U^\op\kappa\otimes R)\alpha_U(\kappa(x))
= (\kappa\otimes R\otimes R)(\alpha_U\otimes\id)\alpha_U(\kappa(x)) \\
&= (\kappa\otimes R\otimes R)(\id\otimes\Delta)\alpha_U(\kappa(x))
= (\kappa\otimes\sigma\Delta R)\alpha_U(\kappa(x))
= (\id\otimes\sigma\Delta)\alpha_U^\op(x),
\end{align*}
and so $\alpha_U^\op$ satisfies the coaction equation.  As $R^2=\id$, verifying the Podle\'s density
condition is routine.

For $x\in B'$, we know that $b=\kappa(x)\in B$, and so $\alpha_U^\op(x) = (\kappa\otimes R)\alpha_U(b)
\in \kappa(B)\otimes R(A) = B'\otimes A$, as claimed.
\end{proof}

It is easy to see that $\alpha_U^\op(S) \subseteq S\otimes A$ if and only if $\alpha_U(\kappa(S))\subseteq
\kappa(S)\otimes A$.  In the tracial situation, it will turn out that $\alpha$ acts on $A_G$, that is,
$U$ commutes with $A_G$, if and only if $\alpha_U^\op(S) \subseteq S\otimes A$, equivalently, 
$\alpha_U(\kappa(S))\subseteq \kappa(S)\otimes A$, see Corollary~\ref{corr:tracial_bimod_action}.
When $\psi$ is not a trace, things are not quite so precise, Theorem~\ref{thm:bimod_quan_action}.

We start towards this proof by looking more closely at $\alpha_U$.  We would like to find the implementing
unitary for this coaction, which first requires us to find a natural invariant state.
Set $H = L^2(B)$, and fix some basis $(e_i)$ for $B$ such that $(\Lambda(e_i))$ is an orthonormal basis of
$H$.  To ease notation, we shall identify $\mathbb M_n$ with $\mathbb M_n\otimes 1\subseteq M_n(A)$, thus
allowing us, for example, to write $XU$ for a scalar matrix $X$.

\begin{lemma}\label{lem:inv_state_bh}
Recall the matrix $X_{ij} = (\Lambda(e_i)|\Lambda(e_j^*))$ from Lemma~\ref{lem:X_matrix}, and set $P=X^*X$.
Then $P$ is positive and invertible, and $x_0 = (P^t)^{1/2}$ agrees with $\nabla^{1/2}$ on $H=L^2(B)$.
Define $\tilde\psi(x) = \Tr(P^tx)$ for $x\in\mc B(H)$. Then the coaction
$\alpha_U$ preserves $\tilde\psi$.
\end{lemma}
\begin{proof}
From Lemma~\ref{lem:X_matrix}, we know that $X^{-1}=\overline X$, and so $P$ is positive and invertible, with
$P^{-1} = \overline X X^t$.  By definition, $x_0^{-2} = (P^t)^{-1} = X X^*$, and we calculate that
\begin{align*}
(XX^*)_{ij}
&= \sum_k X_{ik} \overline{ X_{jk} }
= \sum_k \psi(e_i^* e_k^*) \psi(e_k e_j)
= \sum_k \Tr(Qe_i^* e_k^*) \Tr(Qe_k e_j) \\
&= \sum_k \Tr(Q e_k^* Qe_i^*Q^{-1}) \Tr(QQ^{-1} e_j Q e_k)
= \sum_k (\Lambda(e_k)|\Lambda(Qe_i^*Q^{-1})) (\Lambda(Qe_j^*Q^{-1})|\Lambda(e_k)) \\
&= (\Lambda(Qe_j^*Q^{-1})|\Lambda(Qe_i^*Q^{-1})
= \Tr(QQ^{-1}e_j Q Q e_i^*Q^{-1})
= \Tr(Q e_i^* Q^{-1} e_j Q) \\
&= (\Lambda(e_i)|\Lambda(Q^{-1}e_jQ))
= (\Lambda(e_i)|\nabla^{-1}\Lambda(e_j)),
\end{align*}
and so $x_0^{-2}=XX^*$ represents the operator $\nabla^{-1}$, so $x_0$ agrees with $\nabla^{1/2}$.

As $U = X \overline U \overline X$ so that $U^t = X^* U^* X^t$, we see that
\[ U^t P \overline U = X^* U^* X^t X^* X \overline X U X = X^* U^* (\overline X X)^t (X \overline X) U X
= X^* U^* U X = X^*X = P. \]
It follows that for $x\in\mc B(H)$,
\begin{align*}
(\tilde\psi\otimes\id)\alpha_U(x)
&= \sum_{i,j,k,l} \tilde\psi( e_{ij} x e_{kl} ) U_{ij} (U^*)_{kl}
= \sum_{i,j,k,l} (\Lambda(e_j)|x\Lambda(e_k)) \Tr(P^t e_{il}) U_{ij} U^*_{lk} \\
&= \sum_{i,j,k,l} x_{jk} P_{il} (U^t)_{ji} (\overline U)_{lk}
= \sum_{j,k} x_{jk} (U^t P \overline U)_{jk}
= \sum_{j,k} x_{jk} P_{jk}
= \Tr(P^t x) = \tilde\psi(x),
\end{align*}
and so $\alpha_U$ preserves $\tilde\psi$.
\end{proof}

Form now $L^2(\mc B(H), \tilde\psi)$ with GNS map $\tilde\Lambda$, so
\[ (\tilde\Lambda(\theta_{\xi,\eta})|\tilde\Lambda(\theta_{\xi',\eta'}))
= \tilde\psi(\theta_{\eta,\xi} \theta_{\xi',\eta'} )
= (\eta|\eta') \Tr(P^t \theta_{\xi',\xi})
= (\eta|\eta') (x_0(\xi')|x_0(\xi))
= (\eta|\eta') (\overline{x_0(\xi)}|\overline{x_0(\xi')})_{\overline H}. \]
Thus the map
\[ w: L^2(\mc B(H), \tilde\psi) \rightarrow H\otimes\overline H; \quad \tilde\Lambda(\theta_{\xi,\eta})
\mapsto \eta \otimes \overline{x_0\xi} \]
extends by linearity to a unitary.  For $x\in\mc B(H)$ we have that
$x \Lambda(\theta_{\xi,\eta}) = \Lambda(\theta_{\xi,x\eta})$ and so $w x w^* = x\otimes 1$.  Hence
$L^2(\mc B(H),\tilde\psi)$ is unitarily equivalent to $H\otimes\overline H$ with the usual action of $\mc B(H)$.

Recall the map $\mc B(H)\rightarrow\mc B(\overline H); x\mapsto x^\top$ which is an anti-$*$-homomorphism.
Similarly, so is the unitary antipode $R:A\rightarrow A$.  Following common usage in quantum group theory,
for $Y\in \mc B(H)\otimes A$ we write $Y^{\top\otimes R}$ for the image of $U$ under the tensor product of
$x\mapsto x^\top$ and $R$.  One can check that $U^{\top\otimes R}$ is also a unitary corepresentation,
termed the \emph{contragradient corepresentation}, compare \cite[Section~3.3.3]{sw}.  With respect to the
orthonormal basis $(\overline{\Lambda(e_i)})$ of $\overline H$, we have that $x\mapsto x^\top$ is the
transpose map on matrices, and so $U^{\top\otimes R}$ has matrix $(R(U_{ji}))_{i,j}$.

\begin{lemma}\label{lem:what_is_tildeU}
Let $\tilde U$ be the unitary implementation of $\alpha_U$ with respect to $\tilde\psi$.
Then $(w\otimes 1)\tilde U(w^*\otimes 1) = U_{13} U_{23}^{\top\otimes R}$.
\end{lemma}
\begin{proof}
As in Section~\ref{sec:unitary_impl}, let $A\subseteq\mc B(K)$, and then for $1\leq s,t\leq n$ and $\xi\in K$,
\begin{align*}
(w\otimes 1)\tilde U(w^*\otimes 1) & (\Lambda(e_s)\otimes\overline{\Lambda(e_t)}\otimes\xi) 
= (w\otimes 1)\tilde U(\tilde\Lambda(\theta_{x_0^{-1}\Lambda(e_t),\Lambda(e_s)})\otimes\xi) \\
&= (w\otimes 1)\alpha_U(\theta_{x_0^{-1}\Lambda(e_t),\Lambda(e_s)})(\tilde\Lambda(1)\otimes\xi) \\
&= (w\otimes 1)(U(\theta_{x_0^{-1}\Lambda(e_t),\Lambda(e_s)}\otimes 1)U^*)(\tilde\Lambda(1)\otimes\xi) \\
&= \sum_{i,j,k,l} w\tilde\Lambda(e_{ij} \theta_{x_0^{-1}\Lambda(e_t),\Lambda(e_s)} e_{kl}) \otimes
   U_{ij} U^*_{lk} \xi \\
&= \sum_{i,k,l} (x_0^{-1}\Lambda(e_t)|\Lambda(e_k)) \Lambda(e_i) \otimes \overline{x_0\Lambda(e_l)}
  \otimes U_{is} U^*_{lk} \xi \\
&= \sum_{i,j,k,l} \Lambda(e_i) \otimes \overline{\Lambda(e_j)}
  \otimes (x_0^{-1})_{tk} (x_0)_{lj} U_{is} U^*_{lk} \xi \\
&= \sum_{i,j} \Lambda(e_i) \otimes \overline{\Lambda(e_j)}
  \otimes U_{is} (x_0^{-1} U^* x_0)_{tj} \xi.
\end{align*}
Compare this to
\begin{align*}
U_{13} U_{23}^{\top\otimes R} (\Lambda(e_s)\otimes\overline{\Lambda(e_t)}\otimes\xi)
&= \sum_{i,j,k,l} e_{ik}\Lambda(e_s) \otimes e_{jl} \overline{\Lambda(e_t)} \otimes U_{ik} R(U_{lj})\xi \\
&= \sum_{i,j} \Lambda(e_i) \otimes \overline{\Lambda(e_j)} \otimes U_{is} R(U_{tj})\xi.
\end{align*}
We hence need to show that $(x_0^{-1} U^* x_0)_{tj} = R(U_{tj})$ for each $t,j$.

To show this, we use some compact quantum group theory.  Let $S$ be the antipode on $\mc O_A$, which has
decomposition $S = R\tau_{-i/2}$ where $(\tau_t)$ is the scaling group, and $\tau_{-i/2}$ is the analytic
extension.  Recall the definition of the
operator $\rho_U$ from \cite[Section~1.4]{nt}, and from before \cite[Definition~1.4.5]{nt} recall that
$\rho_U$ is the unique (up to a positive scalar) operator with $((\rho_U^t)^{1/2} \otimes 1) \overline{U}
((\rho_U^t)^{-1/2}\otimes 1)$ is unitary.  From Lemma~\ref{lem:X_matrix}, we know that $U = X\overline{U}X^{-1}$,
and so $\rho_U^{1/2}$ is $|X|^t$, up to a scalar.  As $P=X^*X$, we see that $|X|^t = (P^t)^{1/2} = x_0$.
From after \cite[Remark~1.7.7]{nt} we have that
$(\id\otimes\tau_{i/2})(U) = (\rho_U^{-1/2}\otimes 1)U(\rho_U^{1/2}\otimes 1)$ which thus equals
$x_0^{-1} U x_0$.  Hence
\[ R(U_{tj}) = S \tau_{i/2}(U_{tj}) = S( (x_0^{-1} U x_0)_{tj} ) = (x_0^{-1} U^* x_0)_{tj}, \]
as required.
\end{proof}

This previous result should be compared with \cite[Corollary~2.6.3]{vaesthesis} (which itself follows quickly
from \cite[Proposition~4.2]{v1}).  We do not invoke this directly, as we will need to use the exact form
of the unitary $w$, and it seems about as much work to check that $w$ really is the intertwining unitary
as it does to give the above lemma (which is also of independent interest).

Recall the functional $\psi^\op$ on $B^\op$ studied in Lemma~\ref{lem:Bop_GNS}, where we identify
$L^2(B^\op)$ with $\overline{L^2(B)}$.  Thus $L^2(B\otimes B^\op)$ is identified with $H\otimes\overline H$,
with GNS map $\Lambda\otimes\Lambda^\op$.
As usual, we use Theorem~\ref{thm:qu_ad_mats_axioms_12} to link $A_G$ to a projection
$e_G = \Psi'_{1/2,0}(A_G)$, so explicitly, we assume $A_G$ is self-adjoint satisfying axioms
(\ref{defn:quan_adj_mat:idem}) and (\ref{defn:quan_adj_mat:undir}) of Definition~\ref{defn:adj_op}.

\begin{lemma}
Let $A_G$ be a quantum adjacency matrix associated to
projection $e_G\in B\otimes B^\op$.
Then:
\begin{enumerate}
\item $w \tilde\Lambda(A_G) = (\Lambda\otimes\Lambda^\op)((\sigma_{-i}\otimes\id)(e_G))$;
\item $w\tilde\Lambda(aTb) = (a\otimes \sigma_{i/2}(b)^\top)w\tilde\Lambda(T)$ for $T\in\mc B(H), a,b\in B$;
\end{enumerate}
Let $S_0$ be the $B'$-operator bimodule generated by $A_G$, and set $S_0'=\kappa(S_0)$.
Then:
\begin{enumerate}[resume]
\item $S_0'$ is the $B$-operator bimodule generated by $A_G$;
\item the orthogonal projection $f$ of $L^2(\mc B(H),\tilde\psi)$ onto $\tilde\Lambda(S_0')$
  satisfies $w f w^* = (J\otimes\overline J)e_G(J\otimes\overline J)$;
\item $fw^*(\Lambda(1)\otimes\overline{\Lambda(1)}) = \tilde\Lambda(A_G)$.
\end{enumerate}
\end{lemma}
\begin{proof}
Let $T = \theta_{\Lambda(a),\Lambda(b)}$ for $a,b\in B$, so 
\begin{align*}
w \tilde\Lambda(T) = \Lambda(b) \otimes \overline{x_0 \Lambda(a)}
= \Lambda(b) \otimes \overline{\nabla^{1/2} \Lambda(a)}
= \Lambda(b) \otimes \overline{\Lambda(\sigma_{-i/2}(a))}
= (\Lambda\otimes\Lambda^\op)(b\otimes \sigma_{-i/2}(a)^*).
\end{align*}
As $\Psi'_{1/2,0}(T) = \sigma_{i/2}(b) \otimes a^*$, we see that
\begin{align*}
w \tilde\Lambda(A_G) 
&= (\Lambda\otimes\Lambda^\op)\big((\sigma_{-i/2}\otimes\sigma_{i/2})\Psi'_{1/2,0}(A_G)\big) \\
&= (\Lambda\otimes\Lambda^\op)\big((\sigma_{-i/2}\otimes\sigma_{i/2})(e_G)\big)
= (\Lambda\otimes\Lambda^\op)((\sigma_{-i}\otimes\id)(e_G)),
\end{align*}
as $e_G = (\sigma_{-i/2}\otimes\sigma_{-i/2})(e_G)$.

With $T = \theta_{\xi,\eta}$ and $a,b\in B$ we see that
\[ w\tilde\Lambda(aTb) = w\tilde\Lambda(\theta_{b^*\xi,a\eta}) = a\eta \otimes \overline{\nabla^{1/2}b^*\xi}
= a\eta \otimes \overline{ (\nabla^{-1/2} b \nabla^{1/2})^* \nabla^{1/2}\xi }
= (a\otimes \sigma_{i/2}(b)^\top) w\tilde\Lambda(T), \]
as claimed.

As $\kappa$ restricts to a bijection $B\rightarrow B'$, and $\kappa(A_G) = J A_G^* J = A_G$ because
$A_G$ is self-adjoint, and using Proposition~\ref{prop:A_comms_J_nabla}, it follows that
$S_0'$ is the $B$-operator bimodule generated by $A_G$.
It now follows that $w\tilde\Lambda(S_0') = (B\otimes B^\op)(\Lambda\otimes\Lambda^\op)
((\sigma_{-i}\otimes\id)(e_G))$.

Set $f_0 = (\sigma_{-i}\otimes\id)(e_G) \in B\otimes B^\op$, so that
\begin{align*}
(J\otimes\overline J)(\sigma_{-i/2} \otimes \sigma_{i/2})(f_0^*)(J\otimes\overline J)
&= (J\otimes\overline J)(\sigma_{-i/2} \otimes \sigma_{i/2})(\sigma_i\otimes\id)(e_G)(J\otimes\overline J) \\
&= (J\otimes\overline J)(\sigma_{i/2} \otimes \sigma_{i/2})(e_G)(J\otimes\overline J) \\
&= (J\otimes\overline J)e_G(J\otimes\overline J)
\in B' \otimes (B')^\op.
\end{align*}
However, for $u\in B\otimes B^\op$, with reference to Lemma~\ref{lem:Bop_GNS}, we see that
\[ (J\otimes\overline J)(\sigma_{-i/2} \otimes \sigma_{i/2})(f_0^*)(J\otimes\overline J)
   (\Lambda\otimes\Lambda^\op)(u)
= (\Lambda\otimes\Lambda^\op)(u f_0)
= u (\Lambda\otimes\Lambda^\op)(f_0). \]
Thus letting $u$ vary shows that the image of the projection $(J\otimes\overline J)e_G(J\otimes\overline J)$
is equal to $(B\otimes B^\op)(\Lambda\otimes\Lambda^\op)(f_0) = w\tilde\Lambda(S_0')$.
Thus necessarily $w f w^* = (J\otimes\overline J)e_G(J\otimes\overline J)$.
Hence
\begin{align*}
wfw^*(\Lambda(1)\otimes\overline{\Lambda(1)})
&= (J\otimes\overline J)e_G(J\otimes\overline J)(\Lambda(1)\otimes\overline{\Lambda(1)}) \\
&= (J\otimes\overline J) (\Lambda\otimes\Lambda^\op)(e_G)
= (\Lambda\otimes\Lambda^\op)((\sigma_{-i/2} \otimes \sigma_{i/2})(e_G)) \\
&= (\Lambda\otimes\Lambda^\op)((\sigma_{-i} \otimes \id)(e_G))
= w\tilde\Lambda(A_G).
\end{align*}
Thus $fw^*(\Lambda(1)\otimes\overline{\Lambda(1)}) = \tilde\Lambda(A_G)$.
\end{proof}

Again, in the following we assume that $A_G$ is self-adjoint satisfying axioms
(\ref{defn:quan_adj_mat:idem}) and (\ref{defn:quan_adj_mat:undir}) of Definition~\ref{defn:adj_op}.

\begin{theorem}\label{thm:bimod_quan_action}
Let $\alpha$ be a coaction of a compact quantum group $(A,\Delta)$ on $B$, which preserves $\psi$, and has unitary
implementation $U$.  Let $A_G$ be a quantum adjacency matrix, and let $S_0$ be the $B'$-operator bimodule
generated by $A_G$.  The following are equivalent:
\begin{enumerate}
\item\label{thm:bimod_quan_action:one}
  $\alpha$ acts on $A_G$, that is, $U A_G = A_G U$;
\item\label{thm:bimod_quan_action:two}
  $\alpha_U$ leaves $S_0' = \kappa(S_0)$ invariant, that is, $\alpha_U(S_0') \subseteq S_0'\otimes A$;
\item\label{thm:bimod_quan_action:three}
  $\alpha_U^\op(S_0) \subseteq S_0\otimes A$.
\end{enumerate}
\end{theorem}
\begin{proof}
We have already seen that (\ref{thm:bimod_quan_action:two}) and (\ref{thm:bimod_quan_action:three}) are
equivalent.

Suppose that $\alpha_U(S_0') \subseteq S_0'\otimes A$, and let $A\subseteq\mc B(K)$.
By definition of the unitary implementation $\tilde U$, given $s_0 \in S_0'$
and $\xi\in K$, we see that
\[ \tilde U(\tilde\Lambda(s_0')\otimes\xi) = \alpha_U(s_0')(\tilde\Lambda(1)\otimes\xi)
\in (S_0'\otimes A)(\tilde\Lambda(1)\otimes\xi)
= \tilde\Lambda(S_0') \otimes A\xi. \]
Let $f$ be the orthogonal projection of $L^2(\mc B(H),\tilde\psi)$ onto $\tilde\Lambda(S_0')$, so this calculation
shows that $\tilde U(f\otimes 1) = (f\otimes 1)\tilde U(f\otimes 1)$.  A non-trivial result about compact quantum
groups (see the proof of \cite[Proposition~6.2]{mv} for example) shows that then
$\tilde U(f\otimes 1) = (f\otimes 1)\tilde U$.

From the previous lemma, it hence follows that for $\xi\in K$,
\begin{align}
\alpha_U(A_G)(\tilde\Lambda(1)\otimes\xi)
&= \tilde U(\tilde\Lambda(A_G)\otimes\xi)
= \tilde U( fw^*(\Lambda(1)\otimes\overline{\Lambda(1)}) \otimes \xi ) \notag \\
&= (f\otimes 1)\tilde U( w^*(\Lambda(1)\otimes\overline{\Lambda(1)}) \otimes \xi ) \notag \\
&= (fw^*\otimes 1) U_{13} U^{\top\otimes R}_{23}(\Lambda(1)\otimes\overline{\Lambda(1)} \otimes \xi ),
\label{eq:one}
\end{align}
in the final step using Lemma~\ref{lem:what_is_tildeU}.

For $\xi\in K$, we have that $U(\Lambda(1)\otimes\xi) = \alpha(1)(\Lambda(1)\otimes\xi)
= \Lambda(1)\otimes\xi$, from the definition of $U$.  Then for $a\in B, \xi_1\in K$, also
\[ (U^*(\Lambda(1)\otimes\xi) | \Lambda(a)\otimes\xi_1 )
= ( \Lambda(1)\otimes \xi | \alpha(a)(\Lambda(1)\otimes\xi_1) )
= (\xi | (\psi\otimes\id)\alpha(a) \xi_1)
= \psi(a) (\xi|\xi_1), \]
as $\alpha$ is $\psi$-preserving.  Thus $U^*(\Lambda(1)\otimes\xi) = \Lambda(1)\otimes\xi$.  
Let $U = \sum_i x_i \otimes a_i \in \mc B(H)\otimes A$, so that $\sum_i \Lambda(x_i^*) \otimes a_i^*(\xi)
= U^*(\Lambda(1)\otimes\xi) = \Lambda(1)\otimes\xi$.  It is always possible to choose $A\subseteq\mc B(K)$
in such a way that there is an involution $j:K\rightarrow K$ with $R(a) = j a^* j$ for $a\in A$.
(For example, let $A\subseteq\mc B(K)$ be the universal representation.)  Then
\[ U^{\top\otimes R}(\overline{\Lambda(1)}\otimes\xi)
= \sum_i x_i^\top\overline{\Lambda(1)} \otimes R(a_i)(\xi)
= \sum_i \overline{ \Lambda(x_i^*) } \otimes j a_i^* j(\xi)
= \overline{\Lambda(1)} \otimes \xi. \]
It hence follows that, continuing the calculation from \eqref{eq:one},
\begin{align*}
\alpha_U(A_G)(\tilde\Lambda(1)\otimes\xi)
&= (fw^*\otimes 1)(\Lambda(1)\otimes\overline{\Lambda(1)} \otimes \xi )
= \tilde\Lambda(A_G) \otimes \xi
= (A_G\otimes 1)(\tilde\Lambda(1)\otimes\xi).
\end{align*}
As $\alpha_U(A_G) \in \mc B(H) \otimes A$, and $\tilde\Lambda:\mc B(H)\rightarrow L^2(\mc B(H))$ is injective,
and is the map $x\mapsto x\tilde\Lambda(1)$, the previous calculation shows that $\alpha_U(A_G)
= A_G\otimes 1$.  This is equivalent to $U(A_G\otimes 1)U^* = A_G\otimes 1$, that is, that $U$ and $A_G$
commute.  We have hence shown that (\ref{thm:bimod_quan_action:two}) implies (\ref{thm:bimod_quan_action:one}).

Conversely, if $A_G$ and $U$ commute, then $\alpha_U(A_G) = A_G\otimes 1$.  As $\alpha_U$ extends $\alpha$
in the sense that $\alpha_U(b) = \alpha(b)$ for $b\in B\subseteq\mc B(H)$, we see that
\[ \alpha_U(a A_G b) = \alpha(a) \alpha_U(A_G) \alpha(b) \in (B\otimes A)(A_G\otimes 1)(B\otimes A)
= S_0'\otimes A \qquad (a,b\in B), \]
and so $\alpha_U(S_0') \subseteq S_0'\otimes A$, showing that (\ref{thm:bimod_quan_action:one}) implies
(\ref{thm:bimod_quan_action:two}).
\end{proof}

Given the discussion before, we might have hoped instead to find that $UA_G=A_GU$ was equivalent to
$\alpha(\kappa(S)) \subseteq \kappa(S)\otimes A$, but it just seems that this is not so.  However,
working with $S_0$ and $\kappa(S_0)$ instead of $S$ and $\kappa(S)$ is not so unnatural in light of
Remark~\ref{rem:gen_by_A_itself}, because $S_0$ is exactly the ``tracial'' quantum graph which arises from
$S$ using Theorem~\ref{thm:reduce_tracial_case}.  The dependence on the state $\psi$ arises from the
fact that we always assume that $(A,\Delta)$ coacts on $B$ in a way which preserves $\psi$.

In the tracial case we have a more direct result.

\begin{corollary}\label{corr:tracial_bimod_action}
Let $\psi$ be a trace, and let $A_G$ correspond to the $B'$-operator bimodule $S$.  Then $\alpha$
acts on $A_G$ if and only if $\alpha_U^\op$ leaves $S$ invariant.
\end{corollary}

In these final two results, we seem to be working with $\alpha_U$ and $S_0'$ (or equivalently $\alpha_U^\op$ on $S$) where of course $S_0'$ is a $B$-bimodule, not a $B'$-bimodule.  However, this appears to simply be an artifact of having fixed our Hilbert space as $L^2(B)$, so that $B$ becomes anti-isomorphic to $B'$.

We leave for future work an investigation of whether there is a quantum group analogue of the ideas
discussed at the end of Section~\ref{sec:auts_quan_graphs}.  That is, if we have a coaction of $(A,\Delta)$
on $\mc B(H)$, and $B\subseteq\mc B(H)$ with a quantum graph $S\subseteq\mc B(H)$, is it possible to
place conditions on the coaction so as we obtain a coaction on $S$ (whatever this might mean, at this level
of generality).  For more on coactions on $\mc B(H)$ see, for example, \cite[Section~3]{dcmn}.

\bigskip

\noindent\emph{Author's Address:}

\noindent
Matthew Daws,
Department of Mathematics and Statistics,
Lancaster University,
Lancaster,
LA1 4YF,
United Kingdom

\noindent\emph{Email:} \texttt{matt.daws@cantab.net}

\end{document}